\documentclass[10pt]{article}

\usepackage[english]{babel}
\usepackage[utf8]{inputenc}
\usepackage{amsmath}
\usepackage{amssymb}
\usepackage{amsthm}
\usepackage{hyperref}
\usepackage{tikz}
\usepackage{todonotes}
\usepackage{slashed}
\usepackage{mathtools}
\usepackage{dsfont}
\usepackage{tensor}
\usepackage{verbatim}
\usepackage{graphicx}
\usepackage[shortlabels]{enumitem}
\usepackage{xstring}
\usepackage[left=3cm,right=3cm,top=3cm]{geometry}
\usepackage{cleveref}
\usepackage{tensor}
\usepackage{mathrsfs}

\newcounter{comments}
\newenvironment{displaycomment}{\begin{list}{}{\rightmargin=1cm\leftmargin=1cm}\item\sf\begin{small}}{\end{small}\end{list}}

\newcommand{\C}{\mathbb{C}}
\newcommand{\R}{\mathbb{R}}
\newcommand{\Z}{\mathbb{Z}}
\newcommand{\mb}[1]{\mathbb{#1}}
\newcommand{\mc}[1]{\mathcal{#1}}

\newcommand*{\defeq}{\mathrel{\vcenter{\baselineskip0.5ex \lineskiplimit0pt
                        \hbox{\scriptsize.}\hbox{\scriptsize.}}}%
        =}

\newcommand{\ph}{\varphi}

\newcommand{\eps}{\varepsilon}

\newcommand{\SO}{\operatorname{SO}}
\newcommand{\Spin}{\operatorname{Spin}}
\newcommand{\U}{\operatorname{U}}

\newcommand{\Imp}{\operatorname{Imp}}

\newcommand{\End}{\operatorname{End}}

\newcommand{\Cl}{\operatorname{Cl}}
\newcommand{\lie}[1]{\mathfrak{#1}}

\newcommand{\pr}{\operatorname{pr}}
\newcommand{\trace}{\operatorname{trace}}
\newcommand{\lact}{\triangleright}

\newcommand{\1}{\mathds{1}}
\newcommand{\Dom}{\operatorname{Dom}}

\renewcommand{\O}{\operatorname{O}}
\renewcommand{\d}{\operatorname{d}}

\def\ttimes#1#2{\hspace{-0.15em}\tensor[_{#1}]{\times}{_{#2}}}
\def\res{\mathrm{\it{res}}}

\def\quot#1{,,#1``}

\def\cstar{C$^{*}\!$}

\theoremstyle{definition}
\newtheorem{definition}{Definition}[section]
\newtheorem{remark}[definition]{Remark}
\newtheorem{example}[definition]{Example}

\theoremstyle{plain}
\newtheorem{theorem}[definition]{Theorem}
\newtheorem{proposition}[definition]{Proposition}
\newtheorem{lemma}[definition]{Lemma}

\setlist{topsep=-0.2em,itemsep=0em}

\crefname{lemma}{Lemma}{Lemmas}
\crefname{section}{Section}{Sections}
\crefname{theorem}{Theorem}{Theorems}
\crefname{example}{Example}{Examples}
\crefname{remark}{Remark}{Remarks}
\crefname{proposition}{Proposition}{Propositions}
\crefname{definition}{Definition}{Definitions}
\crefname{appendix}{Appendix}{Appendices}
\crefname{enumi}{\unskip}{\unskip}
\crefname{equation}{\unskip}{\unskip}

\title{Fusion of implementers for spinors on the circle}

\author{Peter Kristel and Konrad Waldorf}
\date{}

\begin{document}

\maketitle

\begin{abstract}
We consider the  space of odd spinors on the circle, and a decomposition into spinors supported on either the top or on the bottom half of the circle. If an operator preserves this decomposition, and acts on the bottom half in the same way as a second operator acts on the top half, then the fusion of both operators is a third operator acting on the top half like the first, and on the bottom half like the second. 
Fusion restricts to the Banach-Lie group of  restricted orthogonal operators, which  supports a  central extension of implementers on a Fock space. In this article, we construct a lift of  fusion  to this central extension. Our construction  uses Tomita-Takesaki theory for the Clifford-von Neumann algebras of the decomposed space of spinors. Our  motivation  is to obtain an operator-algebraic model for the basic central extension of the loop group of the spin group, on which the fusion of implementers induces a fusion product in the sense considered in the context of  transgression and string geometry. In upcoming work we will use this model  to construct a fusion product on a spinor bundle on the loop space of a string manifold, completing a construction proposed by Stolz and Teichner.    

\medskip

\end{abstract}

     \setlength{\parskip}{0ex}
     \begingroup
     \tableofcontents
     \endgroup
     \setlength{\parskip}{1.5ex}


\section{Introduction}

In a survey article \cite{stolz3} Stolz and Teichner outline the construction of a spinor bundle on the loop space $LM$ of a smooth manifold $M$, equipped with an action of a certain Clifford-von Neumann algebra, and equipped with a so-called \emph{fusion product}. 
The fusion product  establishes an isomorphism between the Connes fusion of the fibres of the spinor bundle over two loops with a common segment, with the fibre over the loop where the common segment is deleted. Stolz and Teichner's construction is supposed to work when $M$ is a string manifold, i.e., $\frac{1}{2}p_1(M)=0$, and is  supposed to involve the choice of a string structure.
The bigger context of their  work is to find an appropriate framework for the Dirac operator on $LM$ postulated by Witten \cite{witten2}. 
A suitable index theory of this operator is expected to solve a number of open problems. For instance, it might offer  an index theorem relating the index of the Dirac operator to the Witten genus, and, it  might come with a version of the Schr\"odinger-Lichnerowicz formula and imply a bound for the scalar curvature of $LM$ or the Ricci curvature of $M$.

This paper is the first of three papers whose goal is a complete and fully rigorous construction of a spinor bundle on $LM$, equipped with  a Clifford action and a fusion product. The starting point of our work is a new concept of a string structure on $M$, which was not available when \cite{stolz3} was written. This new concept is based on Killingback's spin structures on loop spaces \cite{killingback1} and additionally equipped with a version of a fusion product, see \cite{Waldorfb}. The concept is defined relative to an arbitrary model of the basic central extension
\begin{equation}
\label{eq:intro:ce}
\U(1) \to \widetilde{L\Spin}(d) \to L\Spin(d)
\end{equation}
of the loop group of $\Spin(d)$ as a  so-called \emph{fusion extension}, where $d$ is the dimension of $M$. A fusion extension of a loop group $LG$ is a central extension $\U(1) \to \widetilde{LG} \to LG$ equipped with a multiplicative fusion product, in, roughly-speaking, the same sense as described above. More precisely, consider three paths $\beta_1$, $\beta_2$, and $\beta_3$ in $G$, all sharing a common initial point and a common end point, as sketched below. 
\begin{figure}[h]
\begin{equation*}
        \begin{tikzpicture}[scale=1]
        \node at (-0.3,1) {\small$\beta_1$};
        \node at (0.7,1) {\small$\beta_2$};
        \node at (2.3,1) {\small$\beta_3$};
        \path[->,in=270,out=0] (1,0) edge  (2,1);
        \path[-,out=90,in=0] (2,1) edge  (1,2);
        \path[->,out=180,in=270] (1,0) edge  (0,1);
        \path[-,out=90,in=180] (0,1) edge  (1,2);
        \path[->] (1,0) edge  (1,1);
        \path[-] (1,1) edge  (1,2);
        \end{tikzpicture}
\end{equation*}
\end{figure}
Then, form loops $\beta_i\cup \beta_j \in LG$ by first following $\beta_i$ and then going back along the reverse of $\beta_j$. The rule $(\beta_1\cup\beta_2,\beta_2\cup\beta_3)\mapsto \beta_1\cup\beta_3$ implements the idea of deleting the common segment $\beta_2$ of the two loops $\beta_1\cup\beta_2$ and $\beta_2\cup\beta_3$, resulting in the loop $\beta_1\cup\beta_3$. 
Now, a  multiplicative fusion product is an associative product
\begin{equation*}
\widetilde{LG}_{\beta_1 \cup \beta_2}  \times \widetilde{LG}_{\beta_2\cup\beta_3} \to \widetilde{LG}_{\beta_1 \cup \beta_3}
\end{equation*}
which is multiplicative with respect to the group structure of $\widetilde{LG}$ (and satisfies a further regularity condition, see \cref{def:fusion product}). A general treatment of  fusion extensions may be found in \cite{Waldorfc}. Many different models for the central extension \cref{eq:intro:ce} are known; for instance, the Mickelsson model \cite{mickelsson1}, which is related to conformal field theory, or the transgression of the basic gerbe over $\Spin(d)$, which is related to higher-categorical geometry on the group $\Spin(d)$ \cite{waldorf5}.  A fusion product can be constructed explicitly in both of these models \cite{Waldorfc}, but we do not know of any construction on other models.  However, neither model comes equipped with a representation suited to our goal of constructing spinor bundles.

In the present paper, we construct a new, operator-algebraic model for the central extension \cref{eq:intro:ce} as a fusion extension, of which the main novelty is the fusion product.  Our model comes naturally equipped with a representation on a  Fock space $\mc{F}$, known as the free fermions.  This representation is the infinite-dimensional analog of the spin representation, and  will be used for the construction of the spinor bundle, in a rather straightforward way: the spin structure on $LM$ that underlies our concept of a string structure lifts the structure group of $LM$ from $L\Spin(d)$ to the basic central extension \cref{eq:intro:ce}; then we simply take the associated bundle with canonical fibre the Fock space $\mathcal{F}$. More interestingly, we will show in third second paper that the fusion product on our operator-algebraic model, together with the fusion product that is part of the string structure, combine into a fusion product on the spinor bundle, as anticipated by Stolz and Teichner in \cite{stolz3}.    

Our operator-algebraic model for the central extension \cref{eq:intro:ce} is obtained from a fairly well-known construction starting with the real Hilbert space $V$ of odd $d$-dimensional spinors on the circle, which possesses a canonical Lagrangian subspace consisting of spinors  that extend to anti-holomorphic functions on the disk; see  \cref{sec:oddspinors} for details. Associated to this structure is a Clifford algebra $\Cl(V)$ with a unitary representation, the Fock space $\mathcal{F}$. 
An orthogonal operator $g\in \O(V)$ induces a so-called Bogoliubov automorphism $\theta_g$ of $\Cl(V)$, and a unitary operator $U\in \U(\mathcal{F})$ is said to implement $g$ if $\theta_g(a) = UaU^{*}$ as elements of $\mc{B}(\mathcal{F})$. The group of all implementers forms a central extension
\begin{equation}
\label{intro:implementers}
\U(1) \to \Imp(V) \to \O_{\res}(V)
\end{equation}
of Banach-Lie groups, where $\O_{\res}(V)$ is the subgroup of $\O(V)$ consisting of implementable operators.   There is a Fr\'echet-Lie group  homomorphism $L\SO(d) \to \O_{\res}(V)$ producing operators that act pointwise on spinors. The pullback of the central extension \cref{intro:implementers} to $L\SO(d)$ and then further to $L\Spin(d)$ is our operator-algebraic model for the basic central extension \cref{eq:intro:ce}. The  central extension \cref{intro:implementers} has been studied extensively by Araki \cite{Ara85}, Neeb \cite{Ne02}, Ottesen \cite{Ott95}, and others, and  we use many of their results. Yet we found it necessary to clarify various aspects, in particular related to the Lie group structure  on the group of implementers. All this is comprised in  \cref{sec:implementers}, with some generalities  moved to an \cref{app:BanachCentral}. One result we derive and use later is that $\widetilde{L\Spin}(d)$ acts by \emph{even} operators with respect to the natural $\Z_2$-grading on $\mathcal{F}$; this appears as \cref{LSpinEven}.

The next part of the present paper is devoted to the construction of the fusion product, which requires a number of preliminary results described in \cref{sec:freefermions}. We study the splitting $V=V_{-} \oplus V_{+}$ into spinors supported on the bottom and on the top half of the circle, respectively. We construct a reflection operator that exchanges $V_{-}$ with $V_{+}$ and induces a Lie group homomorphism $\tau: \O_{\res}(V) \to \O_{\res}(V)$. 
Our first   result here is the construction of a  lift of $\tau$ to the central extension $\Imp(V)$; see \cref{thm:KappaTau,lem:KappaIsSmooth}. Our next results concern the Clifford algebras $\Cl(V_{\pm})$, their induced representations on $\mathcal{F}$, and their completions to von Neumann algebras. The main point here is the identification of the modular conjugation operator $J$ of  Tomita-Takesaki theory for $\Cl(V_{-})''$. It has  been computed  by Wassermann  \cite{wassermann98}, Henriques \cite{Hen14}, and Jannssens \cite{janssens13}, and we have reproduced that computation in \cref{sec:ModularConjugation}  for convenience. 
The computation relates $J$ directly to the reflection operator $\tau$, see \cref{thm:modular}. We use it in order to prove a result, \cref{lem:ImpCommute}, about the commutativity of two subgroups of $\Imp(V)$, namely the subgroups  $\Imp_{0}(V)_{-}$ and $\Imp_{0}(V)_{+}$ of even implementers of orthogonal transformations fixing $V_{+}$ and $V_{-}$ pointwise, respectively.

The construction of the fusion product is then described in \cref{sec:FusionOnLSpin}, starting with a quick introduction to the theory of fusion products on central extensions of loop groups. We introduce a new general method of constructing multiplicative fusion products, via a concept we call \emph{fusion factorization}, see \cref{def:fusfac} and \cref{thm:fusfromfusfac}.
This method  can be applied to a class of central extensions that we call \emph{admissible} (\cref{def:admissible}). Admissibility is related to the commutativity of two subgroups, which in case of the basic central extension \cref{eq:intro:ce} are essentially the subgroups    $\Imp_{0}(V)_{-}$ and $\Imp_{0}(V)_{+}$ mentioned above; from this we conclude that \cref{eq:intro:ce} is admissible (\cref{prop:univcead}). The core of this article is the construction of a canonical fusion factorization for the central extension  \cref{eq:intro:ce}, which boils down to the problem of trivializing the pullback of the central extension \cref{intro:implementers} along the map $P\Spin(d) \to L\Spin(d) \to L\SO(d) \to \O_{\res}(V)$.
Here, $P\Spin(d)$ is a space of paths in $\Spin(d)$, and the first of these maps is $\beta \mapsto \beta \cup\beta$; hence the image in $\O_{\res}(V)$  lies in the fixed points of the reflection $\tau$ studied before. In a first step, we use our lift of $\tau$  to $\Imp(V)$ in order to reduce the central extension to one by $\Z_2$. In a second step, we trivialize the resulting double covering using a uniqueness result for standard forms of the von Neumann algebra $\Cl(V_{-})''$, which we apply to two cones in $\mathcal{F}$ that differ by an implementer in the reduced extension. 

We also address the problem of putting an appropriate connection on $\widetilde{L\Spin}(d)$, considered as a principal $\U(1)$-bundle. Such connections are used in string geometry in order to include differential-geometric information, for instance, string connections.
While the  connection itself is straightforward to define, we require a certain compatibility condition with the fusion product. We give a general criterion  for a fusion factorization, which guarantees the compatibility between the induced fusion product and a connection, see \cref{def:fusfaccomp,prop:fuscompconn}. We prove that this criterion is indeed satisfied in our case, and conclude in \cref{prop:compresult} that our construction satisfy the compatibility condition.

In the last part of this paper, \cref{sec:stringgeometry}, we compare our operator-algebraic model of the fusion extension \cref{eq:intro:ce} with the model obtained by transgression of the basic gerbe, which is commonly used in string geometry. Since both models realize the basic central extension of $L\Spin(d)$, it is clear that they are abstractly isomorphic as central extensions. In fact, we prove in \cref{th:comparison} that they are canonically isomorphic as fusion extensions with connection, i.e. the isomorphism can be chosen such that it preserves the fusion products and the connections, and it is characterized uniquely by these properties.  This result will be the starting point for our second paper: it allows  to use our operator-algebraic model for the central extension \cref{eq:intro:ce} for the purposes of string geometry.

\paragraph{Acknowledgements. }
We would like to thank Alan Carey, Bas Jannssens, Matthias Ludewig, Danny Stevenson, Stephan Stolz, and Peter Teichner for helpful discussions.
Moreover, we are grateful to a referee for pointing out a gap in our previous construction of the Banach-Lie group structure on $\Imp(V)$.
This project was funded by the German Research Foundation (DFG) under project code  WA 3300/1-1.

\section{Odd spinors on the circle}

\label{sec:oddspinors}

We discuss a concrete and simple model for the odd spinor bundle on the circle and its sections.

\subsection{The odd spinor bundle on the circle}

We equip the disk $D^2\subset \C$ and the circle $S^1=\partial D^2$ with induced orientations and  metrics. We have $\mathrm{SO}(1)=\{1\}$ and define $\mathrm{Spin}(1) \defeq \mathbb{Z}_2=\{\pm1\}$. Thus, a spin structure on $S^{1}$ is the same as a principal $\mathbb{Z}_2$-bundle over $S^1$, i.e., a double cover.

There are, up to isomorphism, only two spin structures on $S^{1}$, namely, the connected double cover, called \emph{odd}, and the non-connected double cover, called \emph{even}. We will be interested in the odd spin structure, for which we write $\Spin(S^{1})$, and which we realize as the submanifold
        \begin{equation*}
                \Spin(S^{1}) \defeq \{ (e^{i \ph}, \pm e^{i\ph/2}) \in S^{1} \times S^{1} \mid \ph \in \R \},
        \end{equation*}
        equipped with the projection  onto the first component. The $\mathbb{Z}_{2}$-action is on the second component.
        The \emph{odd spinor bundle} $\mathbb{S}$ is the associated complex line bundle
        \begin{equation*}
                \mathbb{S} \defeq \Spin(S^{1}) \times_{\mathbb{Z}_{2}} \C.
        \end{equation*}


\begin{remark}
This definition of the odd spinor bundle fits in the general theory of spin structures, as $\C$ is a model for the Clifford algebra of $\R$. 
\end{remark}

Next we  show that the bundles $\Spin(S^{1})$ and $\mathbb{S}$ can be related to appropriate bundles over the disk $D^{2}$, see \cite{LM89} and the Mathoverflow question \cite{MO12} for a motivation of this discussion. 
        We define $\Spin(2)$ to be the group $\SO(2)$ equipped with the map
        \begin{align*}
        \Spin(2) \defeq \SO(2) \rightarrow \SO(2) \;,\; 
        A  \mapsto A^{2}.
        \end{align*}
Since all principal bundles over $D^2$ are trivializable, we may use the trivial bundle as a model for both the oriented orthonormal frame bundle $\SO(D^2)$ and the (unique) spin structure $\Spin(D^2)$, with the projection map given by
\begin{equation*}
 \Spin(D^2) \defeq D^2 \times \SO(2) \to D^2 \times \SO(2) \defeq \SO(D^2) \;,\; (z,A) \mapsto (z,A^2) \text{.}
\end{equation*}

We write $R_{\varphi} \in \SO(2)$ for the rotation of the plane by an angle $\ph$.
The inclusion $S^1 \subset D^2$ and the isomorphism $S^1 \cong \SO(2)$ induce an embedding 
\begin{align*}
        \iota: \Spin(S^{1}) \rightarrow \Spin(D^{2}) \;,\;
        (e^{i\ph},\pm e^{i\ph/2}) \mapsto (e^{i\ph},  \pm R_{\ph/2})\text{,}
\end{align*}
which yields a commutative square:
\begin{equation*}
        \begin{tikzpicture}[scale=1.5]
        \node (A) at (0,1) {$\Spin(S^{1})$};
        \node (B) at (2,1) {$\Spin(D^{2})$};
        \node (C) at (0,0) {$S^{1}$};
        \node (D) at (2,0) {$D^{2}$};
        \path[->,font=\scriptsize]
        (A) edge node[above]{$\iota$} (B)
        (A) edge node[right]{} (C)
        (B) edge node[right]{} (D)
        (C) edge node[above]{} (D);
        \end{tikzpicture}
\end{equation*}
In this sense, the spin structure on the disk restricts to the odd spin structure on the circle.

Next, we consider the complex line bundle         \begin{equation*}
        \mathbb{D} = \Spin(D^{2}) \times_{\SO(2)} \C.
        \end{equation*}
over $D^2$, by letting $A\in \SO(2)$ act on $\Spin(D^2)=D^2 \times \SO(2)$ by multiplication, and on $\lambda\in \C$ via $(A, \lambda) \mapsto A \triangleright \lambda \defeq A^{-1} \lambda$. This action was chosen such that we obtain a minus sign in the exponent in \cref{eq:trivS}, which in turn determines the identification between the spaces $L^{2}(S^{1})$ and $L^{2}_{-2\pi}$ later on.
The bundle $\mathbb{D}$ is not the spinor bundle of the disk, but it has the following interesting property:

\begin{lemma}
\label{lem:restrictionD}
The restriction of $\mathbb{D}$ to $S^1$ is the odd spinor bundle $\mathbb{S}$.
\end{lemma}

\begin{proof}
It is easy to check that
        \begin{equation*}
        \iota\times\mathrm{id}_{\mathbb{C}}: \Spin(S^{1}) \times \C \rightarrow \Spin(D^{2}) \times \C
        \end{equation*}
descends to  a vector bundle morphism $\mathbb{S} \rightarrow \mathbb{D}$ over the inclusion $S^1 \subset D^2$. 
%
%
It has trivial kernel in each fibre; hence, it induces an isomorphism. 
\end{proof}

Since $\mathbb{D}$ is trivializable, a consequence of \cref{lem:restrictionD} is that $\mathbb{S}$ is trivializable, too. A trivialization of $\mathbb{D}$
may be given by the following pair of inverse maps:
\begin{align*}
D^2 \times \C &\to \mathbb{D} &\mathbb{D} &\to D^2 \times \C
\\
(z,\lambda) &\to [(z,\1),\lambda]_{\SO(2)}
&[(z,A),\lambda]_{\SO(2)} &\mapsto (z,A^{-1}\lambda)
\end{align*}
The induced trivialization of $\mathbb{S}$ is then given by
        \begin{align}
        \label{eq:trivS}
                \mathbb{S}  \rightarrow S^{1} \times \C \;,\; 
                [(e^{i\ph},  \pm e^{i\ph/2}), \lambda]_{\mathbb{Z}_{2}}  \mapsto (e^{i\ph}, \pm e^{- i\ph/2} \lambda).
        \end{align}

\subsection{Smooth and square-integrable odd spinors}

\label{sec:SectionsOfSpinorBundle}

We consider the space of smooth sections of $\mathbb{S}$, which we denote by $\Gamma(\mathbb{S})$.  Because the spinor bundle $\mathbb{S}$ is trivializable, and we have the trivialization \cref{eq:trivS}, we have a canonical isomorphism
\begin{equation*}
t: \Gamma(\mathbb{S}) \to C^{\infty}(S^1,\C)\text{.}
\end{equation*}
However, it turns out to be more natural to identify $\Gamma(\mathbb{S})$ with the space $C^{\infty}_{-2\pi}=C^{\infty}_{-2\pi}(\R,\C)$  of $2 \pi$ anti-periodic smooth maps from the real numbers into the complex numbers. The reason is that in \cref{ex:mainexample,sec:freefermions} we will equip the space $\Gamma(\mathbb{S})$ with a real structure that is easily described when $\Gamma(\mathbb{S})$ is identified with $C^{\infty}_{-2 \pi}$, but takes  an awkward form on $C^{\infty}(S^{1},\mathbb{C})$.

\begin{lemma}
        If $f \in C^{\infty}_{-2\pi}$, then the map
        \begin{align*}
                \R \rightarrow \mathbb{S} \;, \; 
                \ph \mapsto [(e^{i \ph}, e^{i \ph/2}), f(\ph)]_{\mathbb{Z}_{2}}
        \end{align*}
        descends along $\R \to S^1:\varphi\mapsto e^{i\varphi}$ to a smooth section $\sigma_{f}:S^{1} \rightarrow \mathbb{S}$. Furthermore, the map
        \begin{align*}
                \sigma: C^{\infty}_{-2\pi} \rightarrow \Gamma(\mathbb{S}) \, , \,
                f \mapsto \sigma_{f},
        \end{align*}
        is an isomorphism.
\end{lemma}
\begin{proof}
        It suffices to show that, for all $\ph \in \R$, it holds that $\sigma_{f}(\ph) = \sigma_{f}(\ph + 2 \pi)$. Indeed,
        \begin{align*}
        \sigma_{f}(\ph) &= [(e^{i \ph}, e^{i \ph/2}), f(\ph)]_{\mathbb{Z}_{2}} \\&= [(e^{i\ph},-e^{i \ph/2}), -f(\ph)]_{\mathbb{Z}_{2}} \\
        &= [(e^{i(\ph+2 \pi)},e^{i (\ph+2\pi)/2}), f(\ph+2 \pi)]_{\mathbb{Z}_{2}} 
\\&= \sigma_{f}(\ph + 2 \pi).
        \end{align*}
        Let us now consider the converse; to that end, let $\sigma:S^{1} \rightarrow \mathbb{S}$ be a smooth section. There exists a unique  function $f:\R \rightarrow \C$ such that, for all $\ph \in \R$,
        \begin{equation*}
        \sigma(e^{i\ph}) = [(e^{i\ph},e^{i \ph/2}),f(\ph)]_{\mathbb{Z}_{2}}.
        \end{equation*}
        It is not hard to see that $f$ is smooth and satisfies $f(\varphi) = - f(\varphi +  2\pi)$; finally, one observes that $\sigma = \sigma_{f}$.
\end{proof}

\begin{remark}
The composition of the isomorphisms $\sigma$ and $t$ results in an isomorphism
\begin{equation*}
\tilde\sigma \defeq  t \circ \sigma : C^{\infty}_{-2\pi} \to C^{\infty}(S^1,\C)\text{,}
\end{equation*}
such that $(\tilde\sigma f)(e^{i \ph}) = e^{-i \ph/2} f(\ph)$, for all $f\in C^{\infty}_{-2\pi}$ and $\ph \in \R$.
\end{remark}

All three vector spaces ($\Gamma(\mathbb{S})$, $C^{\infty}(S^1,\C)$, and $C^{\infty}_{-2\pi}$) are equipped with $L^2$ inner products, for instance
\begin{equation*}
(f,g) \mapsto \langle f, g \rangle \defeq  \int_{0}^{2 \pi} f(\ph) \overline{g(\ph)} \text{d} \ph, \quad f,g \in C^{\infty}_{-2 \pi}.
\end{equation*}
The canonical isomorphisms $\sigma$, $t$, and $\tilde\sigma$ are isometries.
We denote by $L^2_{-2\pi}$ the Hilbert completion of $C^{\infty}_{-2\pi}$, by $L^2(S^1)$ the Hilbert completion of $C^{\infty}(S^1,\C)$, and by $L^2(\mathbb{S}) $ the Hilbert completion of $\Gamma(\mathbb{S})$. The isomorphisms $\sigma$, $t$, and $\tilde \sigma$ extend to isometric isomorphisms
\begin{equation*}
L^2_{-2\pi}  \cong L^2(\mathbb{S})\cong L^2(S^1)\text{.}
\end{equation*}
The Hilbert space $L^2_{-2\pi}$ has a basis  $\{ \xi_{n} \}_{n \in \mathbb{Z}}$ given by
$\xi_{n}(\ph) = e^{- i\left(n + \frac{1}{2}\right)\ph}$, which is identified under $\tilde\sigma$ with the standard basis  $\{z^{-n-1} \}_{n \in \mathbb{Z}}$ of $L^{2}(S^{1})$.

%

\section{Implementers on Fock spaces}

\label{sec:implementers}

In this section we describe the theory of implementable operators on Fock spaces. Most of the results in this section are well-known; unfortunately, there exist many variations of the basic setting, and many competing conventions, and we have not been able to find a consistent treatment of all aspects of the theory we will need later.

Throughout this section, we let $V$ be a complex Hilbert space equipped with a real structure $\alpha$, i.e.~an  anti-unitary map $\alpha : V \rightarrow V$ with the property that $\alpha^{2} = \1$. Using the real structure we equip $V$ with a non-degenerate symmetric complex bilinear form
 $b: V \times V \rightarrow \C$ defined by
        \begin{equation*}
        b(v,w) = \langle v, \alpha(w) \rangle.
        \end{equation*} 
Our discussion in the present \cref{sec:implementers} will be fairly general; later we will restrict to the situation described below as \cref{ex:mainexample}, and variations thereof.

\subsection{Lagrangians in  Hilbert spaces}

\begin{definition}
        A linear subspace $L \subset V$ is called \emph{Lagrangian} if it is a closed $b$-isotropic subspace, such that $V$ decomposes as $V = L \oplus \alpha(L)$.
\end{definition}

\begin{example}
\label{ex:mainexample}
We consider the Hilbert space  $V \defeq  L^2(\mathbb{S})\otimes \C^{d}$, where $\C^{d}$ is equipped with the standard inner product. Under the identification $L^2(\mathbb{S}) \cong L^2_{-2\pi}$, this Hilbert space has an orthonormal basis indexed by two numbers, $n \in \mathbb{N}$ and $j = 1,...,d$, namely, the functions
\begin{equation*}\label{eq:BasisForV}
        \xi_{n,j}(t) = e^{-i(n+\frac{1}{2})t} \otimes e_{j}, \quad t \in \R
\end{equation*}
where $\{e_{j}\}_{j=1,...,d}$ is the standard basis for $\C^{d}$ (see Section \ref{sec:SectionsOfSpinorBundle}). We consider the real structure $\alpha$ on $V$  defined as the complex anti-linear extension of the map
\begin{equation*}
        \alpha(\xi_{n,j})(t) \defeq  e^{i(n + \frac{1}{2})t} \otimes e_{j} = \xi_{-(n +1),j}(t).
\end{equation*}
In other words, $\alpha$ is pointwise complex conjugation.
We consider the   Lagrangian $L \subset V$  defined to be the closed complex linear span
\begin{equation*}
    L\defeq     \langle \xi_{n,j} \mid n \geqslant 0, j=1,...,d \rangle.
\end{equation*}
We remark that the corresponding subspace of $L^2(S^1) \otimes \C^{d}$ spanned by $\{ \tilde\sigma (\xi_{n}) \}_{n \in \mathbb{N}_{0}} = \{z^{-n-1} \}_{n \in \mathbb{N}_{0}}$ is the $L^2$-closure of the subspace consisting of functions on the circle which extend to anti-holomorphic functions on the disk that vanish at zero.
In precisely that sense, we consider the Lagrangian of spinors on the circle that extend to anti-holomorphic functions on the disk.

\end{example}

We proceed with the general theory of Lagrangians in Hilbert spaces. We denote by $\U(V)$ the usual unitary group of $V$, i.e., 
        \begin{equation*}
        \U(V) \defeq  \{ T \in \text{Gl}(V) \mid T^{*} \langle \cdot, \cdot \rangle = \langle \cdot, \cdot \rangle \}\text{.}
        \end{equation*}
The \emph{orthogonal group} $\O(V)$ is the subgroup of $\U(V)$ of transformations that commute with the real structure $\alpha$, i.e.,
        \begin{equation*}
        \O(V) \defeq  \{ T \in \U(V) \mid \alpha T = T \alpha \}.
        \end{equation*}
It is important to note that if $L \subset V$ is a Lagrangian, and $T \in \O(V)$, then $T(L)$ is Lagrangian; this property does not generally hold for $T \in \U(V)$.

       A \emph{unitary structure} on $V$ is a map $\mc{J} \in \O(V)$ with the property that $\mc{J}^{2} = -\1$, see \cite[Chapter 2, Section 1]{PR95}.
If $L$ is a subspace of $V$, let us write $P_{L}$ for the orthogonal projection onto $L$, and $P_{L}^{\perp}$ for the orthogonal projection onto the orthogonal complement of $L$. A straightforward verification \cite[Chapter 2, Section 1]{PR95} shows the following.

\begin{lemma}\label{lem:LagrangiansVsUnitary}
        There is a one-to-one correspondence between unitary structures $\mc{J} \in \O(V)$ and Lagrangian subspaces $L \subset V$, established by the  assignments
\begin{equation*}
L \mapsto \mc{J}_{L} \defeq  i(P_{L} - P_{L}^{\perp})
\quad\text{ and }\quad
\mc{J} \mapsto \ker(\mc{J} - i \1)\text{.}
\end{equation*}
\end{lemma}

\subsection{Clifford algebras}\label{sec:CliffAndFock-GeneralTheory}

If $A$ is a unital \cstar-algebra, then a complex linear map $f: V \rightarrow A$ is called  \emph{Clifford map} if it satisfies the properties
        \begin{equation*}
        f(v)f(w) + f(w)f(v) = 2 b(v,w) \1
        \quad\text{ and }\quad
                f(v)^{*} = f(\alpha(v))
        \end{equation*}
        for all $v,w\in V$.
We note that every Clifford map is bounded and injective. 
The Clifford \cstar-algebra is the \cstar-algebra through which any Clifford map factors:

\begin{definition}
        A \emph{Clifford \cstar-algebra}  of  $V$  is a unital \cstar-algebra $\Cl(V)$ equipped with a 
Clifford map $\iota : V \rightarrow \Cl(V)$, such that for every unital \cstar-algebra $A$ and every Clifford map $f: V \rightarrow A$   there exists a unique unital isometric $*$-homomorphism $\Cl(f)$ such that the  following diagram commutes: 
        \begin{equation*}
                \begin{tikzpicture}[scale=1.5]
                        \node (A) at (0,1) {$V$};
                        \node (B) at (1,1) {$A$};
                        \node (C) at (0,0) {$\Cl(V)$};
                \path[->,font=\scriptsize]
                        (A) edge node[above]{$f$} (B)
                        (A) edge node[left]{$\iota$} (C)
                        (C) edge node[right]{$\Cl(f)$} (B);
                \end{tikzpicture}
        \end{equation*}
\end{definition}

The Clifford \cstar-algebra $\Cl(V)$ is unique up to unique unital $*$-isomorphism. Its existence is proved in \cite[Sections 1.1 \& 1.2]{PR95}. The construction is based on the purely algebraic Clifford algebra, which is then completed to a \cstar-algebra under a certain norm.
We will need two properties that can easily be deduced from the universal property. The first is that the subspace $\iota(V) \subseteq \Cl(V)$ generates the  Clifford algebra as a \cstar-algebra. The second concerns the complex conjugate Hilbert space, which we denote by  $\overline{V}$, and which we consider equipped with the \emph{same} real structure $\alpha$.  Then,  there is a unique anti-linear unital   $*$-isomorphism $\Cl(V) \rightarrow \Cl(\overline{V})$ fixing $V$ pointwise.

In the following we will fix a Clifford \cstar-algebra $\Cl(V)$ and its Clifford map $\iota: V \to \Cl(V)$. 
If $g \in \O(V)$, then the map $\iota \circ g: V \rightarrow \Cl(V)$ is a Clifford map, and hence there is a unique unital isometric $\ast$-homomorphism $\theta_{g}: \Cl(V) \rightarrow \Cl(V)$, such that the following diagram commutes:
\begin{equation*}
        \begin{tikzpicture}[scale=1.3]
                \node (A) at (0,1) {$V$};
                \node (B) at (2,1) {$V$};
                \node (C) at (0,0) {$\Cl(V)$};
                \node (D) at (2,0) {$\Cl(V)$};
                \path[->,font=\scriptsize]
                (A) edge node[above]{$g$} (B)
                (A) edge (C)
                (B) edge (D)
                (C) edge node[below]{$\theta_{g}$} (D);
        \end{tikzpicture}
\end{equation*}
If $g,g' \in \O(V)$, then we have $\theta_{g}\theta_{g'} = \theta_{gg'}$; 
moreover,  $\theta_{\1}=\mathrm{id}_{\Cl(V)}$.
It follows that $\theta$  is a group homomorphism $\theta: \O(V) \rightarrow \operatorname{Aut}(\Cl(V))$ into the group of unital $*$-automorphisms of $\Cl(V)$. The automorphism $\theta_{g}$ is called the \emph{Bogoliubov automorphism} associated to $g$. Since $\theta_g$ restricts to $g$ on $V$, the homomorphism $\theta$ is injective; in other words, it is a faithful representation of $\O(V)$ on $\Cl(V)$ by unital $*$-automorphisms.

\begin{remark}\label{rem:EvenOddCliffordDeco}
The Bogoliubov automorphism $\theta_{-\1}$ is involutive and hence induces a $\mathbb{Z}_2$-grading on $\Cl(V)$, see \cite[p. 27]{PR95}. The image of $\iota:V \to \Cl(V)$ is  odd. If $A$ is a $\mathbb{Z}_2$-graded \cstar-algebra and the image of a Clifford map $f:V \to A$ is odd, then  the induced $\ast$-homomorphism $\Cl(f): \Cl(V) \to A$ is even, i.e., it preserves the gradings. 
\end{remark}

\subsection{Fock spaces}\label{sec:FockSpaces}

Let $L \subseteq V$ be a Lagrangian subspace. We define the \emph{Fock space} $\mathcal{F}_L$ to be the Hilbert completion of the algebraic exterior algebra
\begin{equation*}
\Lambda L\defeq         \bigoplus_{n=0}^{\infty} \Lambda^{n} L.
\end{equation*}
We equip $\mathcal{F}_L$ with an action of $\Cl(V)$ as follows. Let $(v,w) \in L \oplus \alpha(L) = V$. We use the bilinear form $b: V \times V \rightarrow \C$ to identify $\alpha(L)$ with the dual of $L$, denoted by $L^{*}$, and write $w^{*} \in L^{*}$ for the element corresponding to $w \in \alpha(L)$. Now we write $a(w^{*}): \mathcal{F}_L \rightarrow \mathcal{F}_L$ for contraction with $w^{*}$ and we write $c(v): \mathcal{F}_L \rightarrow \mathcal{F}_L$ for left multiplication. One may now verify that the assignment $(v,w) \mapsto c(v) + a(w^{*})$ is a Clifford map $V \rightarrow \mathcal{B}(\mathcal{F}_L)$.
This means that this map extends to a unital isometric $\ast$-homomorphism $\pi_L:\Cl(V) \rightarrow \mathcal{B}(\mathcal{F}_L)$.
We shall adopt the notation $a \lact v \defeq  \pi_{L} (a)(v)$ for $a \in \Cl(V)$ and $v \in \mc{F}$.
Finally, one may verify that the vector $\Omega \defeq  1 \in \Lambda^{0}L$ is annihilated by $\alpha(L) \subseteq V \subseteq \Cl(V)$. 
In the terminology of \cite{PR95} this means that $\Omega$ is a \emph{vacuum vector} for this representation. 
A basic result (\cite[Theorem 2.4.2]{PR95}) is the following. 

\begin{proposition}\label{lem:FockIrreducible}
        The Fock space  $\mathcal{F}_L$ is an irreducible $\Cl(V)$-representation.
\end{proposition}

As a corollary to \cref{lem:FockIrreducible} we have that the von Neumann algebra $\Cl(V)'' \subseteq \mathcal{B}(\mathcal{F}_L)$ is in fact equal to $\mathcal{B}(\mathcal{F}_L)$, and hence a  factor of type I. This is because irreducibility of $\Cl(V)$ means that $\Cl(V)' = \C \1$ and hence that $\Cl(V)'' = \mc{B}(\mc{F}_L)$.

\begin{remark}\label{rem:EvenOddFockDeco}
        Just like the Clifford algebra, the Fock space $\mathcal{F}_L$ is $\mathbb{Z}_{2}$-graded with graded components the completions of even and odd exterior products of $L$.
The grading on $\mathcal{F}_L$ induces a grading on $\mathcal{B}(\mathcal{F}_L)$, for which the image of $V \to \mathcal{B}(\mathcal{F}_L)$ is odd. Hence, the representation $\pi_{L}:\Cl(V) \rightarrow \mathcal{B}(\mc{F}_{L})$ preserves this grading (see \cref{rem:EvenOddCliffordDeco}).
\end{remark}

\begin{remark}\label{rem:RealVSComplexHilbert}
        Our definition of the Clifford algebra and its representation on  Fock space is consistent with \cite{Ara85} and \cite{BLJ02}. However, there are some competing conventions. 
For example, one might start with a complex Hilbert space $H$, which will play the role of our Lagrangian $L$.
In this case, the Fock space will be the completion of $\Lambda H$, see for example \cite{Ott95} and \cite{Ne09}. Information on the relationship between these approaches is given in Chapter 2.6 of \cite{Ott95}.
In \cite{PR95} yet another approach is taken. There, it is assumed that $V$ is a real Hilbert space, equipped with a unitary structure $\mc{J}: V \rightarrow V$. In this case, $V \otimes \C$ is a complex Hilbert space, naturally equipped with both a real structure and a unitary structure, which puts us in the setting we have described so far. In $\cite{PR95}$, the real Hilbert space is then equipped with a complex structure by setting $iv \defeq  \mc{J}(v)$ for all $v \in V$, and one writes $V_{\mc{J}}$ for the complex Hilbert space obtained in this way. There is then an isometric isomorphism $V_{\mc{J}} \rightarrow L_{\mc{J}}, v \mapsto 2^{-1/2}(v - i \mc{J}v)$, see \cite{PR95} Section 2.1.
\end{remark}

\subsection{Implementable operators}

\label{sec:Implementers}

We now fix  a Lagrangian subspace $L \subseteq V$. We write $\mc F = \mathcal{F}_L$ for its Fock space, and we consider $\Cl(V) \subset \mathcal{B}(\mathcal{F})$ via the faithful representation $\pi_L$. Further, we write  $\mc{J}=\mc{J}_L$ for the unitary structure corresponding to $L$, see \cref{lem:LagrangiansVsUnitary}.
Given an element $g \in \O(V)$, one might wonder if there exists a unitary operator $U \in \U(\mc{F})$, such that the  equation \begin{equation*}
\theta_{g}(a) = UaU^{*}
\end{equation*}
in $\mc{B}(\mc{F})$ holds for all $a \in \Cl(V)$. In this case the  operator $g\in \O(V)$ is said to be \emph{implementable} in $\mc{F}$, and the unitary operator $U$ is called an \emph{implementer} that  \emph{implements} $g$. We recall below the criterion for $g\in \O(V)$ to be implementable and then discuss the structure of the set of all implementers.

For a bounded operator $A\in \mathcal{B}(V)$ we write  $\|A \|$ for the usual operator norm,  
and $\|A\|_2$ for the Hilbert-Schmidt norm, i.e. $\| A \|_{2} \defeq  \trace \left( A^{*} A \right)$. 
We recall that $A$ is called a \emph{Hilbert-Schmidt operator} if $\|A\|_2$ is finite. We have the following important result, see \cite[Theorem 3.3.5]{PR95} or \cite[Theorem 6.3]{Ara85}.

\begin{theorem}\label{thm:Implentable}
An orthogonal operator $g\in \O(V)$ is implementable if and only if $[g, \mc{J}]$ is a Hilbert-Schmidt operator.
\end{theorem}

Suppose $g,h \in \O(V)$ are implementable. Since the Hilbert-Schmidt operators form an ideal in the algebra of bounded operators, the identity
$[gh, \mc{J}] = g[h, \mc{J}] + [g, \mc{J}]h$
implies that $gh\in \O(V)$ is again implementable. Similarly, the equation
$[g^{-1}, \mc{J}] = - g^{-1}[g,\mc{J}]g^{-1}$
tells us that if $g$ is implementable, then $g^{-1}$ is implementable. Thus, the implementable operators form a subgroup, which we call the \emph{restricted orthogonal group} and denote it by $\O_{\res}(V)$. 

\begin{remark}\label{rem:DecomposeWRTL+alphaL}
        Let $P_{L}: V \rightarrow L$ be the projection onto $L$. We recall from \cref{lem:LagrangiansVsUnitary} that $\mc{J} = i(P_{L} - P_{L}^{\perp})$. Let $A \in \mc{B}(V)$. Then, we may write $A$ in block form with respect to the decomposition $V = L \oplus \alpha(L)$ as
        \begin{equation*}
                A = \begin{pmatrix}
                        a & b \\
                        c & d
                \end{pmatrix}.
        \end{equation*}
        We have the relations
        \begin{align*}
                a &= P_{L} A P_{L}, & b &= P_{L} A P_{L}^{\perp}, \\
                c &= P_{L}^{\perp} A P_{L}, & d &= P_{L}^{\perp}AP_{L}^{\perp}.
        \end{align*}
        From which it follows that the condition that $[A,\mc{J}]$ is Hilbert-Schmidt is equivalent to the statement that both $b$ and $c$ are Hilbert-Schmidt. If $A$ is unitary, then we have that if $b$ is Hilbert-Schmidt, then $c$ is Hilbert-Schmidt and vice-versa.
\end{remark}

We claim that $\O_{\res}(V)$ is a Banach-Lie group with the underlying topology  induced from the so-called $\mc{J}$-norm. 
In the following we will describe the Banach-Lie group structure explicitly.
We let $\mc{B}_{\res}(V)$  be the unital algebra
        \begin{equation*}
                \mc{B}_{\res}(V) \defeq  \{ A \in \mc{B}(V) \mid \| [A, \mc{J}] \|_{2} < \infty \}.
        \end{equation*}
On the algebra $\mathcal{B}_{\res}(V)$, the $\mc{J}$-norm is defined by
        \begin{equation*}
                \| A \|_{\mc{J}} \defeq  \|A \| + \| [\mc{J},A] \|_{2} \text{.}
        \end{equation*}
It is elementary to check that the $\mathcal{J}$-norm turns $\mc{B}_{\res}(V)$ into a Banach algebra.
We see that $\O_{\res}(V) = \mathcal{B}_{\res}(V)^{\times} \cap \mathrm{O}(V)$ is a closed subgroup of the Banach-Lie group $\mathcal{B}_{\res}(V)^{\times}$.  Unlike in finite dimensions, this is not sufficient for being a Banach-Lie group itself, making further considerations necessary. To start with, we equip $\O_{\res}(V) \subset \mathcal{B}_{\res}(V)^{\times}$ with the induced topology, so that it becomes a topological group. We note the following result about this topology, which is part of Theorem 6.3 in \cite{Ara85}.

\begin{proposition}\label{lem:OJTwoComponents}
        The topological group $\O_{\res}(V)$ has two connected components.
\end{proposition}

\begin{remark}\label{rem:equivalentNorms}
        We consider the following norms on $\mc{B}_{\res}(V)$:
        \begin{align*}
        A &\mapsto \| A \| + \| [A, \mc{J}] \|_{2},& \quad A &\mapsto \| P_{L} A P_{L} \| + \| [A, \mc{J}] \|_{2}, \\
        A &\mapsto \| A \| + \| P_{L} A P_{L}^{\perp} \|_{2}, & A &\mapsto \| P_{L} A P_{L} \| + \| P_{L} A P_{L}^{\perp} \|_{2}.
        \end{align*}
        Their restrictions to  $\O_{\res}(V)$ are all equivalent.       We mention this because some sources use one of the other norms to define the topology on $\O_{\res}(V)$.
\end{remark}

Next we construct explicitly a Banach-Lie group structure on $\O_{\res}(V)$. 
As usual, in unital Banach algebras the exponential map
\begin{align*}
                 \exp: \mc{B}_{\res}(V) &\rightarrow \mc{B}_{\res}(V)^{\times}\;,\; A \mapsto e^{A}\defeq \sum_{n=0}^{\infty} \frac{1}{n!}A^{n}
\end{align*}
is smooth and a local diffeomorphism at $0$. 
We define a Banach-Lie algebra $\lie{o}_{\res}(V)$ as the subspace
        \begin{equation*}
                \lie{o}_{\res}(V) \defeq  \{ A \in \mc{B}_{\res}(V) \mid A^{*} = - A\text{ and }[\alpha,A]=0 \},
        \end{equation*}
        equipped with the Lie bracket given by the usual commutator bracket of operators.
It will indeed turn out to be the Lie algebra of $\O_{\res}(V)$.

\begin{lemma}\label{lem:ExpIsLocalDiffeo}
        The exponential map is a local diffeomorphism at $0$ from $\lie{o}_{\res}(V)$ to $\O_{\res}(V)$.
\end{lemma}
\begin{proof}
        First, we prove that $\exp (\lie{o}_{\res}(V)) \subseteq \mc{\O}_{\res}(V)$. We have that $A^{*} = -A$ implies that $\exp(A) \in \U(V)$, then  $[\alpha,A]=0$ implies $\exp(A)\in \O_{\res}(V)$. Now, let $W' \subseteq \mc{B}_{\res}(V)$ and $U' \subseteq \mc{B}_{\res}(V)^{\times}$ be open neighbourhoods of $0$ and $1$ respectively with the property that $\exp: W' \rightarrow U'$ is a diffeomorphism. It follows that $W \defeq  W' \cap \lie{o}_{\res}(V)$ is an open neighbourhood of $0 \in \lie{o}_{\res}(V)$ and $U \defeq  U' \cap \O_{\res}(V)$ is an open neighbourhood of $1 \in \O_{\res}(V)$. We claim that $\exp$ maps $W$ diffeomorphically to $U$. Clearly $\exp$ maps $W$ diffeomorphically to $\exp (W)$. It remains to show that $\exp (W) = U = U' \cap \O_{\res}(V)$. The inclusion $\exp(W) \subseteq U' \cap \O_{\res}(V)$ is clear. For the other inclusion it suffices to show  $U' \cap \O_{\res}(V) \subseteq \exp( \lie{o}_{\res}(V))$. Or, in other words, that the logarithm of $A\in U'\cap \O_{\res}(V)$ is in $\lie{o}_{\res}(V)$. This is easily verified using the series expansion of the logarithm around $1$.
\end{proof}

With \cref{lem:ExpIsLocalDiffeo} at hand it is standard to equip $\O_{\res}(V)$ with the structure of a Banach-Lie group with Lie algebra $\lie{o}_{\res}(V)$. 

\subsection{Implementers} 

\label{sec:SmoothStructureOfImp(V)}

We define $\Imp(V)  \subset \U(\mathcal{F})$ to be the set of all implementers of  operators $g\in \O_{\res}(V)$. Suppose that $U \in \Imp(V)$, then since $\theta: \mathrm{O}(V) \to \mathrm{Aut}(\Cl(V)) $ is injective, the operator $g$ that is implemented by $U$ is determined uniquely; in other words, we have a well-defined map 
$q:\mathrm{Imp}(V) \to \O_{\res}(V)$. 
 If $U$ and $U'$ implement operators  $g$ and $g'$, respectively, then $UU'$ implements $gg'$. Likewise, $U^{-1}$ implements $g^{-1}$. Hence, $\mathrm{Imp}(V)$ is a subgroup of $\U(\mathcal{F})$, and $q$ is a group homomorphism. If $U,U'\in \Imp(V)$ implement the same operator $g$, then we have
\begin{equation*}
UaU^{*}=U'aU'^{*}
\end{equation*}
for all $a\in \mathrm{Cl}(V)$, which implies $UU'^{-1} \in \mathrm{Cl}(V)'$ and hence $UU'^{-1}\in \U(1)\1$. Thus, we have a central extension
\begin{equation}
\label{eq:ce}
 \U(1) \to \Imp(V) \to \O_{\res}(V) 
\end{equation}
of groups. Our next goal is to equip $\Imp(V)$ with the structure of a Banach-Lie group, such that \cref{eq:ce} is a central extension of Banach-Lie groups.

For this purpose, we infer the existence of a   local section $\sigma:U \to \Imp(V)$, defined on  an open neighbourhood $U$ of $\1 \in\O_{\res}(V)$ on which the exponential map is injective. We refer to \cite{Ara85,PR95,Ott95,Ne09} for constructions of this section, and recall some steps in the following. 
Let $\mc{L}(\mc{F})$ be the algebra of unbounded skew-symmetric operators on $\mc{F}$, with invariant dense domain equal to the algebraic Fock space $\Lambda L$. We shall outline how to produce for each $A \in \lie{o}_{\res}(V)$ an element $\tilde A  \in \mathcal{L}(\mathcal{F})$, such that $\exp(\tilde A)$ is unitary and implements $\exp(A)$. Then, $\sigma(\exp(A))\defeq \exp(\tilde A)$. In order to define $\tilde A$, we first require the following extension of \cref{rem:DecomposeWRTL+alphaL}, which can be proved by an explicit computation of $[\alpha, A]$.

\begin{lemma}\label{lem:LinearAntiLinearDeco}
        With respect to the decomposition $V = L \oplus \alpha(L)$, an element $A \in \lie{o}_{\res}(V)$ can uniquely be written as
        \begin{equation*}
        A = \begin{pmatrix}
        a & a' \alpha \\
        \alpha a' & \alpha a \alpha
        \end{pmatrix},
        \end{equation*}
        with $a: L \rightarrow L$ a bounded linear skew-symmetric transformation, and $a': L \rightarrow L$ a Hilbert-Schmidt anti-linear skew-symmetric transformation.
\end{lemma}

Given a decomposition of $A$ as in \cref{lem:LinearAntiLinearDeco},
we define a skew-symmetric unbounded operator $\tilde{A}_{0}$ on $\mc{F}$ by
\begin{align*}
        \tilde{A}_{0}: \Lambda^{n} L &\rightarrow \Lambda^{n} L, \\
        f_{1} \wedge ... \wedge f_{n} &\mapsto \sum_{i=1}^{n} f_{1} \wedge ... \wedge af_{i} \wedge ... \wedge f_{n};
\end{align*}
this operator has invariant dense domain $\Lambda L$, see \cite[Section 2.3]{Ott95}. Next, since $a'$ is Hilbert-Schmidt and anti-linear, there exists a unique element $\hat{a}' \in \Lambda^{2} L$, such that
\begin{equation}
\label{lem:notationhat}
        \langle \hat{a}', f_{1} \wedge f_{2} \rangle = \langle a'(f_{1}), f_{2} \rangle,
\end{equation}
for all $f_{1},f_{2} \in L$, see \cite[Lemma D3]{Ne09} or \cite[Section 2.4]{Ott95}. Moreover, the following equation then holds
\begin{equation*}
        \| \hat{a}' \|^{2} = \frac{1}{2} \| a' \|^{2}_{2}.
\end{equation*}
We then obtain a skew-symmetric unbounded operator $\tilde A_{1}$ with invariant dense domain $\Lambda L$ by setting
\begin{align*}
        \tilde A_1: \Lambda^{n} L &\rightarrow \Lambda^{n-2} L \oplus \Lambda^{n+2}L\;,\;           \xi \mapsto \frac{1}{2} \left( \iota_{\hat{a}'} \xi - \hat{a}' \wedge \xi \right)
\end{align*}
where $\iota_{\hat{a}'}$ stands for the adjoint of the map $\xi \mapsto \hat{a}' \wedge \xi$. We now set $\tilde A \defeq  \tilde A_0 + \tilde A_1 \in \mathcal{L}(\mathcal{F})$. It is straightforward to see that $A \mapsto \tilde A$ is linear. That $\exp(\tilde A)$ implements $\exp(A)$ is proved in \cite[page 44]{Ott95}.
This completes our recollection of the construction of $\sigma$. 
We are now in position to state the main result of this subsection.  

\begin{theorem}
\label{theorem:Banachextension}
There exists a unique Banach-Lie group structure on $\Imp(V)$ such that the section $\sigma$ is smooth in an open neighborhood of  $\1$. Moreover, when equipped with this Banach-Lie group structure,  
\begin{equation*}
 \U(1) \to \Imp(V) \to \O_{\res}(V) 
\end{equation*}
is a central extension of Banach-Lie groups.
\end{theorem}

\begin{remark}
This existence of a smooth structure on $\Imp(V)$ is probably well-known and mentioned, e.g., in \cite{Ara85,Ne09}. However, we have not seen the construction of a Banach-Lie group structure explicitly carried out anywhere; in addition, we later need that our section $\sigma$ is smooth w.r.t.~that structure. This is why we describe the Banach-Lie group structure in detail.
\end{remark}

\begin{proof}[Proof of \cref{theorem:Banachextension}]
We write $\O_{\res; 0}(V)$ for the connected component of the identity of $\O_{\res}(V)$,
and we write $\Imp_{0}(V)$ for the restriction of $\Imp(V)$ to  $\O_{\res; 0}(V)$.
We first prove the result for $\Imp_0(V)$.
We choose an open neighborhood $V \subset U$ of $\1$ such that $V^2 \subset U$. 
We let $f_{\sigma}:V \times V \to \U(1)$ be the 2-cocycle associated to the section $\sigma$ via the formula
\begin{equation*}
  \sigma(g_{1})\sigma(g_{2}) =f_{\sigma}(g_{1},g_{2})\sigma(g_{1}g_{2})\text{.}
\end{equation*}
We prove below that $f_{\sigma}$ is smooth in an open neighborhood of  $(\1,\1)$.
It is a standard result that this implies the desired result for  $\Imp_0(V)$.
For the convenience of the reader we have described a proof of this standard result in Appendix \ref{app:BanachCentral}, see \cref{lem:CocycleExtension}. 

In order to show that $f_{\sigma}$ is smooth in an open neighborhood of $(\1,\1)$, we consider the map 
\begin{align*}
\psi_{\sigma}:U \times U &\to \C
\;,\;  (g,h) \mapsto \left \langle  \sigma(g)\sigma(h)\Omega,\Omega  \right \rangle\text{.}
\end{align*} 
 It is proved in \cite[Section 10]{Ne09} that $\psi_{\sigma}$ is smooth in an open neighborhood  of $(\1,\1)$.
 Since $\psi_{\sigma}$ is non-zero at $(\1,\1)$ and smooth, there exists an open neighborhood $U' \subset U$ of $\1$ such that $\psi_{\sigma}(g,h)\neq 0$ for all $g,h\in U'$.  Let $V'$ be an open neighborhood such that $V'^2\subset U'$. 
We note from the definition of $\sigma$ that $\sigma(\1)=\1$. We now compute, for $g_{1}, g_{2} \in V'$,
\begin{align*}
                f_{\sigma}(g_{1},g_{2})  \psi_{\sigma}(1,g_1g_2)&=f_{\sigma}(g_{1},g_{2}) \langle \sigma(g_{1}g_{2})\Omega,\Omega \rangle \\&= \langle f_{\sigma}(g_{1},g_{2}) \sigma(g_{1}g_{2})\Omega, \Omega \rangle \\&=\langle \sigma(g_{1})\sigma(g_{2})\Omega , \Omega \rangle       \\&=\psi_{\sigma}(g_1,g_2)\text{.}
\end{align*}
        It follows that
        \begin{equation*}
                f_{\sigma}(g_{1},g_{2}) = \frac{\psi_{\sigma}(g_1,g_2)}{  \psi_{\sigma}(1,g_1g_2)};
        \end{equation*}
        hence, we see that $f_{\sigma}$ is smooth on $V' \times V'$. 

        To extend the result from the identity component $\Imp_0(V)$ to the full group $\Imp(V)$ it suffices to prove that, for any $U \in \Imp(V)$, the conjugation map $C_{U}: \Imp(V) \rightarrow \Imp(V), x \mapsto UxU^{-1}$ is smooth in an open $\1$-neighbourhood.
        To see this, apply \cref{lem:LocalSuf} with $G = \Imp(V)$ and $K = \Imp_{0}(V)$.
        In fact, due to \cref{lem:OJTwoComponents}, it suffices to find any element $U_0 \in \Imp(V) \setminus \Imp_0(V)$ such that $C_{U_0}: \Imp(V) \rightarrow \Imp(V)$ is smooth in an open $\1$-neighbourhood, because if $U \in \Imp(V) \setminus \Imp_0(V)$, then $U_0^{-1}U \in \Imp(V)_{0}$ and $C_{U} = C_{U_0}C_{U_0^{-1}U}$ exhibits $C_{U}$ as the composition of smooth maps.

        We complete the proof by finding a suitable $U_0$.
        First, pick an arbitrary unit vector $v \in L$, and set $U_0 = (v, \alpha(v))/\sqrt{2} \in L \oplus \alpha(L)$.
        It is shown in \cite[p.~110]{Ara85} that the operator $E \defeq P_{U_0} - P_{U_0}^{\perp}$ is contained in $\O_{\res}(V)\setminus \O_{\res;0}(V)$, and is implemented by $U_0 \in V \subseteq \Cl(V)$.
        Our goal is now to show that the map $C_{U_0}$ is smooth in an open neighbourhood of $\1$.
        Let $A \in \lie{o}_{\res}(V)$, we then have that $C_{U_0}(\exp(\tilde{A}))=U_0 \exp(\tilde{A}) U_0^{*} = \exp(U_0 \tilde{A} U_0^{*})$ implements $E \exp(A) E^{*} = \exp(EAE^{*})$.
        On the other hand, $\exp(\widetilde{EAE^{*}})$ implements $\exp(EAE^{*})$ as well, thus it follows that 
        \begin{equation*}
              \lambda_A \defeq  \exp(-U_0 \tilde{A} U_0^{*}) \exp(\widetilde{EAE^{*}}) \in \U(1),
        \end{equation*}
        and, additionally, that $\exp(U_0 \tilde{A} U_0^{*})$ and $\exp(\widetilde{EAE^{*}})$ commute. 
We now see that $t \mapsto \exp(t U_0 \tilde{A} U_0^{*})$ and $t \mapsto \exp(t\widetilde{EAE^{*}})$
        are strongly continuous, unitary, one-parameter groups; moreover, they commute, whence
        \begin{equation*}
                t \mapsto \lambda_{A}(t) = \exp(-t U_0\tilde{A}U_0^{*}) \exp(t\widetilde{EAE^{*}})
        \end{equation*}
        is a strongly continuous, unitary, one-parameter group as well.
        Because $\lambda_{A}(t) \in \U(1)$,  there exists an imaginary number $z \in i \R$ such that
$\lambda_{A}(t) = e^{tz}$.
        It follows that the two one-parameter subgroups
        \begin{align*}
                t \mapsto \exp(t(U_0 \tilde{A} U_0^{*}+z  \1)) \quad \text{and} \quad  t\mapsto \exp(t\widetilde{EAE^{*}})
        \end{align*}
        are identical.
        Whence $z\1 = \widetilde{EAE^{*}} - U_0\tilde{A}U_0^{*}$ by Stone's theorem on one-parameter unitary groups.
        Using \cite[Eq.~(2.10)]{Ott95}, i.e.~$\langle \widetilde{EAE^{*}} \Omega, \Omega \rangle = 0$, we now compute
        \begin{equation*}
                z = \langle z \Omega, \Omega \rangle = - \langle U_0 \tilde{A} U_0^{*} \Omega, \Omega \rangle = - \langle \tilde{A} v, v \rangle = -\langle av,v \rangle,
        \end{equation*}
        where $a \defeq P_{L} A P_{L}$.
        Now define the smooth map $z: \lie{o}_{\res}(V) \rightarrow \C, A \mapsto -\langle av, v\rangle$.
        The equation, valid for $A$ in an open neighbourhood of $0 \in \lie{o}_{\res}(V)$,
        \begin{equation*}
                C_{U_0}(\sigma(\exp(A))) = \exp(U_0 \tilde{A} U_0) = \exp( \widetilde{EAE^{*}} - z(A)\1)=\sigma(\exp(EAE^{*}))\exp(z(A))^{-1}
        \end{equation*}
        then shows that $C_{U_0}$ is smooth in a neighbourhood of the identity.
\end{proof}

\begin{remark}
        The section $\sigma$ is not continuous when $\Imp(V) \subseteq \U(\mc{F})$ is equipped with the norm-topology; as a consequence, the inclusion $\Imp(V) \to \U(\mathcal{F})$ is not a homomorphism of Banach-Lie groups. 
\end{remark}

With the Banach-Lie group structure on $\Imp(V)$ at hand, we can now use its Banach-Lie algebra, which is a central extension
\begin{equation*}
 \R \to \mathfrak{imp}(V) \to \mathfrak{o}_{\res}(V)\text{.} 
\end{equation*}
Here, we have identified the Lie algebra of $\U(1)$ with $\R$. 
The section $\sigma$ induces a section $\sigma_{*}:\mathfrak{o}_{\res}(V) \to \mathfrak{imp}(V)$, which in turn determines a Lie algebra 2-cocycle $\omega_{\sigma}: \lie{o}_{\res}(V) \times \lie{o}_{\res}(V) \rightarrow \R$ by
\begin{equation}
\label{eq:2cocycle}
\omega_{\sigma}(X,Y) \defeq  [\sigma_{*}(X),\sigma_{*}(Y)]-\sigma([X,Y])\text{.}
\end{equation}
It was computed  in \cite[Theorem 6.10]{Ara85}, resulting in
\begin{align}
\label{eq:CocycleOnOJ}
\omega_{\sigma}(A_{1},A_{2}) = \frac{1}{8} \trace ( \mc{J}[\mc{J},A_{1}][\mc{J},A_{2}]).
\end{align}
The same cocycle can be described in two further ways. 
If we put $A_{3} \defeq  [A_{1},A_{2}]$ and write 
\begin{equation*}
        A_{i} = \begin{pmatrix}
        a_{i} & b_{i} \\
        c_{i} & d_{i}
        \end{pmatrix}
        \end{equation*}
        with respect to the decomposition $V = L \oplus \alpha(L)$, then it is shown in \cite[Theorem 6.10]{Ara85} that
\begin{equation}
\label{eq:CocycleOnOJ2ndform}
        \omega_{\sigma}(A_{1},A_{2}) = \frac{1}{2} \trace \left( [a_{1},a_{2}] - a_{3} \right).
        \end{equation}
Finally, according to \cite[Theorem 10.2]{Ne09} we may write
\begin{equation*}
        A_{i} = \begin{pmatrix}
        a_{i} & a_{i}' \alpha \\
        \alpha a_{i}' & \alpha a_{i} \alpha
        \end{pmatrix}
\end{equation*}
for the decomposition of $A_{i}$ according to \cref{lem:LinearAntiLinearDeco}, and then obtain
\begin{align}
\label{eq:CocycleOnOJ3rdform}
\omega_{\sigma}(A_{1},A_{2}) =- \frac{1}{2i} \trace([a'_{1},a'_{2}]).
\end{align}
To see that \cref{eq:CocycleOnOJ3rdform} and \cref{eq:CocycleOnOJ} coincide, one may use \cref{lem:LagrangiansVsUnitary,rem:DecomposeWRTL+alphaL}.

Because $\mathfrak{imp}(V)$ is a central extension of $\mathfrak{o}_{\res}(V)$ by $\R$, it follows that, as a vector space, the Banach-Lie algebra of $\Imp(V)$ is $\lie{imp}(V) \defeq  \lie{o}_{\res}(V) \oplus \R$. The bracket is then given by
\begin{equation*}
        [(A_{1},\lambda_{1}),(A_{2},\lambda_{2})] = ([A_{1},A_{2}],\omega_{\sigma}(A_{1},A_{2})),
\end{equation*}
and the norm is given by
\begin{equation*}
        \| (A, \lambda) \|^{2} = \|A\|_{\mc{J}}^{2} + | \lambda |^{2}.
\end{equation*}
The exponential  map $\exp: \lie{imp}(V) \to \Imp(V)$ is $(A,\lambda) \mapsto \exp(\tilde A + i\lambda\1)$. 
We may relate $\lie{imp}(V)$ to the algebra $\mathcal{L}(\mathcal{F})$ used in the definition of the section $\sigma$
by considering the injective linear map
        \begin{align*}
                 \lie{imp}(V) &\rightarrow \mc{L}(\mc{F})\;,\;     (A, \lambda) \mapsto \tilde A + i \lambda \1.
        \end{align*}
That way, the exponential map of $\Imp(V)$ factors through the exponential of $\mathcal{L}(\mathcal{F})$.

%

For later purpose, we consider the unitary representation of the Banach-Lie group  $\Imp(V)$ of implementers on the Fock space $\mathcal{F}$, obtained  from the inclusion $\Imp(V) \subset \U(\mathcal{F})$.

\begin{proposition}
The set of smooth vectors
\begin{equation*}
        \mc{F}^{\infty} \defeq  \{ v \in \mc{F} \mid \Imp(V) \rightarrow \mc{F}, U \mapsto Uv \text{ is smooth} \}
\end{equation*}
contains the algebraic Fock space $\Lambda L$; in particular, $\mathcal{F}^{\infty}$ is dense in $\mc{F}$.
\end{proposition}

\begin{proof}
        We claim that the vacuum vector $\Omega$ is a smooth vector for $\Imp(V)$. A general theorem for unitary representations  \cite[Theorem 7.2]{neebdiffvect} implies  that $\Omega$ is a smooth vector if the map $U \mapsto \langle \Omega, U \Omega \rangle$ is smooth in an open neighbourhood of $\1$. 
By our characterization of the Lie group structure on $\Imp(V)$, the map $U \times \U(1) \to \Imp(V):(g,\lambda) \mapsto \lambda\sigma(g)$  is a local diffeomorphism  at $(\1,1)$. Now, it suffices to show that the map $(g,\lambda) \mapsto \langle \Omega, \lambda g(\sigma) \Omega \rangle$ is smooth in a neighborhood of $(\1,1)$. But the latter expression is equal to $\lambda\ \overline{\psi_{\sigma}(1,g)}$, where $\psi_{\sigma}$ appeared in the proof of \cref{theorem:Banachextension}, and is smooth. This proves the claim.

        Let $v \in \Lambda L$; there exists  $a \in \Cl(V)$ such that $a \lact \Omega = v$. We claim that the map $\psi_v:\Imp(V) \rightarrow \mc{F}$ with $\psi_v(U) \defeq  Uv$ is smooth. We have $U v = U a \lact \Omega = \theta_{g}(a) \lact U \Omega$ where $g \defeq  q(U)$; thus, $\psi_v$ can be decomposed as 
        \begin{equation*}
                \begin{tikzpicture}[scale=1.5]
                \node (A1) at (0,1) {$\Imp(V)$};
                                \node (A2) at (1.75,1) {$\O_{\res}(V) \times \Imp(V)$};
                \node (A3) at (4,1) {$\O_{\res}(V) \times \mc{F}$};
                \node (A4) at (6,1) {$\mc{F}$};
                \node (B1) at (0,0.5) {$U$};
                \node (B2) at (1.75,0.5) {$(g,U)$};
                \node (B3) at (4,0.5) {$(g, U\Omega)$};
                                \node (B4) at (6,.5) {$\theta_{g}(a) \lact U\Omega$};
                \path[|->,font=\scriptsize]
                (B1) edge (B2)
                (B2) edge (B3)
                (B3) edge (B4);
                \path[->,font=\scriptsize]
                                (A1) edge (A2)
                                (A2) edge (A3)
                                (A3) edge (A4);
        \end{tikzpicture}
        \end{equation*}
        The first map is clearly smooth. The second map is smooth because $\Omega$ is a smooth vector. To show that the third map is smooth we argue as follows. Because $v$ is in the \emph{algebraic} Fock space we have that $a$ is generated by a finite number of vectors in $V$. Hence, it suffices to show that the map
        \begin{align*}
                \O_{\res}(V) \times \mc{F} & \rightarrow \mc{F}, \\
                (g,f) & \mapsto \theta_{g}(x) \lact f = g(x) \lact f,
        \end{align*}
        is smooth for all $x \in V \subset \Cl(V)$. In order to see that this is true, we observe that that map
      \begin{equation*}
                \begin{tikzpicture}[scale=1]
                \node (A1) at (0,1) {$\O_{\res}(V)$};
                \node (A2) at (2,1) {$\mc{B}_{\res}(V)$};
                \node (A3) at (4,1) {$V$};
                \node (A4) at (6,1) {$\Cl(V)$};
                \node (A5) at (8,1) {$\mc{B}(\mc{F})$};
                \node (B1) at (0,0.25) {$g$};
                \node (B2) at (1.75,0.25) {$g$};
                \node (B3) at (4,0.25) {$g(x)$};
                \node (B4) at (6,.25) {$g(x)$};
                \node (B5) at (8,.25) {$g(x)$};
                \path[|->,font=\scriptsize]
                (B1) edge (B2)
                (B2) edge (B3)
                (B3) edge (B4)
                (B4) edge (B5);
                \path[->,font=\scriptsize]
                (A1) edge (A2)
                (A2) edge (A3)
                (A3) edge (A4)
                (A4) edge (A5);
       \end{tikzpicture}
       \end{equation*}
                                is smooth; the first part is smooth by definition of the Banach-Lie group structure on $\O_{\res}(V)$, and the remaining map  is a bounded linear map between Banach spaces, and hence smooth.
Finally, the evaluation map $\mc{B}(\mc{F}) \times \mc{F} \rightarrow \mc{F}$ is clearly smooth.
\end{proof}

Finally, we recall from \cref{lem:OJTwoComponents} that the topological group $\O_{\res}(V)$ has two connected components; this implies that $\Imp(V)$ has  two connected components as well. We recall from Remark \ref{rem:EvenOddFockDeco} that the Fock space is graded, hence so is $\Imp(V) \subset \mc{B}(\mc{F})$. We now have the following result, which is \cite[Theorem 6.7]{Ara85} and \cite[Remark 10.8]{Ne09}.
\begin{proposition}\label{lem:Imp1Even}
        All elements of $\Imp(V)$ are homogeneous, and all elements of the connected component of the identity in $\Imp(V)$ are even.
\end{proposition}

\begin{remark}
\label{rem:gradingImpV}
        The central extension $\Imp(V) \rightarrow \O_{\res}(V)$ can be considered in the setting of Fréchet-Lie groups. A topological version of this is considered in \cite{Ara85} and in \cite{Ott95}. In these sources the group $\O_{\res}(V)$ is equipped with the $\mc{J}$-strong topology, which is strictly weaker than the $\mc{J}$-norm topology we have considered. Yet it has two components \cite[Theorem 6.3]{Ara85}. In \cite{Carey1984} the coarsest topology on $\O_{\res}(V)$ is determined in which the projective representation on $\mathcal{F}$ is continuous, but this topology is not the one of any manifold.  We refer to  Chapter 2.4, Theorem 6 in \cite{Ott95}, and  \cite{PR95} pages 109 and 110 for more information on the connection between the different treatments of the groups $\O_{\res}(V)$ and $\Imp(V)$.

\end{remark}

\begin{remark}
\label{re:connectionfromsplitting}
The Lie algebra section $\sigma_{*}: \lie{o}_{\res}(V) \to \lie{imp}(V)$ determines a connection $\nu_{\sigma}$ on the Banach principal $\U(1)$-bundle $\Imp(V) \to \O_{\res}(V)$, whose horizontal subspaces are the left-translates of the image of $\sigma_{*}$. As a 1-form on $\Imp(V)$, it is given by $\nu_{\sigma} \defeq   \theta^{\Imp(V)} - \sigma_{*}(q^{*}\theta^{\O_{\res}(V)})$, where $\theta^{G}$ denotes the left-invariant Maurer-Cartan form on a Lie group $G$.
The curvature of $\nu_{\sigma}$ is $-\textstyle\frac{1}{2}\omega_{\sigma}(\theta \wedge \theta) \in \Omega^2(\O_{\res}(V))$. We will further discuss this connection in \cref{sec:comparison}.
\end{remark}

\subsection{The basic central  extension}

\label{sec:basic}

We consider the Banach-Lie group central extension $\Imp(V) \rightarrow \O_{\res}(V)$ of Section \ref{sec:SmoothStructureOfImp(V)},  with respect to the data specified  in   \cref{ex:mainexample}. That is,    $V = L^{2}(\mb{S})\otimes \C^{d}$,  the real structure $\alpha$ is pointwise complex conjugation,  the Lagrangian $L$ consists of those spinors that extend to anti-holomorphic functions on the disk, and the unitary structure $\mathcal{J}$ corresponds to $L$ under the bijection of \cref{lem:LagrangiansVsUnitary}.  
 
Our goal is to give an operator-algebraic construction of the basic central extension 
\begin{equation*}
U(1) \rightarrow \widetilde{L \Spin}(d) \rightarrow L\Spin(d)
\end{equation*}
of the loop group of $\Spin(d)$.
The existence of such models using implementers on Fock space is well-known, see, e.g. \cite{PS86,Ne02,SW}, but we have not found a complete treatment of all aspects and in our specific setting. 
Before we start, we briefly recall how the group $L\Spin(d) \defeq  C^{\infty}(S^1,\Spin(d))$ can be equipped with the structure of Fréchet-Lie group, see \cite[Section 3.2]{PS86}  for more details. We will then downgrade all Banach-Lie groups to Fr\'echet-Lie groups, and handle all smoothness issues within the Fr\'echet setting. 

The vector space $L \lie{spin}(d)$ is a Fréchet space when equipped with the topology of uniform convergence of functions and all partial derivatives. 
The pointwise exponential $\exp : L \lie{spin}(d) \rightarrow L\Spin(d)$ may then be used to define charts for a Fr\'echet-Lie group structure on $L\Spin(d)$. 
We write $\End(d)$ for the algebra of endomorphism  of $\C^{d}$. 
The algebra $L \End(d) \defeq  C^{\infty}(S^1,\End(d))$ is a Fréchet algebra in the same way as $L\Spin(d)$. It acts through bounded operators  on $V=L^{2}( \mb{S})\otimes \C^{d}$ via pointwise multiplication on $\C^{d}$, which we denote by $m: L\End(d) \to \mathcal{B}(V)$. 

\begin{lemma}\label{thm:MIsSmooth}
        The image of $m$ is contained in the Banach algebra $\mc{B}_{\res}(V)$. Furthermore,  $m$ is a  continuous homomorphism of Fréchet algebras.
\end{lemma}

\begin{proof}
        Let $f \in L \End(d)$. Then we may write $f$ as
        \begin{equation*}
        f(z) = \sum_{n\in \mathbb{Z}} A_{n} z^{n}, \quad z \in S^{1},
        \end{equation*}
        where $A_{n}$ are elements of $\End(d)$. In the proof of Theorem 6.3.1 in \cite{PS86} it is shown that we now have
        \begin{equation}\label{eq:MBound2}
        \| [m(f), \mc{J}] \|_{2} = \left( \sum_{n\in \mathbb{Z}} |n| \|A_{n} \|^{2} \right)^{1/2}.
        \end{equation}
        To see that this is finite we proceed as follows. If $f$ is smooth, then its derivative $f'$ is square-integrable, with $L^{2}$-norm given by
        \begin{equation*}
        \| f' \|^{2}_{L^{2}} = \sum_{n \in \mathbb{Z}} n^{2} \| A_{n} \|^{2} \geqslant \sum_{n\in \mathbb{Z}} |n| \|A_{n} \|^{2}.
        \end{equation*}
        This shows that $m(f) \in \mc{B}_{\res}(V)$. Further, because $S^{1}$ is compact, we have
        \begin{equation*}
        | f |_{1} = \sup_{t \in S^{1} } \| f' \| \geqslant \| f' \|_{L^{2}} \geqslant \|[m(f), \mc{J}]\|_{2},
        \end{equation*}
        which will be useful in the next step.
        It is easy to see that $m$ is an algebra homomorphism, and in particular linear. Thus, it remains is to show that $m$ is continuous. A simple calculation shows $\| m(f) \|^{2}\leqslant |f|_{0}$ for arbitrary $f \in L \End(d)$;
        hence,
        \begin{equation*}
        \| m(f) \|_{\mc{J}} = \| m (f) \| + \| [m(f), \mc{J}] \|_{2} \leqslant |f|_{0} + |f|_{1},
        \end{equation*}
        which implies that $m$ is continuous.
\end{proof}
Both $L\End(d)$ and $\mc{B}_{\res}(V)$ are equipped with an exponential map.
We have already investigated the properties of the exponential map on $\mc{B}_{\res}(V)$ in Section \ref{sec:Implementers}; the well-definedness of  the exponential on $L\End(d)$ is a standard fact.
Since $m$ is a  continuous homomorphism of Fr\'echet algebras by \cref{thm:MIsSmooth}, we have a commutative diagram
        \begin{equation*}
        \begin{tikzpicture}[scale=1.3]
        \node (A) at (0,1) {$L \End(d)$};
        \node (B) at (2,1) {$\mc{B}_{\res}(V)$};
        \node (C) at (0,0) {$L \End(d)$};
        \node (D) at (2,0) {$\mc{B}_{\res}(V).$};
        \path[->,font=\scriptsize]
        (A) edge node[above]{$m$} (B)
        (A) edge node[left]{$\exp$} (C)
        (B) edge node[right]{$\exp$} (D)
        (C) edge node[below]{$m$} (D);
        \end{tikzpicture}
        \end{equation*}

One sees readily that $m$ restricts to a map $m: L\SO(d) \rightarrow \O_{\res}(V)$. Using the natural projection map $\Spin(d) \rightarrow \SO(d)$ we obtain a projection map $\pi: L\Spin(d) \rightarrow L\SO(d) $. Let us define $M \defeq  m \circ \pi: L\Spin(d) \rightarrow \O_{\res}(V)$.
\cref{thm:MIsSmooth} now implies:

\begin{proposition}\label{lem:MisSmooth}
        The map $M: L \Spin(d) \rightarrow \O_{\res}(V)$ is homomorphism of Fr\'echet-Lie groups.
\end{proposition}

We may now pull the central extension $\U(1) \to \Imp(V) \rightarrow \O_{\res}(V)$ back along the map $M$ to obtain a central extension
\begin{equation*}
 \U(1) \rightarrow \widetilde{L\Spin}(d) \rightarrow L\Spin(d) 
\end{equation*}
of Fr\'echet-Lie groups,
fitting into a commutative diagram
\begin{equation*}
\begin{tikzpicture}[scale=1.3]
\node (A) at (0,1) {$\widetilde{L \Spin}(d)$};
\node (B) at (2,1) {$\Imp(V)$};
\node (C) at (0,0) {$L \Spin(d)$};
\node (D) at (2,0) {$\O_{\res}(V).$};
\path[->,font=\scriptsize]
(A) edge node[above]{$\widetilde{M}$} (B)
(A) edge (C)
(B) edge (D)
(C) edge node[below]{$M$} (D);
\end{tikzpicture}
\end{equation*}
In particular, we see that  the elements of $\widetilde{L\Spin}(d)$ act through $\widetilde M$ on the Fock space $\mathcal{F}$.
\begin{proposition}\label{LSpinEven}
        As operators on Fock space, all elements of $\widetilde{L\Spin}(d)$ are even.
\end{proposition}
\begin{proof}
        Since $\Spin(d)$ is connected and simply connected, $L\Spin(d)$ is connected; hence $M(L\Spin(d))$ is contained in the connected component of the identity of $\O_{\res}(V)$. \cref{lem:Imp1Even} then tells us that the elements of $\widetilde{L\Spin}(d)$ are even.
\end{proof}

In Section \ref{sec:SmoothStructureOfImp(V)} we considered a local section $\sigma$ of 
the projection $\Imp(V) \rightarrow \O_{\res}(V)$, and we considered the corresponding 2-cocycle $\omega_{\sigma}$  of \cref{eq:CocycleOnOJ}. 
        Next, we pull $\omega_{\sigma}$ back to a cocycle on $L \lie{so}(d)$, and from there to a cocycle on $L \lie{spin}(d)$.
\begin{lemma}
\label{th:cocycleonLspin}
        The pullback of the 2-cocycle $\omega_{\sigma}$ on $\lie{o}_{\res}(V)$ to $L\lie{so}(d)$ is given by
        \begin{align*}
        L\lie{so}(d) \times L\lie{so}(d) &\rightarrow \R\;,\;      (f,g) \mapsto -\frac{1}{4\pi i} \int_{0}^{2 \pi} \trace( f(t) g'(t) ) \text{d} t
        \end{align*}
\end{lemma}
\begin{proof}
        This is essentially Proposition 6.7.1 on page 89 of \cite{PS86}; we present their proof adapted to our notation. 
        Let $f,g \in L\lie{so}(d)$ and let $A_{1},A_{2}$ and $A_{3}$ be the operators corresponding to $f,g$ and $[f,g]$ respectively; and write 
\begin{equation*}
        A_{i} = \begin{pmatrix}
        a_{i} & b_{i} \\
        c_{i} & d_{i}
        \end{pmatrix}
        \end{equation*}
        with respect to the decomposition $V = L \oplus \alpha(L)$. Using the formula \cref{eq:CocycleOnOJ2ndform} for $\omega$,  we need to prove that
        \begin{equation*}
        \frac{1}{2} \trace \left( [a_{1},a_{2}] - a_{3} \right) = - \frac{1}{4 \pi i} \int_{0}^{2 \pi} \trace( f(t)g'(t)) \text{d}t.
        \end{equation*}
        By linearity it suffices to consider $f(t) = X e^{i k t}$ and $g(t) = Y e^{i m t}$ with $X,Y \in \lie{gl}_{d}(\R)$ and $k,m \in \mathbb{Z}$. We now distinguish two cases:
        \begin{description}
                \item[$k+m\neq 0$]: In this case one immediately sees that $\int_{0}^{2\pi} (f(t)g'(t)) \text{d}t = 0$. On the other hand, the operators $[a_{1}, a_{2}]$ and $a_{3}$ have no diagonal components, and hence $\trace([a_{1},a_{2}] - a_{3}) = 0$.
                \item[$k+m = 0$]: In this case we have
                \begin{equation*}
                - \frac{1}{4\pi i} \int_{0}^{2\pi} \trace(f(t)g'(t)) \text{d}t = -\frac{k}{2} \trace(XY).
                \end{equation*}
                On the other hand we note that the operators $[a_{1},a_{2}]$ and $a_{3}$ preserve the subspaces $\C^{d} \cdot z^{q}$ for $q \in \mathbb{N}_{\geqslant 0}$. Let $v \in \C^{d}$ be arbitrary. If $q \geqslant k$, then $[a_{1},a_{2}] v z^{q} = a_{3} v z^{q}$. If $q < k$ then $[a_{1},a_{2}] v z^{q} = -YX v z^{q}$, while $a_{3} v z^{q} = [X,Y] v z^{q}$, hence
                \begin{equation*}
                \frac{1}{2} \trace([a_{1},a_{2}] - a_{3}) = -\frac{1}{2}\sum_{j = 0}^{k-1} \trace(XY) = -\frac{k}{2} \trace(XY).
                \end{equation*}
        \end{description}
        This concludes the proof.
\end{proof}
The last link in our argument is the following lemma, which is well-known and easy to check using any explicit description of the root lattice of $\mathfrak{spin}(d)$.

\begin{lemma} \label{lem:coroots}
        For all $d > 2$, the bilinear form
        \begin{align*}
                \langle \cdot , \cdot \rangle: \lie{spin}(d) \times \lie{spin}(d) & \rightarrow \R\;,\;                 (X,Y)  \mapsto -\frac{1}{2}\trace(XY),
        \end{align*}
        is the basic one, i.e., it is the smallest bilinear form such that $\langle h_{\alpha}, h_{\alpha} \rangle$ is even for every coroot $h_{\alpha}$.
\end{lemma}

\begin{theorem}
\label{basicce}
    If $2<d\neq 4$, then the pullback of the central extension $\U(1)\to\Imp(V) \rightarrow \O_{\res}(V)$ along $M$ is the basic central extension of $L\Spin(d)$.
\end{theorem}
\begin{proof}
        The stipulation $2<d \neq 4$ is necessary, because $\lie{spin}(d)$ is simple only in that case, and there is no basic inner product in the non-simple case.
        According to \cite[Theorem 4.4.1 (iv) \& Proposition 4.4.6]{PS86} and the preceding discussion, a central extension of $L\Spin(d)$, coming from a Lie algebra cocycle $\omega$ is  basic if
        \begin{equation*}
                \omega(f,g) = \frac{1}{2\pi i} \int_{0}^{2\pi} \langle f(t), g'(t) \rangle \text{d}t,
        \end{equation*}
        where $\langle \cdot, \cdot \rangle$ is the basic inner product. Now, \cref{th:cocycleonLspin,lem:coroots} complete the proof.
\end{proof}

\section{Free fermions on the circle}\label{sec:freefermions}

In this section we will be more explicit about the representation of the Clifford \cstar-algebra $\mathrm{Cl}(V)$  on the Fock space $\mathcal{F}=\mathcal{F}_L$, in the case of our main    \cref{ex:mainexample}. Thus,     $V = L^{2}(\mb{S})\otimes \C^{d}$,  where  the real structure $\alpha$ is pointwise complex conjugation,  and the Lagrangian $L$ consists of those spinors that extend to anti-holomorphic functions on the disk. The results of this section will be used in \cref{sec:FusionFactorization} for the construction of a fusion factorization on the basic central extension  of $L\Spin (d)$ that we constructed in \cref{sec:basic}.

\subsection{Reflection of free fermions}

Let us write $I_{+} \subset S^{1}$ for the open upper semi-circle and $I_{-} \subset S^{1}$ for the open lower semicircle:
\begin{align*}
I_{+} &\defeq  \{ e^{i \ph} \in S^{1} \mid \ph \in (0, \pi) \}, &
I_{-} &\defeq  \{ e^{i \ph} \in S^{1} \mid \ph \in (\pi, 2\pi) \}.
\end{align*}
If $f \in C^{\infty}_{-2 \pi}$ (see \cref{sec:SectionsOfSpinorBundle}), then we write $\operatorname{Supp}(f)$ for the support of $f$. We consider the subspaces
\begin{equation*}
        C^{\infty}_{-2 \pi}|_{\pm} \defeq  \{ f\in C^{\infty}_{-2 \pi} \mid \exp(i\operatorname{Supp}(f)) \subseteq I_{\pm} \}
\end{equation*}
and denote their Hilbert completions by $V_{\pm}$. One sees immediately that $V$ decomposes as $V = V_{-} \oplus V_{+}$, and that $\alpha$ restricts to real structures on $V_{\pm}$. The Clifford algebras $\Cl(V_{+})$ and $\Cl(V_{-})$ can be considered as subalgebras of $\Cl(V)$, and the algebra product 
\begin{equation*}
\Cl(V_{-}) \times \Cl(V_{+}) \subset \Cl(V) \times \Cl(V) \to \Cl(V)
\end{equation*}
induces a unital $*$-isomorphism $\Cl(V_{-}) \otimes \Cl(V_{+}) \cong \Cl(V)$ of $\mathbb{Z}_2$-graded \cstar-algebras.

\begin{remark}
More generally, if $I$ is an open connected non-dense subset of the circle $S^{1}$ we can write $V_{I} \subset V$ for the functions with support $I$.
The assignment $I \mapsto \Cl(V_{I})'' \subset \mc{B}(\mc{F})$ is then an isotonic net of von Neumann algebras on the circle. This net can be equipped with the structure of local M\"obius covariant net (conformal net) see, for example, \cite{Gabbiani1993,Bischoff2012,Hen14}. This net is called the \emph{free fermions on the circle}. Here, we use that name to address the decomposition $V = V_{-} \oplus V_{+}$.
\end{remark}

\begin{lemma}
\label{lem:generalposition}
        The decompositions $V = V_{-} \oplus V_{+}$ and $V = L \oplus \alpha(L)$ are in general position, that is,
        \begin{align*}
                V_{+} \cap L &= \{0\} = V_{+} \cap \alpha(L), &
                V_{-} \cap L &= \{0\} = V_{-} \cap \alpha(L).
        \end{align*}
\end{lemma}

\begin{proof}
        This can be proven using the fact that the elements of $L$ are (limits of) anti-holomorphic functions and that the elements of $V_{\pm}$ are (limits of) functions that vanish on an open segment of $S^{1}$;
see for example \cite[Section 14]{wassermann98}.
\end{proof}

We  denote the restriction of complex conjugation to the circle $S^1 \subset \C$  by $v$, that is, we write
\begin{equation*}
        v: S^1 \rightarrow S^1\;,\;  z \mapsto \bar z.
\end{equation*}
Note that $v$ exchanges $I_{+}$ with $I_{-}$.
We denote by $\tau$ the map
\begin{equation*}
        \tau: V \rightarrow V\;,\;  f \mapsto f \circ v.
\end{equation*}
Note that $\tau$ exchanges $V_{+}$ with $V_{-}$ and at the same time exchanges $L$ with $\alpha(L)$. The map $\tau$ is an isometric isomorphism that preserves the real structure,
%
%
%
i.e., $\tau\in \O(V)$. The associated Bogoliubov automorphism
 $\theta_\tau: \Cl(V) \rightarrow \Cl(V)$
 exchanges the subalgebras $\Cl(V_{+})$ and $\Cl(V_{-})$.
Since $\tau$ exchanges $L$ with $\alpha(L)$, one can check that $[\tau,\mc{J}]=2\tau \mc{J}$, which is not Hilbert-Schmidt. Thus, $\tau \notin \O_{\res}(V)$ and hence $\theta_{\tau}$ is not implementable, by \cref{thm:Implentable}. 
Since $\alpha$ interchanges $L$ with $\alpha(L)$ as well, the map $\alpha \circ \tau$ preserves both $L$ and $\alpha(L)$.
\begin{remark}
        We should note that when restricted to the basis of $V$ given in \cref{eq:BasisForV} the map $\alpha \circ \tau$ is $\1$. This means that $\alpha \circ \tau$ is the complex antilinear extension of the map that fixes the basis elements $\{\xi_{n,j}\}$.
\end{remark}

\begin{lemma}
        The map $\alpha \circ \tau: V \rightarrow V$ extends uniquely to a complex anti-linear $*$-automorphism $\kappa : \Cl(V) \rightarrow \Cl(V)$ of the Clifford algebra.
\end{lemma}
\begin{proof}
        Write $\overline{V}$ for the complex conjugate of $V$ (with the same real structure $\alpha$). Write $\overline{b}: \overline{V} \times \overline{V} \rightarrow \C$ for the corresponding complex conjugate bilinear form.
        It is straightforward to check that $\iota\circ \alpha \circ \tau:V \to \Cl(\overline{V})$ is a Clifford map.
        Hence it extends uniquely to a  $*$-homomorphism $\kappa':\Cl(V) \rightarrow \Cl(\overline{V})$.
        Composing $\kappa'$ with the unique anti-linear $*$-isomorphism $\Cl(\overline{V}) \rightarrow \Cl(V)$ we obtain $\kappa$.
\end{proof}

        Let $\operatorname{k}: \mc{F} \rightarrow \mc{F}$ be the \quot{Klein transformation}, that is, $\operatorname{k}$ acts on $\mc{F}$ as the  identity on the even part, and as multiplication by $i$ on the odd part. Note that $\operatorname{k}$ is  unitary. The Klein transformation will be useful for us due to the following property, which is straightforward to check; see, e.g., \cite{Hen14}.

        \begin{lemma}\label{lem:KleinTransform}
                Conjugation by the Klein transformation takes the commutant to the graded commutant. Explicitly, if $\mc{A} \subseteq \mc{B}(\mc{F})$, then $\operatorname{k} \mc{A}' \operatorname{k}^{-1} = \mc{A}^{\backprime}$.
        \end{lemma}

We note the following routine facts, just in order to fix notation.
Let $T: L \rightarrow L$ be an (anti-) unitary isomorphism. Then, there exists a unique (anti-) unitary algebra isomorphism $\Lambda L \rightarrow \Lambda L$ extending $T$. That operator in turn extends to a unique (anti-) unitary isometric isomorphism $\Lambda_{T}: \mc{F} \rightarrow \mc{F}$. 
Moreover, if $T:V \rightarrow V$ is an (anti-) unitary isomorphism that preserves $L$, then we will write $\Lambda_{T}$ as a shortcut for $\Lambda_{T|_{L}}$.

\begin{lemma}\label{lem:kJImplementsKappa}
        The anti-unitary operator $\Lambda_{\alpha \tau}: \mc{F} \rightarrow \mc{F}$ implements the anti-linear automorphism $\kappa$, i.e., for all $a \in \Cl(V)$ we have
        \begin{equation*}
                \kappa(a) = \Lambda_{\alpha \tau} a \Lambda_{\alpha \tau}.
        \end{equation*}
In particular,  $\kappa: \Cl(V) \rightarrow \Cl(V)$ extends uniquely to an anti-linear automorphism of  $\mc{B}(\mc{F})$.
\end{lemma}
\begin{proof}
        Because the algebra $\Cl(V)$ is generated by the elements $f \in V$ it suffices to prove that $\kappa(f) = \Lambda_{\alpha \tau} f \Lambda_{\alpha \tau}$ for all $f \in V$. Let $f \in V$ and let $g_{1}, \dots g_{n}, \in L$, then we compute
        \begin{align*}
                \Lambda_{\alpha \tau} f \lact \Lambda_{\alpha \tau} (g_{1} \wedge \dots \wedge g_{n}) &= \Lambda_{\alpha \tau} f \lact (\alpha \tau g_{1} \wedge \dots \wedge \alpha \tau g_{n}) \\
                &= \alpha \tau(P_{L}f) \wedge g_{1} \wedge \dots \wedge g_{n} + \Lambda_{\alpha \tau} \iota_{\alpha(f)} (\alpha\tau g_{1} \wedge \dots \wedge \alpha \tau g_{n}) \\
                &= P_{L} (\alpha \tau(f)) \wedge g_{1} \wedge \dots \wedge g_{n} + \iota_{\tau \alpha (\alpha(f))} (g_{1} \wedge \dots \wedge g_{n}) \\ 
                &= \kappa (f) \lact g_{1} \wedge \dots \wedge g_{n}.\qedhere
        \end{align*}
\end{proof}

\begin{remark}
        In fact, the map $\Lambda_{\alpha \tau}$ is closely related to the modular conjugation that is part of Tomita-Takesaki theory of the triple $(\Cl(V_{-})'', \mc{F}, \Omega)$, see \cref{thm:modular}.
\end{remark}

If $g \in \O(V)$, then we write $\tau(g) \defeq  \tau \circ g \circ \tau \in \O(V)$.
With this notation we have, for all $f \in V$ and all $g \in \O(V)$, that $\tau(g(f)) = \tau(g) (\tau(f))$.
Because $\tau$ interchanges $L$ with $\alpha(L)$ we have that $\tau$ maps $\O_{\res}(V)$ into $\O_{\res}(V)$; furthermore, it is clear that it is a group homomorphism and that it is smooth.

\begin{lemma}\label{lem:JFlipJ}
        For all $a \in \Cl(V)$ and all $g \in \O_{\res}(V)$ we have
$\kappa(\theta_{g}(a)) = \theta_{\tau(g)} (\kappa(a))$.
\end{lemma}
\begin{proof}
        Let $a = f_{1}\dots f_{n} \in \Cl(V)$ be arbitrary, then we compute
        \begin{equation*}
                \kappa(\theta_{g}(a)) = \kappa((gf_{1})... (gf_{n})) = \tau(gf_{1}^{*}) \dots \tau(gf_{n}^{*}) = \theta_{\tau(g)} \kappa(a). \qedhere
        \end{equation*}
\end{proof}

Now we are in position to prove the first main result about the anti-linear automorphism $\kappa$.

\begin{proposition}\label{thm:KappaTau}
        Let $g \in \O_{\res}(V)$ and let $U \in \Imp(V)$ implement $g$. Then $\kappa(U)$ implements $\tau(g)$. In particular, $\kappa$ restricts to a group homomorphism $\kappa: \Imp(V) \to \Imp(V)$. 
\end{proposition}

\begin{proof}
        Let $a \in \Cl(V)$ be arbitrary, then we compute, using Lemma \ref{lem:JFlipJ} and the fact that $\kappa^{2} = \1$,
        \begin{align*}
                \kappa(U) a \kappa(U)^{*} &= \kappa(U\kappa(a)U^{*}) = \kappa(\theta_{g}(\kappa(a))) = \theta_{\tau(g)}(a).\qedhere
        \end{align*}
\end{proof}

In order to prove  our second main result about the anti-linear automorphism $\kappa$ we require the following lemma about the relation between $\kappa$ and the local section $\sigma$ of the central extension $\Imp(V) \to \O_{\res}(V)$ constructed in \cref{sec:SmoothStructureOfImp(V)}.
We recall that  $\sigma$ was defined by an equation $\sigma (\exp(A)) = \exp(\tilde{A})$, with a certain unbounded operator $\tilde A$ on $\mathcal{F}$ associated to  $A \in \lie{o}_{\res}(V)$. 

\begin{lemma}
\label{lem:kappasigmatau}
We have $\Lambda_{\alpha \tau} \widetilde{A} \Lambda_{\alpha \tau} = \widetilde{\tau(A)}$ for all  $A \in \lie{o}_{\res}(V)$. Moreover, we have
$\kappa \circ \sigma = \sigma \circ \tau$.
\end{lemma}
\begin{proof}
We decompose $A$ and $\tau(A)$ according to \cref{lem:LinearAntiLinearDeco}, that is, we write
        \begin{equation*}
                A = \begin{pmatrix}
                                a & a' \alpha \\
                                \alpha a' & \alpha a \alpha
                        \end{pmatrix}, \quad \quad
                \tau(A) = \tau A \tau = \begin{pmatrix}
                        \tau \alpha a \alpha \tau & \tau \alpha a' \tau \\
                        \tau a' \alpha \tau & \tau a \tau
        \end{pmatrix}.
        \end{equation*}
        Then we have $\widetilde{A} = \widetilde{A}_{0} + \widetilde{A}_{1}$, and $\widetilde{\tau(A)} = \widetilde{\tau(A)_{0}} + \widetilde{\tau(A)_{1}}$. Now, let $f_{1}, \dots, f_{n} \in L$ be arbitrary. We compute
    \begin{align*}
            \Lambda_{\alpha \tau} \widetilde{A}_0 \Lambda_{\alpha \tau} f_{1} \wedge ... \wedge f_{n} &= \Lambda_{\alpha \tau} \sum_{j=1}^{n} \alpha \tau f_{1} \wedge ... \wedge a \alpha \tau f_{j} \wedge ... \wedge \alpha \tau f_{n} \\
            &= \sum_{j=1}^{n} f_{1} \wedge ... \wedge \alpha \tau a \tau \alpha f_{j} \wedge ... \wedge f_{n} \\
            &= \widetilde{\tau(A)_{0}} f_{1} \wedge \dots\wedge f_{n}.
    \end{align*}
    Next, $\widetilde{A}_{1}$ has two parts, a degree-increasing part $\widetilde{A}_{1}^{+}$, and a degree-decreasing part $\widetilde{A}_{1}^{-}$. These are related by $(\widetilde{A}_{1}^{+})^{*} = - \widetilde{A}_{1}^{-}$, hence it suffices to show that $\Lambda_{\alpha \tau} \widetilde{A}_{1}^{+} \Lambda_{\alpha \tau} = \widetilde{\tau(A)}_{1}^{+}$. We compute
    \begin{align*}
            \Lambda_{\alpha \tau}\widetilde{A}_{1}^{+}\Lambda_{\alpha \tau} f_{1} \wedge ... \wedge f_{n} &= \Lambda_{\alpha \tau} \hat{a}' \wedge \Lambda_{\alpha \tau}(f_{1} \wedge ... \wedge f_{n}) = \Lambda_{\alpha \tau} (\hat{a}') \wedge f_{1} \wedge ... \wedge f_{n},
    \end{align*}
 where the notation $\hat a'$ was defined and characterized in \cref{lem:notationhat}.  
    Hence, it suffices to show that $\Lambda_{\alpha \tau}(\hat{a}') = \widehat{\tau \alpha a' \alpha \tau}$. This follows from the computation
    \begin{align*}
            \langle \Lambda_{\alpha \tau}(\hat{a}') , f_{1} \wedge f_{2} \rangle &= \langle \alpha \tau f_{1} \wedge \alpha \tau f_{2}, \hat{a}' \rangle = \langle \alpha \tau f_{2}, a'(\alpha \tau f_{1}) \rangle = \langle \tau \alpha a' \alpha \tau f_{1}, f_{2} \rangle. 
    \end{align*}
Finally, in order to prove the equation  
$\kappa \circ \sigma = \sigma \circ \tau$, we compute: 
\begin{align*}
  \kappa \sigma e^{A} &=  \Lambda_{\alpha \tau} e^{\tilde{A}}\Lambda_{\alpha \tau} 
  = e^{\Lambda_{\alpha \tau}\tilde{A} \Lambda_{\alpha \tau}} =e^{\widetilde{\tau(A)}} =\sigma \tau e^{A}\text{.}
  \qedhere
 \end{align*}
\end{proof}

\begin{proposition}\label{lem:KappaIsSmooth}
        The group homomorphism $\kappa: \Imp(V) \rightarrow \Imp(V)$ is smooth.
\end{proposition}
\begin{proof}
Let $A \in \lie{o}_{\res}(V)$, then we see that $\tau(A) \defeq  \tau A \tau: V \rightarrow V$ is in $\lie{o}_{\res}(V)$ as well.
We consider the real-linear map $\kappa_{\lie{imp}}: \lie{imp}(V) \rightarrow \lie{imp}(V), (A, \lambda) \mapsto (\tau(A), -\lambda)$, which is easily checked to be an isometry, and is hence bounded. 
        Let $(A,\lambda) \in \lie{imp}(V)$.  We have, using \cref{lem:kappasigmatau},
        \begin{align*}
                \kappa \exp(A,\lambda) &= \Lambda_{\alpha \tau} e^{\tilde{A} + i \lambda \1} \Lambda_{\alpha \tau} = e^{\Lambda_{\alpha \tau} \widetilde{A} \Lambda_{\alpha \tau} - i \lambda \1} = e^{\widetilde{\tau(A)} - i \lambda \1} = \exp(\kappa_{\lie{imp}}(A,\lambda)),
        \end{align*}
        this shows that the diagram
        \begin{equation*}
                \begin{tikzpicture}[scale=1.3]
        \node (A) at (0,1) {$\lie{imp}(V)$};
        \node (B) at (2,1) {$\lie{imp}(V)$};
        \node (C) at (0,0) {$\Imp(V)$};
        \node (D) at (2,0) {$\Imp(V)$};
        \path[->,font=\scriptsize]
        (A) edge node[above]{$\kappa_{\lie{imp}}$} (B)
        (A) edge node[left]{$\exp$} (C)
        (B) edge node[right]{$\exp$}(D)
        (C) edge node[below]{$\kappa$} (D);
        \end{tikzpicture}
        \end{equation*}
        commutes;
    which in turn shows that $\kappa$ is smooth and that $\kappa_{\lie{imp}}$ is its derivative.
\end{proof}

\subsection{Tomita-Takesaki theory for the free fermions}\label{sec:TomitaTakesakiFreeFermions}

\label{sec:tomita}
\label{sec:FockSpaceAsStandardForm}

As mentioned before, the free fermions can be extended to a local M\"obius covariant net. Many of the results of this section are inspired by this fact and the theory of such nets \cite[Section II]{Gabbiani1993}.
In general, Tomita-Takesaki theory takes as input a triple $(\mc{A}, \mc{H}, \Omega)$, where the pair $(\mc{A},\mc{H})$ is a von Neumann algebra and $\Omega \in \mc{H}$ is a cyclic and separating vector for the action of $\mc{A}$ on $\mc{H}$. In our current situation,  we consider the von Neumann algebras $\Cl(V_{\pm})''\subseteq \mc{B}(\mc{F})$ generated by $\Cl(V_{+})$ and $\Cl(V_{-})$, respectively. It is well known that these are type III$_{1}$ factors, \cite[Example 4.3.2]{ST04}, \cite[Lemma 2.9]{Gabbiani1993}, \cite[Section 16]{wassermann98}. This statement can also be obtained as  a consequence of the computation of the operator $\Delta^{1/2}$ performed in \cref{sec:ModularConjugation}. We note that these algebras inherit a $\mathbb{Z}_{2}$-grading from the Clifford algebras, which  coincides with the grading of $\mathcal{B}(\mathcal{F})$.

A consequence of the fact that the decompositions $V = V_{-} \oplus V_{+} = L \oplus \alpha(L)$ are in general position (\cref{lem:generalposition}) is that the vacuum vector $\Omega \in \mc{F}$ is cyclic and separating for $\Cl(V_{-})''$; see, for example, \cite[Proposition 3.4]{BLJ02}. Thus, $(\Cl(V_{-})'', \mathcal{F},\Omega)$ is a valid triple. In (ungraded) Tomita-Takesaki theory one then defines the \emph{Tomita operator} $S$ to be the closure of the operator
\begin{align*}
        \Cl(V_{-})'' \Omega &\rightarrow \Cl(V_{-})'' \Omega\;,\; 
        a \lact \Omega \mapsto a^{*} \lact \Omega.
\end{align*}
One then considers the polar decomposition
\begin{equation*}
        S = J_{\Omega} \Delta^{1/2},
\end{equation*}
where $J_{\Omega}$ is a unitary complex-antilinear self-adjoint operator, called the \emph{modular conjugation}, and $\Delta^{1/2}$ is an unbounded positive self-adjoint operator. The main result of Tomita-Takesaki theory for the  triple $(\Cl(V_{-})'',\mc{F},\Omega)$
 is then the following (see, for example, \cite[Chapter IX]{Ta03}).
\begin{theorem}
\label{th:maintota}
        The assignment $a \mapsto J_{\Omega} a^{*} J_{\Omega}$ is an algebra anti-isomorphism from $\Cl(V_{-})''$ to $(\Cl(V_{-})'')'$.
\end{theorem}
In our present case, the modular conjugation can be computed explicitly.
\begin{proposition}
        \label{thm:modular}
        The modular conjugation $J_{\Omega}$ for the triple $(\Cl(V_{-})'',\mathcal{F},\Omega)$ is the operator $J = \operatorname{k}^{-1} \Lambda_{\alpha \tau}$.
\end{proposition}
This result will be used below and then again in \cref{sec:fusionofimplementers}.
Direct proofs can be found in, for example, \cite[Section 15]{wassermann98}, \cite{Hen14}, and \cite{janssens13}.
Because our conventions are slightly different from the references cited, we have adapted the proof in \cite{janssens13} to a proof of \cref{thm:modular} in \cref{sec:ModularConjugation}.
        Another way to prove \cref{thm:modular} is to extend the construction of the free fermions to a local M\"obius covariant net, and then apply the Bisognano-Wichmann theorem, see \cite[Section 3.2]{Bischoff2012}.

The following result goes under the name of \quot{Twisted Haag duality}, \cite{BLJ02}.

\begin{proposition}\label{thm:HaagDuality}
        The graded commutant of $\Cl(V_{-})''$ in $\mathcal{B}(\mathcal{F})$ is $\Cl(V_{+})''$.
\end{proposition}
\begin{proof}
        This is proven in  \cite{Hen14}, as follows: we use that $\operatorname{k} (\Cl(V_{-})'')' \operatorname{k}^{-1}$ is the graded commutant of $\Cl(V_{-})''$, see \cref{lem:KleinTransform}. From \cref{th:maintota,thm:modular} we have  $(\Cl(V_{-})'')' = J \Cl(V_{-})'' J$. Substituting this one obtains that 
        \begin{equation*}
                \operatorname{k} J \Cl(V_{-})'' J \operatorname{k}^{-1} = \Lambda_{\alpha \tau} \Cl(V_{-})'' \Lambda_{\alpha \tau} = \theta_{\alpha \tau}(\Cl(V_{-})'') =  \Cl(V_{+})''
        \end{equation*}
        is the graded commutant of $\Cl(V_{-})''$.
        Alternatively, \cref{thm:HaagDuality} can be proved using \cite[Theorem 5.8]{BLJ02};
to apply that theorem one needs \cref{lem:generalposition}.
\end{proof}

Another important result about the representation of the von Neumann algebra $\Cl(V_{-})''$ on the Fock space $\mathcal{F}$ is that $\mathcal{F}$ is a so-called standard form.
First we recall the definition of a standard form, see \cite[Chapter IX, Definition 1.13]{Ta03}) or \cite{Ha75}.
\begin{definition}\label{def:GeneralStandardForm}
        A \emph{standard form} of a von Neumann algebra $\mc{A}$ is a quadruple $(\mc{A}, H, J, P )$, where $H$ is an  $\mc{A}$-module (i.e., a Hilbert space with a $\ast$-homomorphism $\mathcal{A}\to \mathcal{B}(H)$), $J$ is an anti-linear isometry with $J^{2} = 1$, and $P$ is a closed self-dual cone in $H$, subject to the following conditions:
        \begin{enumerate}
                \item $J\mc{A}J = \mc{A}'$
                \item $JaJ = a^{*}$, if $a \in \mc{A} \cap \mc{A}'$
                \item $J v = v$ for all $v \in P$
                \item $aJaJP \subseteq P$ for all $a \in \mc{A}$
        \end{enumerate}
\end{definition}

The following result, proved in \cite[Theorem 2.3]{Ha75}, tells us that standard forms are unique up to unique isomorphism.
\begin{theorem}\label{thm:StandardFormIsUnique}
        Suppose that $(\mc{A}_{1}, H_{1}, J_{1}, P_{1} )$ and $(\mc{A}_{2}, H_{2}, J_{2}, P_{2} )$ are standard forms. Suppose furthermore that $\pi$ is an isomorphism of $\mc{A}_{1}$ onto $\mc{A}_{2}$, then there exists a unique unitary $u$ from $H_{1}$ onto $H_{2}$ such that
        \begin{enumerate}
                \item $\pi(a) = uau^{*}$, for all $a \in \mc{A}_{1}$
                \item $J_{2} = u J_{1}u^{*}$
                \item $P_{2} = u P_{1}$
        \end{enumerate}
\end{theorem}

If $H$ is a left $\mc{A}$-module with a cyclic and separating vector $\xi \in H$, then one can equip $H$ with the structure of standard form of $\mc{A}$ as follows.
Let $S_{\xi}: a \lact \xi \mapsto a^{*} \lact \xi$ be the Tomita operator and let $S_{\xi} = J_{\xi} \Delta_{\xi}^{1/2}$ be its polar decomposition.
Let $P_{\xi}$ be the closed self-dual cone in $H$ given by the closure of $\{J_{\xi}a J_{\xi} a \lact \xi \in H \mid a \in \mc{A} \}$. It is then a standard result that the quadruple $(\mc{A}, H_{\xi},J_{\xi} ,P_{\xi} )$ is a standard form of $\mc{A}$, we give the appropriate references here. That the modular conjugation $J_{\xi}$ satisfies (1) of \cref{def:GeneralStandardForm} is  the main result of Tomita-Takesaki theory, see \cref{th:maintota}. That (2) of \cref{def:GeneralStandardForm} holds can be found in \cite[Lemma 3]{Ara74}. Finally, that (3) and (4) of \cref{def:GeneralStandardForm} hold can be found in \cite[Theorem 4]{Ara74}.
Applying this to the free fermions on the circle and using that $\Omega \in \mc{F}$ is cyclic and separating for $\Cl(V_{-})''$ we obtain the following result, which we shall require in \cref{sec:fusionofimplementers}. 

\begin{proposition}
\label{lem:FockRepIsStandard}
The quadruple $( \Cl(V_{-})'', \mc{F}, J_{\Omega}, P_{\Omega} )$ is a standard form of $\Cl(V_{-})''$.
\end{proposition}

\subsection{Restriction to the even part}\label{sec:RestrictionToEven}

As mentioned before, the Fock space $\mc{F}$ is a $\mathbb{Z}_{2}$-graded Hilbert space. The von Neumann algebra $\Cl(V_{-})''$ is $\mathbb{Z}_{2}$-graded as well. Even though the Tomita-Takesaki theory of the triple $(\Cl(V_{-})'', \mc{F}, \Omega)$ can to some extent be adapted to the $\mathbb{Z}_{2}$-graded case, some results are only available in the ungraded case. For this reason, we will restrict our considerations to the even parts. In this section we show in which way the results of \cref{sec:tomita}  survive this process.

We write $\mc{F}_{0}$ for the even part of $\mc{F}$. Then, if an algebra $A$ acts on $\mc{F}$, we write $A_{0}$ for the subalgebra consisting of those operators that preserve $\mc{F}_0$. We may now consider the commutant of $\Cl(V_{\pm})_{0}$ in $\mc{B}(\mc{F}_{0})$, and similarly we could consider those elements of $\Cl(V_{\pm})'\subset \mc{B}(\mc{F})$ that preserve $\mc{F}_{0}$.
It is elementary to  show that both procedures have the same result, i.e.,  $\left(\Cl(V_{\pm})_{0}\right)' = \left(\Cl(V_{\pm})'\right)_{0}$. Thus, in the following, we simply write $\Cl(V_{\pm})'_{0}$. 

We have Haag duality for these even algebras (no longer twisted because the commutant coincides with the graded commutant).
\begin{proposition}\label{lem:EvenHaagDuality}
        The commutant of $\Cl(V_{-})_{0}''$ is $\Cl(V_{+})_{0}''$.
\end{proposition}
\begin{proof}
        This follows directly from \cref{thm:HaagDuality}.
\end{proof}

\begin{lemma}
        The vector $\Omega \in \mc{F}_{0}$ is cyclic and separating for $\Cl(V_{-})_{0}''$.
\end{lemma}
\begin{proof}
        The fact that $\Omega$ is separating for $\Cl(V_{-})_{0}''$ is immediate from the fact that $\Omega$ is separating for the bigger algebra $\Cl(V_{-})''$.
        We know that $\Cl(V_{-})'' \lact \Omega$ is dense in $\mc{F}$, it follows that $\Cl(V_{-})_{0}'' \lact \Omega$ is dense in $\mc{F}_{0}$.
\end{proof}

Let us write $S_{0}$ for the Tomita operator corresponding to the triple $(\Cl(V_{-})_{0}'',\mc{F}_{0},\Omega)$, and let $J_{0}$ be the corresponding modular conjugation.

\begin{lemma}
        The operator $S_{0}$ is the restriction of the Tomita operator $S: \mc{F} \rightarrow \mc{F}$ to the subspace $\mc{F}_{0} \subset \mc{F}$. Furthermore, the modular conjugation $J_{0}$ is the restriction of $J$ to $\mc{F}_{0}$ and the modular operator $\Delta^{1/2}_{0}$ is the restriction of $\Delta^{1/2}$ to $\mc{F}_{0}$.
\end{lemma}
\begin{proof}
        The fact that $S$ restricts to $S_{0}$ is obvious. The remaining claims follow from the fact that $J$ and $\Delta^{1/2}$ preserve both the even and the odd subspaces of $\mc{F}$. 
\end{proof}

        We define $\O_{\res}^{\theta}(V)$ to be the subgroup of $\O_{\res}(V)$ consisting of elements that preserve both $V_{+}$ and $V_{-}$. Whenever $g \in \O_{\res}^{\theta}(V)$, then we write $g = g_{-} \oplus g_{+}$, where $g_{\pm} = g|_{V_{\pm}} \in \O(V_{\pm})$.
        Furthermore, we write $\Imp^{\theta}(V)$ for the restriction of the central extension $\Imp(V) \rightarrow \O_{\res}(V)$ to $\O^{\theta}_{\res}(V)$.
                With respect to this decomposition we have $\tau(g_{-} \oplus g_{+}) = \tau(g_{+}) \oplus \tau(g_{-})$.
        We define two subgroups of $\O^{\theta}_{\res}(V)$:
        \begin{align*}
                \O_{\res}(V)_{-} \defeq  \{ g_{-} \oplus \1 \in \O^{\theta}_{\res}(V) \}, \\
                \O_{\res}(V)_{+} \defeq  \{ \1 \oplus g_{+} \in \O^{\theta}_{\res}(V) \}.
        \end{align*}
        Furthermore, we define $\Imp(V)_{-}$ and $\Imp(V)_{+}$ to be the restriction of $\Imp^{\theta}(V) \rightarrow \O^{\theta}_{\res}(V)$ to $\O_{\res}(V)_{-}$ and $\O_{\res}(V)_{+}$ respectively.

\begin{remark}\label{rem:OJVCommute}
        It is clear that both $\O_{\res}(V)_{-}$ and $\O_{\res}(V)_{+}$ are normal in $\O^{\theta}_{\res}(V)$ (but not in $\O_{\res}(V)$) and that they commute with each other. It is furthermore clear that $\Imp(V)_{-}$ and $\Imp(V)_{+}$ are normal in $\Imp^{\theta}(V)$.
\end{remark}

As in the proof of \cref{theorem:Banachextension}, we write $\O_{\res; 0}(V)$ for the connected component of the identity of $\O_{\res}(V)$,
and we write $\Imp_{0}(V)$ for the restriction of $\Imp(V)$ to  $\O_{\res; 0}(V)$, which is even by \cref{lem:Imp1Even}.
Finally, let us write $\Imp_{0}(V)_{-}$ and $\Imp_{0}(V)_{+}$ for the intersections $\Imp(V)_{-} \cap \Imp_{0}(V)$ and $\Imp(V)_{+} \cap \Imp_{0}(V)$ respectively.
With the following result we will ensure in \cref{sec:fusionofimplementers} that our central extension is eligible for applying our method of using fusion factorizations. 

\begin{proposition}\label{lem:ImpCommute}
        The groups $\Imp_{0}(V)_{-}$ and $\Imp_{0}(V)_{+}$ commute with each other.
\end{proposition}
\begin{proof}
        Let $U_{-}\in \Imp_{0}(V)_{-}$ and $U_{+} \in \Imp_{0}(V)_{+}$, and suppose that $U_{-}$ implements $g_{-} \in \O_{\res}(V)_{-}$. Then we have, for all $a \in \Cl(V_{+})$ that
        $U_{-}aU_{-}^{*} = \theta_{g_{-}}(a) = a$,
        and hence $U_{-} \in (\Cl(V_{+})_{0}'')'$. Similarly we see that $U_{+} \in (\Cl(V_{-})_{0}'')'$. \cref{lem:EvenHaagDuality} tells us that $(\Cl(V_{-})_{0}'')' = \Cl(V_{+})_{0}''$. Hence $U_{-}$ commutes with $U_{+}$.
\end{proof}

\section{Fusion on the basic central extension of the loop group}\label{sec:FusionOnLSpin}

In this section we describe the result of  this article, namely the construction of a fusion product on the operator-algebraic model for the basic central extension of $L\Spin(d)$ constructed in \cref{sec:basic}. We recall in \cref{sec:genfusionproducts} some generalities about fusion products on loop group extensions, and introduce in \cref{sec:FusionFactorization}  our new method of constructing fusion products. In \cref{sec:fusionofimplementers} we apply this method in our case, using the results obtained in \cref{sec:freefermions}.

\subsection{Fusion products}

\label{sec:genfusionproducts}

Let $G$ be a Lie group. We write $PG$ for the set of smooth paths in $G$, with sitting instants, i.e.,
\begin{equation}
\label{eq:pathsinG}
PG \defeq  \{\beta:[0,\pi] \to M \;|\; \beta\text{ is smooth and constant around $0$ and $\pi$} \}\text{,}
\end{equation} 
which is a group under the pointwise multiplication. We use sitting instants so that we are able to concatenate arbitrary paths with a common end point: the usual path concatenation $\beta_2 \star \beta_1$ is again a smooth path whenever $\beta_1(\pi)=\beta_2(0)$. 
Unfortunately, with sitting instants, $PG$ is not any kind of manifold. Instead, we regard it as a diffeological space.
A diffeology on a set $X$ consists of a set of maps $c:U \to X$ called \quot{plots}, where  $U \subset \R^{k}$ is open and $k$ can be arbitrary, subject to a number of axioms, see \cite{iglesias1}  for details. A map $f:X \to Y$ between diffeological spaces is called smooth, if its composition with any plot of $X$ results in a plot of $Y$. A diffeological group is a group such that multiplication and inversion are smooth.
Any smooth manifold $M$ or Fr\'echet manifold  $M$ becomes a diffeological space by saying that every smooth map $c: U \to M$, for every open subset $U \subset \R^{k}$ and any $k$, is a plot. 

In case of $PG$, the plots are all maps $c:U \to PG$ such that the adjoint map $c^{\vee}: U \times [0,\pi] \to G: (u,t) \mapsto c(u)(t)$ is smooth.  We remark that path concatenation $\star$ and path reversal $\beta \mapsto \bar\beta$ are smooth group homomorphisms.
The evaluation map $ev: PG \to G \times G, \beta \mapsto (\beta(0),\beta(\pi))$ is a smooth group homomorphism, and since diffeological spaces admit arbitrary fibre products, the iterated fibre products $PG^{[k]} := PG \times_{G \times G} ... \times_{G \times G} PG$ are again diffeological groups. Their plots are just tuples $(c_1,...,c_k)$ of plots of $PG$, such that $ev \circ c_1 = ... = ev \circ c_k$.
We find a smooth group homomorphism
\begin{align*}
PG^{[2]} &\to LG\\ (\beta_1,\beta_2) &\mapsto \beta_1 \cup \beta_2 \defeq  \bar\beta_2 \star \beta_1
\end{align*}
defined on the double fibre product, and three smooth group homomorphisms
\begin{align*}
\pi_1: & PG^{[3]} \to LG :   (\beta_1,\beta_2,\beta_3) \mapsto \beta_1 \cup \beta_2,
\\
\pi_2: & PG^{[3]} \to LG :   (\beta_1,\beta_2,\beta_3) \mapsto \beta_2 \cup \beta_3,
\\
{\mu}: & PG^{[3]} \to LG :   (\beta_1,\beta_2,\beta_3) \mapsto \beta_1 \cup \beta_3,
\end{align*}
defined on the triple fibre product. 

Let $\U(1)\to\widetilde{LG} \stackrel{q}{\rightarrow} LG$ be a Fr\'echet central extension of the loop group $LG$.

\begin{definition}
\label{def:fusion product}
        A \emph{fusion product} on $\widetilde{LG}$  assigns to each element $(\beta_1,\beta_2,\beta_3) \in PG^{[3]}$  a  $\U(1)$-bilinear map
\begin{equation*}
\tilde\mu_{\beta_1,\beta_2,\beta_3}:  \widetilde{LG}_{\beta_1 \cup \beta_2}  \times \widetilde{LG}_{\beta_2\cup\beta_3} \to \widetilde{LG}_{\beta_1 \cup \beta_3}\text{,}
\end{equation*} 
such that the following two conditions are satisfied:
\begin{enumerate}[(i)]

\item 
Associativity: for all $(\beta_1,\beta_2,\beta_3,\beta_4) \in PG^{[4]}$ and all $q_{ij} \in \widetilde{LG}_{\beta_i \cup \beta_j}$, 
\begin{equation*}
\tilde\mu_{\beta_1,\beta_3,\beta_4}(\tilde\mu_{\beta_1,\beta_2,\beta_3}(q_{12} , q_{23}) , q_{34})= \tilde\mu_{\beta_1,\beta_2,\beta_4}(q_{12} , \tilde\mu_{\beta_2,\beta_3,\beta_4}(q_{23} , q_{34}) )\text{.}
\end{equation*}

\item
Smoothness: the map
\begin{align*}
\widetilde{LG} \ttimes{q}{\pi_1} PG^{[3]} \ttimes{\pi_2}{q}\widetilde{LG}  &\to \widetilde{LG}\;,\;
(q_{12},\beta_1,\beta_2,\beta_3,q_{23})\mapsto \tilde\mu_{\beta_1,\beta_2,\beta_3}(q_{12},q_{23})
\end{align*}
is a smooth map between diffeological spaces.

\end{enumerate}
Additionally, a fusion product is called \emph{multiplicative}, if it is a group homomorphism; i.e., for all $(\beta_1,\beta_2,\beta_3),(\beta_1',\beta_2',\beta_3') \in PG^{[3]}$,   $q_{ij}\in \widetilde{LG}_{\beta_i\cup \beta_j}$, and $q'_{ij}\in \widetilde{LG}_{\beta_i' \cup \beta_j'}$, 
\begin{equation*}
\tilde\mu_{\beta_1,\beta_2,\beta_3}(q_{12} , q_{23}) \cdot \tilde\mu_{\beta_1',\beta_2',\beta_3'}(q_{12}' , q_{23}') = \tilde\mu_{\beta_1\beta_1',\beta_2\beta_2',\beta_3\beta_3'}(q^{}_{12}q_{12}' , q^{}_{23}q_{23}')\text{.}
\end{equation*}

\end{definition}

Early versions of fusion products have been studied in \cite{brylinski1} and in \cite{stolz3}. For a more complete treatment of this topic we refer to \cite{waldorf10,Waldorfa,Waldorfc}.
Fusion products are a characteristic  feature of the image of transgression, see \cref{sec:multgrbtrans} and \cite{waldorf10}. The basic central extension of any compact simple Lie group can be obtained by transgression; hence, these models are  automatically equipped with a multiplicative fusion product  \cite{Waldorfa,Waldorfc}. In the present section, we will show that our operator-algebraic model constructed in \cref{sec:basic}  comes with an operator-algebraically defined multiplicative fusion product.

In order to treat connections and fusion products at the same time, we invoke differential forms on diffeological spaces. A differential from on a diffeological space $X$ is a collection $\varphi=\{\varphi_c\}_c$ of differential forms $\varphi_c\in\Omega^k(U)$, one for each plot $c:U \to X$, such that $f^{*}\varphi_{c'}=\varphi_{c}$ for all smooth maps $f: U\to U'$ between the domains of plots $c:U \to X$ and $c':U' \to X'$ with $c'\circ f=c$. Differential forms can be pulled back along smooth maps $f:X \to Y$, by simply putting $(f^{*}\varphi)_c \defeq  \varphi_{c \circ f}$. If a smooth manifold or Fr\'echet manifold is considered as a diffeological space, then diffeological and ordinary differential forms are the same thing, upon identifying $\varphi_c = c^{*}\varphi$.

Suppose that a central extension $\widetilde{LG}$ is equipped with a fusion product. Additionally, we consider it as a principal $\U(1)$-bundle $q:\widetilde{LG} \to LG$, and suppose that it is equipped with a connection $\nu$.
We consider the three smooth maps
\begin{equation}
\label{eq:fustildemaps}
\tilde\pi_1,\tilde\pi_2,\tilde\mu: \widetilde{LG} \ttimes{q}{\pi_1} PG^{[3]} \ttimes{\pi_2}{q}\widetilde{LG} \to \widetilde{LG}\text{,}
\end{equation}
where $\tilde\pi_1$ and $\tilde\pi_2$ are the projections to the first and the third factor, respectively, and $\tilde\mu$ is the map of condition (ii) of \cref{def:fusion product}.

\begin{definition}
\label{def:fusconnpres}
A fusion product $\tilde\mu$ is called \emph{connection-preserving} with respect to a connection $\nu$ on $\widetilde{LG}$ if $\tilde\mu^{*}\nu = \tilde\pi_1^{*}\nu + \tilde\pi_2^{*}\nu$, where $\tilde\pi_1,\tilde\pi_2,\tilde\mu$ are the smooth maps of \cref{eq:fustildemaps}.
\end{definition}

\begin{remark}
Fusion products are best understood using the theory of principal bundles over diffeological spaces, see \cite{waldorf9,waldorf10}. In that terminology, a fusion product is just a smooth bundle morphism
\begin{equation*}
\tilde\mu: \pi_1^{*}\widetilde{LG} \otimes \pi_2^{*}\widetilde{LG} \to \mu^{*}\widetilde{LG}
\end{equation*}
of principal $U(1)$-bundles over $PG^{[3]}$; this (plus a corresponding associativity condition) is equivalent to \cref{def:fusion product}. Moreover, a fusion product is connection-preserving in the sense of \cref{def:fusconnpres} if that bundle morphism is connection-preserving. 
\end{remark}

\subsection{Fusion factorizations}

\label{sec:FusionFactorization}

In this section, we will introduce a new method of defining multiplicative fusion products on central extensions from certain minimal data, called a fusion factorization. We first define a class of central extensions that are admissible for this  method. Let $1 \in PG$ denote the path constantly equal to the unit element in $G$.
\begin{definition}
\label{def:admissible}
        A Fr\'echet central extension $\U(1) \rightarrow \widetilde{LG} \rightarrow LG$ is called \emph{admissible} if it has the following property. For  $\beta_{1},\beta_{3} \in PG$ with endpoints the unit of $G$, and  $q_{12} \in \widetilde{LG}_{\beta_{1} \cup 1}$ and  $q_{23} \in \widetilde{LG}_{1 \cup \beta_{3}}$ we have $q_{12} q_{23} = q_{23} q_{12}$.
\end{definition}

Let $\Delta: PG \rightarrow LG, \beta \mapsto \beta \cup \beta$ be the doubling map.

\begin{definition}
\label{def:fusfac}
        Let $\U(1)\to\widetilde{LG} \rightarrow LG$ be an admissible Fr\'echet  central extension of $LG$. Then, a \emph{fusion factorization} is a smooth group homomorphism $\rho: PG \rightarrow \widetilde{LG}$ such that the following diagram commutes: 
        \begin{equation*}
        \begin{tikzpicture}
                \node (B) at (2,1) {$\widetilde{LG}$};
                \node (C) at (0,0) {$PG$};
                \node (D) at (2,0) {$LG.$};
                \path[->,font=\scriptsize]
                (C) edge node[above]{$\rho$} (B)
                (B) edge (D)
                (C) edge node[below]{$\Delta$} (D);
        \end{tikzpicture}
        \end{equation*}
\end{definition}
The main result of this section is that a  fusion factorization    induces a multiplicative fusion product. 
        Indeed, let $\rho$ be a fusion factorization for an admissible Fr\'echet central extension  $\U(1) \rightarrow \widetilde{LG} \rightarrow LG$. For each triple $(\beta_{1},\beta_{2},\beta_{3}) \in PG^{[3]}$ we set
        \begin{align}
        \label{def:fusionfromfusionfactorization}
                \tilde{\mu}^{\rho}_{\beta_{1},\beta_{2},\beta_{3}}: \widetilde{LG}_{\beta_{1} \cup \beta_{2}} \times \widetilde{LG}_{\beta_{2} \cup \beta_{3}} &\rightarrow \widetilde{LG}_{\beta_{1} \cup \beta_{3}}\;,\; 
                (q_{12}, q_{23})  \mapsto q_{12} \rho(\beta_{2})^{-1} q_{23}.
        \end{align}

\begin{theorem}
\label{thm:fusfromfusfac}
        The map $\tilde{\mu}^{\rho}_{\beta_{1},\beta_{2},\beta_{3}}$ is a multiplicative fusion product.
\end{theorem}
\begin{proof}
First of all, the range of $\tilde{\mu}^{\rho}_{\beta_{1},\beta_{2},\beta_{3}}$ is indeed $\widetilde{LG}_{\beta_{1} \cup \beta_{3}}$, because
\begin{equation*}
        \beta_{1} \cup \beta_{3} = (\beta_{1} \cup \beta_{2}) \Delta(\beta_{2})^{-1} (\beta_{2} \cup \beta_{3}).
\end{equation*}
The map is clearly $\U(1)$-bilinear, and the associativity is straightforward.
        Next, we prove multiplicativity. We start by computing
        \begin{equation*}
                \tilde{\mu}^{\rho}_{\beta_{1},\beta_{2},\beta_{3}}(q_{12} , q_{23}) \tilde{\mu}^{\rho}_{\beta_{1}',\beta_{2}',\beta_{3}'}(q_{12}' , q_{23}') = q_{12} \rho(\beta_{2})^{-1} q_{23} q_{12}' \rho(\beta_{2}')^{-1} q_{23}'
        \end{equation*}
        on the one hand, and
        \begin{align*}
                \tilde{\mu}^{\rho}_{\beta_{1}\beta_{1}',\beta_{2}\beta_{2}',\beta_{3}\beta_{3}'}(q_{12}q_{12}' , q_{23}q_{23}') &= q_{12}q_{12}' \rho(\beta_{2}\beta_{2}')^{-1} q_{23}q_{23}' 
                = q_{12}q_{12}' \rho(\beta_{2}')^{-1}\rho(\beta_{2})^{-1} q_{23}q_{23}'.
        \end{align*}
        We see that to prove multiplicativity it suffices to show that
        \begin{equation*}
                \rho(\beta_{2})^{-1} q_{23} q_{12}' \rho(\beta_{2}')^{-1} = q_{12}' \rho(\beta_{2}')^{-1}\rho(\beta_{2})^{-1} q_{23}.
        \end{equation*}
        This equation holds by the assumption that the central extension $\widetilde{LG}$ is admissible. Finally, let us prove smoothness. The relevant map is
\begin{align*}
\widetilde{LG} \ttimes{q}{\pi_1} PG^{[3]} \ttimes{\pi_2}{q}\widetilde{LG}  &\to \widetilde{LG}\;,\;
 (q_{12},\beta_1,\beta_2,\beta_3,q_{23})\mapsto  q_{12} \rho(\beta_{2})^{-1} q_{23}\text{.}
\end{align*}
Since projections, multiplication, inversion, and $\rho$
are smooth maps, this is a composition of smooth maps and hence smooth.
\end{proof}

In the remainder of this subsection we impose a condition between a fusion factorization $\rho$ and a local section $\sigma$ of the central extension and prove (\cref{prop:fuscompconn}) that this condition guarantees that the associated fusion product $\tilde\mu^{\rho}$ is connection-preserving for the  connection $\nu_{\sigma}$ associated to $\sigma$, see \cref{re:connectionfromsplitting}.

\begin{definition}
\label{def:fusfaccomp}
Let $\U(1)\to\widetilde{LG} \rightarrow LG$ be an admissible Fr\'echet central extension, and suppose that  $\sigma:U \to \widetilde{LG}$ is a smooth local section defined in a neighborhood $U$ of $1\in LG$. A fusion factorization $\rho$ is called \emph{compatible} with $\sigma$, if there exists an open neighborhood $U'\subset \Delta^{-1}(U) \subset PG$ of $1\in PG$, such that $\sigma(\Delta\beta)=\rho(\beta)$ for all $\beta\in U'$.
\end{definition}

The following three lemmas prepare the proof of \cref{prop:fuscompconn} below. 

\begin{lemma}
\label{lem:fusfacflat}
Suppose a fusion factorization $\rho$ is compatible with a section $\sigma$. Then, $\rho$ is  flat  with respect to the connection $\nu_{\sigma}$, i.e., $\rho^{*}\nu_{\sigma}=0$.
\end{lemma}

\begin{proof}
We have to show $(\rho^{*}\nu_{\sigma})_c=0$ for every plot $c:U \to PG$. We first obtain from the definition of $\nu_{\sigma}$ and the definition of $\rho$ that
\begin{equation*}
(\rho^{*}\nu_{\sigma})_c = (\nu_{\sigma})_{\rho\circ c}=(\rho\circ c)^{*}\nu_{\sigma}=(\rho\circ c)^{*}\theta^{\widetilde{LG}} - \sigma_{*}(c^{*}\Delta^{*}\theta^{LG})\text{.}
\end{equation*}
Consider a smooth curve $\varphi:\R\to U$, with $\varphi(0)=: x\in U$ and $\dot\varphi(0)=: v\in T_xU$. Then, we have
\begin{equation*}
(\rho^{*}\nu_{\sigma})_c = \left. \frac{\mathrm{d}}{\mathrm{d}t}\right|_0 \rho(c(x))^{-1}\rho( c(\varphi(t))) - \left. \frac{\mathrm{d}}{\mathrm{d}t}\right|_0 \sigma(\Delta(c(x))^{-1}\Delta(c(\varphi(t))))\text{.}  
\end{equation*}
The compatibility condition of \cref{def:fusfaccomp} now shows that this expression vanishes.
\end{proof}

The section $\sigma$ induces a map
$Z_{\sigma}: LG \times L\mathfrak{g} \to \R$  defined by
\begin{equation*}
Z_{\sigma}(\gamma,X) \defeq  \mathrm{Ad}_\gamma^{-1}(\sigma_{*}(X))-\sigma_{*}(\mathrm{Ad}_\gamma^{-1}(X))\text{;} \end{equation*} 
i.e., it measures the error for the derivative $\sigma_{*}$ being an intertwiner for the adjoint action of $LG$. It is related to the cocycle $\omega_{\sigma}$ by the formula
\begin{equation*}
\omega_{\sigma}(X,Y)=\left . \frac{\mathrm{d}}{\mathrm{d}t} \right |_0 Z_{\sigma}(e^{-tX},Y)\text{,}
\end{equation*}
and satisfies
\begin{equation}
\label{eq:ruleZsigma}
Z_{\sigma}(\gamma_1\gamma_2,X) = Z_{\sigma}(\gamma_1,X)+Z(\gamma_2,\mathrm{Ad}_{\gamma_1}^{-1}(X))\text{.}
\end{equation}
We will use the map $Z_{\sigma}$ in order to describe a relation between the connection $\nu_{\sigma}$ and the group structure on $\widetilde{LG}$. We denote by $\tilde m, \pr_1,\pr_2: \widetilde{LG} \times \widetilde{LG} \to \widetilde{LG}$ the multiplication and the two projections.

\begin{lemma}
\label{eq:multconn}
The equality
\begin{equation*}
\tilde m^{*}\nu_{\sigma} = \pr_1^{*}\nu_{\sigma} + \pr_2^{*}\nu_{\sigma} + (q\times q)^{*}Z_{\sigma}(\pr_2,\pr_1^{*}\theta^{LG})
\end{equation*}
of 1-forms on $\widetilde{LG} \times \widetilde{LG}$ holds. Here, $q:\widetilde{LG} \to LG$ is the projection, and the expression $Z_{\sigma}(\pr_2,\pr_1^{*}\theta^{LG})$ denotes a 1-form on $LG \times LG$, whose value at a point $(\gamma_1,\gamma_2)$ and a tangent vector $(X_1,X_2)\in T_{\gamma_1,\gamma_2}(LG\times LG)$ is given by $Z_{\sigma}(\gamma_2,\gamma_1^{-1}X_1)$.
 
\end{lemma}

\begin{proof}
A  straightforward calculation that only uses the definition of $\nu_{\sigma}$.
\end{proof}

Next, we consider the set $P\mathfrak{g}$ of smooth paths in the Lie algebra $\mathfrak{g}$ with sitting instants, analogous to \cref{eq:pathsinG}. We have a corresponding map $P\mathfrak{g}^{[2]} \to L\mathfrak{g}:(X_1,X_2) \mapsto X_1 \cup X_2 \defeq  \bar X_2 \star X_1$.

\begin{lemma}
\label{lem:Zsigmazero}
Suppose $\widetilde {LG}$ is admissible.
Let   $\beta \in PG$ with endpoints the unit of $G$, and let $X\in P\mathfrak{g}$  with endpoints zero. Then, $Z_{\sigma}(\beta\cup 1,0\cup X)=0$.
\end{lemma}

\begin{proof}

Since the adjoint action of $LG$ on $L\mathfrak{g}$ is pointwise, we have  $\mathrm{Ad}^{-1}_{\beta\cup 1}(0\cup X)=0\cup X$, so that $Z_{\sigma}(\beta\cup 1,0\cup X)=(\mathrm{Ad}^{-1}_{\beta\cup 1}-id)(\sigma_{*}(0\cup X))$. We may represent $0\cup X$ as the derivative of a smooth curve $1\cup \Gamma$ in $LG$, and obtain
\begin{equation*}
Z_{\sigma}(\beta\cup 1,0\cup X) =\left . \frac{\mathrm{d}}{\mathrm{d}t} \right |_0 \widetilde{\beta \cup 1} ^{-1} \cdot \sigma(1 \cup \Gamma(t)) \cdot \widetilde{\beta \cup 1} - \sigma_{*}(0\cup X)\text{,}
\end{equation*}
where $\widetilde{\beta \cup 1}$ is any lift of $\beta \cup 1$ to $\widetilde{LG}$. Admissibility implies now that $Z_{\sigma}(\beta\cup 1,0\cup X)=0$. 
\end{proof}

Now we are in position to prove the following.

\begin{proposition}
\label{prop:fuscompconn}
Let $\U(1) \to \widetilde {LG} \to LG$ be an admissible Fr\'echet central extension, equipped with a smooth section $\sigma$ defined in a neighborhood of the unit of $LG$, and equipped with a compatible fusion factorization $\rho$.
Then, the fusion product $\tilde{\mu}_{\rho}$ is connection-preserving with respect to the connection $\nu_{\sigma}$ in the sense of \cref{def:fusion product}.
\end{proposition}

\begin{proof}
Using the definition of $\tilde\mu^{\rho}$ and \cref{eq:multconn} we obtain, in the notation of \cref{eq:fustildemaps},
\begin{align*}
(\tilde\mu^{\sigma})^{*}\nu_{\sigma} = \tilde\pi_1^{*}\nu_{\sigma}+ \tilde\pi_2^{*}\nu_\sigma +\zeta  
\end{align*}
where $\zeta\in\Omega^1(PG^{[3]})$ is given by
\begin{equation}
\label{eq:zeta}
\zeta \defeq  p_2^{*}i^{*}\rho^{*}\nu_{\sigma}+Z_{\sigma}(\pi_2,p_2^{*}i^{*}\Delta^{*}\theta^{LG})+Z_{\sigma}(1\cup p',\pi_1^{*}\theta^{LG})\text{,}
\end{equation}
where the maps $p_2,p':PG^{[3]}\to PG$ are $p_2(\beta_{1},\beta_2,\beta_3):=\beta_2$ and $p'(\beta_1,\beta_2,\beta_3):=i(\beta_2)\beta_3$, and the map $i: PG \to PG$ is the pointwise inversion. We shall prove that $\zeta=0$. By \cref{lem:fusfacflat}  the first summand in \cref{eq:zeta} vanishes. We write the second summand using \cref{eq:ruleZsigma} as
\begin{equation*}
Z_{\sigma}(\pi_2,p_2^{*}i^{*}\Delta^{*}\theta^{LG})=Z_{\sigma}((\Delta\circ p_2)\cdot (1\cup p'),p_2^{*}i^{*}\Delta^{*}\theta^{LG})=p_2^{*}\Delta^{*}Z_{\sigma}(id,i^{*}\theta^{LG})-Z_{\sigma}(1\cup p',p_2^{*}\Delta^{*}\theta^{LG})\text{.}
\end{equation*}
It is straightforward to show that $\Delta^{*}Z_{\sigma}(id,i^{*}\theta^{LG})=0$, using that the composition $\sigma\circ\Delta$ is a group homomorphism (see \cref{def:fusfac,def:fusfaccomp}).      
All together, we obtain
\begin{equation*}
\zeta= -Z_{\sigma}(1\cup p',p_2^{*}\Delta^{*}\theta^{LG})+Z_{\sigma}(1\cup p',\pi_1^{*}\theta^{LG})=-Z_{\sigma}(1\cup p',p_2^{*}\Delta^{*}\theta^{LG}-\pi_1^{*}\theta^{LG})\text{.}
\end{equation*}
We claim that the values of the 1-form $p_2^{*}\Delta^{*}\theta^{LG}-\pi_1^{*}\theta^{LG}$ on $PG^{[3]}$  are of the form $X \cup 0 \in L\mathfrak{g}$, which proves via \cref{lem:Zsigmazero}  that $\zeta=0$. We consider a plot $c:U \to PG^{[3]}$, consisting of three plots $c_1,c_2,c_3: U \to PG$. Let $x\in U$ and $v\in T_xU$. We compute
\begin{align*}
(\pi_1^{*}\theta^{LG})_c|_x(v)(t)&=(\pi_1 \circ c)^{*}\theta^{LG}|_x(v)(t)\\&=\theta^{G}|_{(c_1 \cup c_2)(x)(t)}(\mathrm{d}(c_1 \cup c_2)|_x(v)(t))\\&=
\begin{cases}
\theta^{G}|_{c_1^{\vee}(x,2t)}(\mathrm{d}c_1^{\vee}|_{(x,2t)}(v,2t)) & \text{ if } 0\leqslant t \leq \pi \\
\theta^{G}|_{c_2^{\vee}(x,2-2t)}(\mathrm{d}c_2^{\vee}|_{(x,2-2t)}(v,2-2t)) & \text{ if } \pi \leqslant t \leqslant 2 \pi
\end{cases}
\end{align*}
and similarly,
\begin{equation*}
(p_2^{*}\Delta^{*}\theta^{LG})_c|_x(v)(t)=
\begin{cases}
\theta^{G}|_{c_2^{\vee}(x,2t)}(\mathrm{d}c_2^{\vee}|_{(x,2t)}(v,2t)) & \text{ if } 0\leqslant t \leq \pi \\
\theta^{G}|_{c_2^{\vee}(x,2-2t)}(\mathrm{d}c_2^{\vee}|_{(x,2-2t)}(v,2-2t)) & \text{ if } \pi \leqslant t \leqslant 2\pi  
\end{cases}
\end{equation*}
This proves the claim.
\end{proof}

\subsection{Fusion factorization for implementers}

\label{sec:fusionofimplementers}

In this subsection we equip our operator-algebraic model of $\widetilde{L\Spin}(d)$ discussed in \cref{sec:basic} with a multiplicative fusion product. For this purpose, we first prove that this central extension is admissible, and then  construct a canonical fusion factorization.

\begin{proposition}
\label{prop:univcead}
        The central  extension $\U(1)\to\widetilde{L\Spin}(d) \rightarrow L\Spin(d)$ is admissible.
\end{proposition}

\begin{proof}
        Let $\beta_{1},\beta_{3} \in P\Spin(d)$ with endpoints equal to the identity  of $L\Spin(d)$. We see that $M(\beta_{1} \cup 1) \in \O_{\res}(V)_{-}$ and $M(1 \cup \beta_{3}) \in \O_{\res}(V)_{+}$, see \cref{sec:RestrictionToEven}. Now, let $q_{12} \in \widetilde{L \Spin}(d)_{\beta_{1} \cup 1}$ and let $q_{23} \in \widetilde{L\Spin}(d)_{1\cup \beta_{3}}$. \cref{LSpinEven} tells us that $M(q_{12}) \in \Imp(V)_{-}$ and $M(q_{23}) \in \Imp(V)_{+}$ are even. Then we apply \cref{lem:ImpCommute} to conclude that $M(q_{12})$ commutes with $M(q_{23})$ and hence $q_{12}$ commutes with $q_{23}$, and we are done.
\end{proof}

In the remainder of this section we will construct a fusion factorization for $\widetilde{L\Spin}(d)$. In fact, we will define a smooth group homomorphism  $\rho: P \Spin(d) \rightarrow \Imp(V)$ such that the diagram
 \begin{equation}
 \label{eq:diagramrho}
        \begin{tikzpicture}
                \node (B) at (4,1) {$\Imp(V)$};
                \node (C) at (0,0) {$P\Spin(d)$};
                \node (D) at (2,0) {$L\Spin(d)$};
                \node (E) at (4,0) {$\O_{\res}(V)$};
                \path[->,font=\scriptsize]
                (C) edge node[above]{$\rho$} (B)
                (B) edge node[right]{$q$} (E)
                (C) edge node[below]{$\Delta$} (D)
                (D) edge node[below]{$M$} (E);
        \end{tikzpicture}
        \end{equation}
 is commutative; this induces a fusion factorization in the obvious way. We start by considering the diffeological group 
\begin{equation*}
        \Imp'(V) \defeq (M\circ \Delta)^{*} \Imp(V) = P\Spin (d) \ttimes{\Delta\circ M}{q} \Imp(V),
\end{equation*}
which is a central extension of $P\Spin(d)$ by $\U(1)$.
We will first reduce it to a central extension by $\mathbb{Z}_{2}$. Let $(\beta,U) \in \Imp'(V)$.
We overload the letter $M$ to denote the obvious map $M: P\Spin(d) \rightarrow \O(V_{-})$, i.e., $M(\beta) = M(\Delta(\beta))|_{V_{-}}$.
Then, using  \cref{thm:KappaTau}, we see that 
\begin{equation*}
q(\kappa(U)) =\tau(M(\Delta(\beta)))= \tau(M(\beta) \oplus \tau M(\beta)\tau^{*}) = M(\beta) \oplus \tau M(\beta)\tau^{*}\text{.}
\end{equation*}
Hence, $U \kappa(U)^{*}$ implements the identity operator, so that $U\kappa(U)^{*} \in \U(1)$. This allows us to define a map $w$ as follows
\begin{align*}
        w: \Imp'(V) &\rightarrow \U(1)\;,\; 
        ( \beta, U) \mapsto U \kappa(U)^{*};
\end{align*}
this map is smooth, because the projection $\Imp'(V) \to \Imp(V)$ is smooth, $\Imp(V)$ is a Lie group, and  $\kappa$ is smooth by \cref{lem:KappaIsSmooth}.
It is straightforward to show that $w$ is a group homomorphism and satisfies  $w( \beta,\lambda U) = \lambda^2 w( \beta,U)$ for all $\lambda \in \U(1)$. It is well-known that such a map determines a reduction of a  central extension from $\U(1)$ to $\mathbb{Z}_2$; in our case, we have a commutative diagram
\begin{equation*}
                \begin{tikzpicture}[scale=1.3]
                \node (E) at (0,2) {$\mathbb{Z}_2$};
                \node (F) at (2,2) {$\U(1)$};
                \node (A) at (0,1) {$w^{-1}(1)$};
                \node (B) at (2,1) {$\Imp'(V)$};
                \node (C) at (0,0) {$P\Spin(d)$};
                \node (D) at (2,0) {$P\Spin(d)$};
                \path[->,font=\scriptsize]
                (E) edge (F)
                (E) edge (A)
                (F) edge (B)
                (A) edge  (B)
                (A) edge  (C)
                (B) edge (D);
                \draw [double equal sign distance] (C) to [out=0, in=180] (D);
        \end{tikzpicture}
\end{equation*}
of diffeological groups and smooth group homomorphisms, whose vertical sequences are exact sequences of groups.

Next we use the modular conjugation $J: \mc{F} \rightarrow \mc{F}$ corresponding to the triple $(\Cl(V_{-})'', \mc{F}, \Omega)$, see \cref{sec:freefermions}. Let $(\beta,U) \in \Imp'(V)$, then using that $U$ is even, \cref{lem:kJImplementsKappa,thm:modular} one sees that $\kappa(U) = JUJ$.
Hence if $(\beta,U) \in w^{-1}(1)$, then $1=U\kappa(U)^{*} = UJU^{*}J $, and hence $UJ = JU$.

The next step is to define a group homomorphism $r:w^{-1}(1) \rightarrow \mathbb{Z}_{2}$; such a group homomorphism then induces a splitting.
To this end, we require the theory of standard forms of von Neumann algebras, see \cref{sec:FockSpaceAsStandardForm}.
Let $(\beta,U) \in  w^{-1}(1)$. We define two cones in $\mc{F}$ as follows
\begin{align*}
        P_{\Omega} &= \{ aJa \Omega \in \mc{F} \mid a \in \Cl(V_{-})'' \}^{\operatorname{cl}}, \\
        P_{U\Omega} &= \{ aJaU \Omega \in \mc{F} \mid a \in \Cl(V_{-})'' \}^{\operatorname{cl}}.
\end{align*}
We then have that the quadruples $(\Cl(V_{-})'',\mc{F},J,P_{\Omega})$ and $(\Cl(V_{-})'',\mc{F},J,P_{U\Omega})$ are standard forms for the von Neumann algebra $\Cl(V_{-})''$, see \cref{def:GeneralStandardForm} and \cref{lem:FockRepIsStandard}. To see that the second quadruple is a standard form, note that the modular conjugation $J_{U\Omega}$ corresponding to the cyclic and separating vector $U\Omega$ is equal to $J$, since $J_{U\Omega} = UJU^{*} = J$.
\cref{thm:StandardFormIsUnique} then implies that there is a unique unitary $u: \mc{F} \rightarrow \mc{F}$ with the following properties
\begin{enumerate}[(1)]
        \item For all $a \in \Cl(V_{-})''$ we have $a = uau^{*}$.
        \item  $u = JuJ$.
        \item  $uP_{\Omega} = P_{U\Omega}$. 
\end{enumerate}
We define $r(\beta,U)\defeq  u$. In the first place, this defines a map $r: w^{-1}(1) \rightarrow \U (\mc{F})$.

It is clear that the operators $\pm \1$ satisfy (1) and (2). The next point of business is to show that $P_{\Omega} = \pm P_{U\Omega}$, from which it follows that $u = \pm \1$. In the sequel, we shall require the even Fock space.
Recall that $U$ is even. Then we define two even cones in $\mc{F}_{0}$ as
        \begin{align*}
                P_{\Omega;0} &= \{ aJa \Omega \in \mc{F}_{0} \mid a \in \Cl(V_{-})_{0}'' \}^{\operatorname{cl}}, \\
                P_{U\Omega;0} &= \{ aJaU \Omega \in \mc{F}_{0} \mid a \in \Cl(V_{-})_{0}'' \}^{\operatorname{cl}}.
        \end{align*}
The quadruples $(\Cl(V_{-})_{0}'',\mc{F}_{0},J,P_{\Omega;0})$ and $(\Cl(V_{-})_{0}'',\mc{F}_{0},J,P_{U\Omega;0})$ are standard forms for the von Neumann algebra $\Cl(V_{-})_{0}''$. The following result is \cite[Theorem 4, parts (5) and (4)]{Ara74}.
\begin{lemma}\label{lem:ConeComparison}
        Let $\xi$ be a cyclic and separating vector in $\mc{F}_{0}$. Then $\xi \in P_{\Omega;0}$ if and only if $J_{\xi} = J$ and
        \begin{equation}\label{eq:CyclicComparison}
                \langle \xi, z \Omega \rangle \geqslant 0
        \end{equation}
        for all $z \in \Cl(V_{-})_{0}'' \cap (\Cl(V_{-})_{0}'')'$, with $z \geqslant 0$.
           Furthermore, if $\xi \in P_{\Omega;0}$ is a cyclic and separating vector, then $P_{\xi} = P_{\Omega;0}$.
\end{lemma}
Note that in our case, we have $\Cl(V_{-})_{0}'' \cap (\Cl(V_{-})_{0}'')' = \C$ (see Theorem \ref{thm:HaagDuality}), hence we may replace $z$ in \cref{eq:CyclicComparison} with $1$.

\begin{lemma}\label{lem:rInZ2}
        If $(\beta,U) \in w^{-1}(1)$, then either $P_{U \Omega} = P_{\Omega}$ (and then $P_{U \Omega;0} = P_{\Omega;0}$) or $P_{U \Omega} = -P_{\Omega}$ (and  then $P_{U \Omega;0} = -P_{\Omega;0}$).
\end{lemma}
\begin{proof}
        We compute
        \begin{equation*}
                \langle \Omega, U \Omega \rangle = \langle J \Omega, U \Omega \rangle = \langle J U \Omega, \Omega \rangle = \langle U \Omega, \Omega \rangle,
        \end{equation*}
        from which follows that $\langle \Omega, U \Omega \rangle$ is real. We now distinguish the following three cases.
        \begin{description}
                \item[$\langle \Omega, U \Omega \rangle > 0$]: In this case,  \cref{lem:ConeComparison} tells us that $U\Omega \in P_{\Omega;0} \subset P_{\Omega}$ and hence $P_{\Omega} = P_{U\Omega}$.
                \item[$\langle \Omega, U \Omega \rangle < 0$]: In this case we have $\langle \Omega, -U \Omega \rangle > 0$, and hence  \cref{lem:ConeComparison} tells us that $P_{\Omega} = P_{-U\Omega} = -P_{U\Omega}$.
                \item[$\langle \Omega, U \Omega \rangle = 0$]: Using  \cref{lem:ConeComparison} it follows that $P_{U\Omega} = P_{\Omega} = - P_{U\Omega}$, a contradiction, hence this case cannot occur. \qedhere
        \end{description}
\end{proof}

It follows that $r(\beta,U) \in \mathbb{Z}_2$.
\begin{lemma}
        The map $r:w^{-1}(1) \to \Z_2$ is a group homomorphism.
\end{lemma}
\begin{proof}
        One sees easily that $r$ is $\mathbb{Z}_{2}$-equivariant. Now, it suffices to show that for all $(\beta,U), (\beta',U') \in w^{-1}(1)$ with
        \begin{equation*}
                r(\beta,U) = 1 = r(\beta',U')
        \end{equation*}
        we have $r(\beta \beta', U U') = 1$. So, let $(\beta,U)$ and $(\beta',U')$ have this property. It is now sufficient to show that $P_{UU'\Omega} = P_{\Omega}$. By assumption we have that $P_{\Omega} = P_{U'\Omega}$, which, by \cref{lem:ConeComparison} implies that
        \begin{equation*}
                0 \leqslant \langle U' \Omega, \Omega \rangle = \langle UU' \Omega, U \Omega \rangle,
        \end{equation*}
        from which, again using \cref{lem:ConeComparison}, it follows that $P_{UU'\Omega} = P_{U\Omega} = P_{\Omega}$, which concludes the proof.
\end{proof}

The group homomorphism $r$ trivializes the central extension $\Z_2 \to w^{-1}(1) \to P\Spin (d)$; the corresponding splitting assigns to $\beta\in P\Spin(d)$ the unique element $(\beta,U)$ in $w^{-1}(1)$  with $r(\beta,U)=\1$, i.e., the unique $(\beta,U)$ with  $UJ=JU$ and $P_{\Omega}=P_{U\Omega}$. In turn, we obtain via $w^{-1}(1)\subset \Imp'(V) \to \Imp(V)$ the claimed group homomorphism
\begin{equation}
\label{eq:deffusfac}
\rho:P\Spin(d) \to \Imp(V)\text{,}
\end{equation}  
making the diagram
\cref{eq:diagramrho} commutative.
We shall summarize the following two characterizations of this map.

\begin{lemma}
\label{lem:charrho}
Let $\beta\in P\Spin(d)$. Then,
\begin{enumerate}

\item 
$\rho(\beta)$ is the unique implementer of $M(\Delta(\beta))$ such that $\rho(\beta)J=J\rho(\beta)$ and $P_{\Omega}=P_{\rho(\beta)\Omega}$. 
\item
$\rho(\beta)$ is the unique implementer of $M(\Delta(\beta))$ such that $\langle \Omega, \rho(\beta) \Omega \rangle > 0$. 

\end{enumerate}
\end{lemma}

\begin{proof}
The first characterization only repeats what we have. We now argue that the second characterization follows from the first. Applying \cref{lem:ConeComparison} for $\xi = \rho(\beta) \Omega$ and $z=1$, and using the fact that $P_{\Omega} = P_{\rho(\beta)\Omega}$ and that $J_{\rho(\beta)\Omega} = \rho(\beta)J\rho(\beta)^{*} = J_{\Omega}$ it follows that $\langle \Omega, \rho(\beta) \Omega \rangle \geqslant 0$. That the inequality is strict then follows from \cref{lem:rInZ2}. This characterization is unique because any two implementers of $M(\Delta(\beta))=M(\beta) \oplus \tau M(\beta) \tau^{*}$ are related by a unique $\lambda \in \U(1)$.
\end{proof}

Now we are in position to finalize our construction of a fusion factorization.

\begin{proposition}
\label{prop:rhoisfusfac}
The map $\rho:P\Spin(d) \to \Imp(V)$ defined in \cref{eq:deffusfac} induces a fusion factorization 
\begin{align*}
P\Spin(d) &\to \widetilde{L\Spin}(d)=M^{*}\Imp(V)
\\\beta&\mapsto (\Delta(\beta),\rho(\beta))\text{.}
\end{align*}
\end{proposition}

\begin{proof}
It remains to show that $\rho$ is smooth, and for this, it suffices to show that the group homomorphism $r:w^{-1}(1) \to \Z_2$ is smooth, where $w^{-1}(1)\subset \Imp'(V)$ is equipped with the subspace diffeology.
The subspace diffeology consists of those plots $c:U \to \Imp'(V)$ whose image is in $w^{-1}(1)$. In particular, if $c:U \to w^{-1}(1)$ is a plot, then the projection to $\Imp(V)$ is smooth. We have to show that $r\circ c:U \to \Z_2$ is smooth, i.e., it is locally constant.
Consider $x\in U$, and let $(\beta_x,U_x):= c(x)$. Consider the open ball around $U$ in $\Imp(V)$ of radius $1/2$, and let $\mathcal{O} \subset U$ be its preimage under the smooth map $U \to \Imp(V)$. We will show that $r$ is constant on $\mathcal{O}$, this proves the lemma. 

Let $y\in \mathcal{O}$, and write $(\beta_y,U_y)\defeq  c(y)$. We want to show that the cone $P_{U_x\Omega;0}$ is equal to the cone $P_{U_y\Omega;0}$. Note that $(\beta_y,U_y)\in w^{-1}(1)$ and $\| U_x - U_y \| \leqslant 1/2$. We set $A \defeq  U_x - U_y$; note that $AJ = JA$. The computation
        \begin{equation*}
        \langle U_x \Omega, A \Omega \rangle = \langle U_x \Omega, A J \Omega \rangle = \langle U_x \Omega, J A \Omega \rangle = \langle A \Omega, J U_x \Omega \rangle = \langle A \Omega, U_x \Omega \rangle,
        \end{equation*}
        implies that $\langle U_x \Omega, A \Omega \rangle$ is real. 
        We then compute
        \begin{align*}
                \langle U_x \Omega, U_y \Omega \rangle &= 1 + \langle U_x \Omega, A \Omega \rangle \geqslant 1 - |\langle U_x \Omega, A \Omega \rangle | \geqslant 1 - \| A \| \geqslant 1/2.
        \end{align*}
        \cref{lem:ConeComparison} now proves that $P_{U_x\Omega;0}=P_{U_y\Omega;0}$.
\end{proof}

We recall that the central extension $\U(1)\to\Imp(V)\to \O_{\res}(V)$ comes with a local section $\sigma:U \to \Imp(V)$ defined in an open $\1$-neighborhood $U$, see \cref{sec:SmoothStructureOfImp(V)}. Hence, its pullback $\widetilde{L\Spin}(d)=M^{*}\Imp(V)$ is equipped with a local section $\tilde\sigma\defeq  M^{*}\sigma$ defined on $\tilde U\defeq  M^{-1}(U)$. 

\begin{lemma}
\label{prop:ourfusfacompconn}
The fusion factorization of \cref{prop:rhoisfusfac}  constructed above is compatible with the local section $\tilde\sigma$ in the sense of \cref{def:fusfaccomp}.
\end{lemma}

\begin{proof}
        Let $U' \subseteq U \subset \O_{\res}(V)$ be an open neighbourhood such that $|\langle \Omega, \sigma(g) \Omega \rangle - 1| < 1$ for $g \in U'$, we put $\tilde U' \defeq  \Delta^{-1}M^{-1}(U')\subset \Delta^{-1}(\tilde U)\subset P\Spin(d)$ and shall prove that $\tilde\sigma(\Delta(\beta))=\rho(\beta)$ for all $\beta\in \tilde U'$. We recall from \cref{lem:kappasigmatau} that $\kappa \circ \sigma = \sigma \circ \tau$  hence $\kappa( \tilde \sigma( \Delta(\beta))) = \tilde\sigma( \Delta(\beta))$. This implies that $\tilde\sigma( \Delta(\beta))$ commutes with the modular conjugation $J$, and hence
        \begin{align*}
                \langle \Omega, \tilde \sigma( \Delta(\beta)) \Omega \rangle &= \langle \Omega, \tilde\sigma( \Delta(\beta)) J \Omega \rangle = \langle \Omega, J \tilde\sigma( \Delta(\beta)) \Omega \rangle = \langle \tilde\sigma( \Delta(\beta)) \Omega, J \Omega \rangle = \overline{ \langle \Omega, \tilde\sigma( \Delta(\beta)) \Omega \rangle},
        \end{align*}
        whence $\langle \Omega, \tilde\sigma( \Delta(\beta)) \Omega \rangle \in \R$. Now, because $M(\Delta(\beta)) \in U'$ we have $\langle \Omega, \tilde\sigma( \Delta(\beta)) \Omega \rangle > 0$, and \cref{lem:charrho} shows $\tilde\sigma( \Delta(\beta))=\rho(\beta)$.
\end{proof}

\begin{remark}
        From \cite[Proof of Theorem 10.2]{Ne09} it follows that $\langle \Omega, \sigma(g) \Omega \rangle > 0$ for all $g \in U$; hence, the reduction to $U' \subseteq U$ is in fact not strictly necessary.
\end{remark}

As a consequence of \cref{prop:ourfusfacompconn,prop:fuscompconn} we obtain:

\begin{proposition}
\label{prop:compresult}
The fusion product $\tilde\mu^{\rho}$ induced by the fusion factorization of \cref{prop:rhoisfusfac} is connection-preserving with respect to the connection $\nu_{\sigma}$ of \cref{re:connectionfromsplitting}. 
\end{proposition}

\section{Implementers and string geometry}\label{sec:stringgeometry}

In this section we show that our operator-algebraic model of a central extension (\cref{sec:basic}), of a connection (\cref{re:connectionfromsplitting}), and of a connection-preserving, multiplicative fusion product (\cref{sec:fusionofimplementers}), yields a so-called \emph{fusion extension with connection}. Then we establish the result that our model is canonically isomorphic to the usual model obtained by transgression, so that both models can be used interchangeably in string geometry.

\subsection{Multiplicative gerbes and transgression}

\label{sec:multgrbtrans}

Let $G$ be a Lie group; important for us is $G=\Spin(d)$. We consider  multiplicative bundle gerbes $\mathcal{G}$ over $G$. We shall recall some minimal facts. Bundle gerbes are geometric objects that represent  classes in the degree three integral cohomology \cite{murray,murray2}. A multiplicative bundle gerbe over $G$ is a bundle gerbe $\mathcal{G}$ over $G$ equipped with an isomorphism
\begin{equation*}
\mathcal{M}: \pr_1^{*}\mathcal{G} \otimes \pr_2^{*}\mathcal{G} \to m^{*}\mathcal{G}
\end{equation*}
over $G \times G$, where $\pr_i$ denote the projections, and $m:G \times G \to G$ is the multiplication of $G$. Moreover, this isomorphism has to be coherently associative \cite{carey4}. Multiplicative bundle gerbes have characteristic  classes in $H^4(BG,\mathbb{Z})$; forgetting the multiplicative structure realizes the usual homomorphism $H^4(BG,\Z) \to H^3(G,\Z)$, see \cite{carey4,waldorf5}. For a compact, simple, connected, simply-connected Lie group, both cohomology groups are isomorphic to $\Z$, and above homomorphism is the identity. A bundle gerbe $\mathcal{G}$ over $G$ that represents a generator in $H^3(G,\Z)$ is called a \emph{basic bundle gerbe}; thus, a basic bundle gerbe admits a (up to isomorphism) unique multiplicative structure. Concrete constructions of a basic bundle gerbe are described in \cite{gawedzki1,meinrenken1}, while concrete constructions of a corresponding multiplicative structure have not yet been carried out (one proposal is described on the last page of \cite{Waldorf}).

 String geometry is based on  the geometry of the basic bundle gerbe $\mathcal{G}_{bas}$ over $G=\Spin(d)$, whose characteristic class is $\frac{1}{2}p_1\in H^4(B\Spin (d),\Z)$ \cite{mclaughlin1,waldorf8}. The geometry consists of a connection on $\mathcal{G}_{bas}$ that is compatible with the multiplicative structure. The curvature of this connection is the Cartan  3-form $H=\frac{1}{24\pi}\left \langle \theta\wedge [\theta \wedge \theta]  \right \rangle$, where $\left \langle -,-  \right \rangle$ denotes the basic inner product on $\mathfrak{spin}(d)$, as in \cref{sec:basic}, and $\theta$ denotes the left-invariant Maurer-Cartan form on $G$. The  3-form $H$ satisfies the equation
\begin{equation}
\label{eq:curvmult}
\pr_1^{*}H + \pr_2^{*}H=m^{*}H+\mathrm{d}\rho\text{,}
\end{equation}
where $\rho\in \Omega^2(G \times G)$ is the 2-form $\rho\defeq \frac{1}{4\pi}\left \langle  \pr_1^{*}\theta\wedge \pr_2^{*}\bar\theta \right \rangle$; here, $\bar\theta$ is the right-invariant Maurer-Cartan form. 

In general, the curvature of a multiplicative bundle gerbe is a pair $(H,\rho)$ of a 3-form $H\in\Omega^3(G)$ and 2-form $\rho\in\Omega^2(G \times G)$ satisfying \cref{eq:curvmult} and an additional \quot{simplicial} condition over $G \times G \times G$, see \cite[Section 2.3]{waldorf5}. Indeed, such a pair defines a degree four cocycle in the simplicial de Rham cohomology of $G$, which computes  $H^4(BG,\R)$ \cite{bott3}. Similar as in Chern-Weil theory, the class of this cocycle coincides with the image of the characteristic class of the multiplicative bundle gerbe in real cohomology \cite[Prop. 2.1]{waldorf5}.

Transgression (to loop groups) is a homomorphism in cohomology, defined as
\begin{align*}
H^3(G,\Z) &\to H^{2}(LG,\Z)\;,\;
\xi \mapsto \int_{S^1} ev^{*}\xi\text{,}
\end{align*}
where $ev:S^1 \times LG \to G$ is the evaluation map. There is an analogous homomorphism in de Rham cohomology. Transgression can also be defined on a geometrical level, taking bundle gerbes with connection over $G$ to principal $\U(1)$-bundles with connection over $LG$, see \cite{brylinski1,gawedzki1}. A multiplicative structure  on a bundle gerbe $\mathcal{G}$ transgresses to a group structure on the corresponding $\U(1)$-bundle, turning it into a central extension which we will denote by $\mathscr{T}_{\mathcal{G}}$ \cite{waldorf5}. 
The basic gerbe $\mathcal{G}_{bas}$ over a compact, simple, connected, simply-connected Lie group $G$ transgresses to the basic central extension of $LG$, i.e. $\mathscr{T}_{\mathcal{G}_{bas}}\cong \widetilde{LG}$, as we will recall below. This establishes the relation between string geometry and the geometry of the basic central extension of
$L\Spin (d)$.

In general, central extensions  $\mathscr{T}_{\mathcal{G}}$ of a loop group $LG$  in the image of the transgression functor come equipped with the following additional structure \cite[Section 5.2]{Waldorfc}: 
\begin{enumerate}[(a)]

\item 
\label{item:trans:fusion}
a multiplicative fusion product $\tilde\mu$ as in \cref{def:fusion product}.

\item
\label{item:trans:connection}
a  connection $\nu$ that is  preserved by $\tilde\mu$  in the sense of \cref{def:fusconnpres}, and additionally has the property of being   superficial and of symmetrizing $\tilde\mu$.

\item
\label{item:trans:pathsplitting}
a multiplicative, contractible path splitting $\kappa$ of the error 1-form of the connection $\nu$. 
\end{enumerate}
The notions of superficial and symmetrizing connections have been defined in \cite{waldorf10}; these will  not be relevant here. Likewise, the notion of a path splitting defined in \cite{Waldorfc} is only listed  for completeness. Central extensions of $LG$  equipped with the structure \cref{item:trans:fusion,item:trans:connection,item:trans:pathsplitting} are called \emph{fusion extensions with connection}; they form a category $\mathcal{F}us\mathcal{E}xt^{\nabla}(LG)$, whose morphisms are fusion-preserving, connection-preserving isomorphisms of central extensions. Transgression is a functor
\begin{equation}
\label{eq:multtrans}
\mathscr{T}: \mathrm{h}_1\mathcal{M}ult\mathcal{G}rb^{\nabla}(G) \to \mathcal{F}us\mathcal{E}xt^{\nabla}(LG) \end{equation}  
defined on the 1-truncation of the bicategory $\mathcal{M}ult\mathcal{G}rb^{\nabla}(G)$ of multiplicative bundle gerbes with connection. All details of these structures can be found in \cite{Waldorfc}. For completeness, and in order to justify the list \cref{item:trans:fusion,item:trans:connection,item:trans:pathsplitting} of additional structures,  we remark that  the transgression functor \cref{eq:multtrans} is an equivalence of categories, whenever $G$ is compact and connected \cite[Theorem 5.3.1]{Waldorfc}.

\begin{proposition}
\label{rem:fusionfactorizationandtransgression}
Let $\U(1)\to\mathcal{L}\to LG$ be a fusion extension with connection, a with fusion product $\tilde\mu$. Then, $\mathcal{L}$ is admissible in the sense of \cref{def:admissible}. Moreover, there is a unique flat fusion factorization $\chi:PG \to \mathcal{L}$ such that  $\tilde\mu=\tilde\mu^{\chi}$.    
\end{proposition}

\begin{proof}
Admissibility is weaker than being disjoint-commutative, which is a property of any fusion extension with connection, see \cite[Theorem 3.3.1]{Waldorfc}. The uniqueness of the fusion factorization can be seen easily from definition \cref{def:fusionfromfusionfactorization} of the associated fusion product. We infer from \cite[Lemma 2.1.2]{waldorf10} the existence of a flat section $\chi$ of $\Delta^{*}\mathcal{L}$, and from  \cite[Prop. 3.1.1]{Waldorfc} that this section is a group homomorphism and neutral with respect to fusion. Using the multiplicativity of $\tilde\mu$ we check that
\begin{equation*}
\tilde\mu(q_{12},q_{23}) =\tilde\mu(q_{12},\chi(\beta_2))\tilde\mu(\chi(\beta_2)^{-1},\chi(\beta_2)^{-1})\tilde\mu(\chi(\beta_2),q_{23})=q_{12}\chi(\beta_2)^{-1}q_{23}=\tilde\mu^{\chi}(q_{12},q_{23})\text{.}
\qedhere
\end{equation*}  
\end{proof}

We remark that the connection $\nu$ of \cref{item:trans:connection} of a fusion extension induces a splitting $\sigma_{\nu}$ on the level of Lie algebras; namely, the one whose image is the horizontal subspace at the unit element. The splitting gives rise to a 2-cocycle $\omega_{\sigma_{\nu}}: L\mathfrak{g} \times L\mathfrak{g} \to \R$ defined from $\sigma_{\nu}$ just as in \cref{eq:2cocycle}.
The section $\sigma_{\nu}$, in turn, induces another connection $\nu'=\nu_{\sigma_{\nu}}$, analogously as described in \cref{re:connectionfromsplitting}. The new connection $\nu'$ does in general  \emph{not} coincide with the original connection $\nu$, and it will be important to distinguish both. 
For example, the connection $\nu'$ is in general not  superficial as required in  \cref{item:trans:connection}.
In a quite general context, it is possible to determine the 2-cocycle $\omega_{\sigma_{\nu}}$ as well as the difference between the two connections, see \cite[Lemmas 2.2.2 and 2.2.3]{Waldorfb}.

\begin{lemma}
\label{th:trans}
Let $\mathcal{G}$ be a multiplicative bundle gerbe over a Lie group $G$, whose curvature $(H,\rho)$ is of the form $H=\frac{1}{24\pi}\left \langle \theta\wedge[\theta\wedge\theta]  \right \rangle$ and $\rho=\frac{1}{4\pi}\left \langle  \pr_1^{*}\theta\wedge \pr_2^{*}\bar\theta \right \rangle$, for some invariant bilinear form  $\left \langle  -,- \right \rangle$ on the Lie algebra $\mathfrak{g}$. Let $\mathscr{T}_{\mathcal{G}}$ be the transgressed central extension, and let $\nu$ be the connection on $\mathscr{T}_{\mathcal{G}}$ that appears under \cref{item:trans:connection}. Then, the following holds: \begin{enumerate}[(a)]
\item
The 2-cocycle determined by the section $\sigma_{\nu}$ is 
        \begin{equation*}
                \omega_{\sigma_{\nu}}(X,Y) = \frac{1}{2\pi i} \int_{0}^{2\pi} \langle X(t), Y'(t) \rangle \mathrm{d}t
        \end{equation*}
        for $X,Y\in L\mathfrak{g}$.
\item
The connection $\nu'$ determined by the section $\sigma_{\nu}$ differs from the connection $\nu$ by a canonical 1-form $\beta \in \Omega^1(LG)$; more precisely, we have $\nu'=\nu+q^{*}\beta$ with
\begin{equation*}
\beta_{\tau}(X)=\frac{1}{4\pi i}\int_0^1\left \langle  \tau(t)^{-1}\partial_t\tau(t),\tau(t)^{-1}X(t) \right \rangle\mathrm{d}t
\end{equation*}
for $\tau\in LG$ and $X\in T_{\tau}LG$.

\end{enumerate}  
\end{lemma}

In the next subsection, we will apply \cref{th:trans} to the case where $G=\Spin(d)$ and $\left \langle  -,- \right \rangle$ is the basic inner product. Then we have $\mathcal{G}=\mathcal{G}_{bas}$, and \cref{th:trans} (a) implies (see \cref{basicce}) that $\mathscr{T}_{\mathcal{G}_{bas}}$ is the basic central extension.

\subsection{Transgression and implementers}

\label{sec:comparison}

One of the goals of the present article is to provide operator-algebraic constructions of the loop group perspective to string geometry. 
In \cref{sec:basic} we have constructed an operator-algebraic model for the basic central extension $\widetilde{L\Spin}(d)$ of the loop group $L\Spin (d)$, together with a local section $\sigma$, inducing a connection $\nu_{\sigma}$. In \cref{sec:fusionofimplementers} we have defined a connection-preserving, multiplicative fusion product $\tilde\mu^{\rho}$ on $\widetilde{L\Spin}(d)$. In the following we compare that structure with the central extension $\mathscr{T}_{\mathcal{G}_{bas}}$ obtained by transgression from the basic gerbe $\mathcal{G}_{bas}$ over $\Spin(d)$, as described in \cref{sec:multgrbtrans}.
We recall that $\mathscr{T}_{\mathcal{G}_{bas}}$ comes equipped with a fusion product $\tilde\mu$ and a connection $\nu$, see  \cref{item:trans:fusion,item:trans:connection} above.

Because both central extensions are the basic one (\cref{basicce,th:trans}), there exists an isomorphism $\widetilde{L\Spin}(d)\cong \mathscr{T}_{\mathcal{G}_{bas}}$ of central extensions of $L\Spin(d)$. Each central extension comes equipped with a section of the associated Lie algebra extension: the section $\sigma_{*}$ of $\widetilde{L\Spin}(d)$ is induced by the local section $\sigma$ of \cref{sec:SmoothStructureOfImp(V)}, and the section $\sigma_{\nu}$ of $\mathscr{T}_{\mathcal{G}_{bas}}$ is induced by the connection $\nu$. 

\begin{lemma}
There exists a unique isomorphism $\varphi: \widetilde{L\Spin}(d)\to \mathscr{T}_{\mathcal{G}_{bas}}$ of central extensions that exchanges the two Lie algebra sections, i.e.  $\sigma_{\nu} = \varphi_{*} \circ \sigma_{*}$.  Moreover, $\varphi$ is connection-preserving for the induced connections $\nu'$ on $\mathscr{T}_{\mathcal{G}_{bas}}$ and $\nu_{\sigma}$ on $\widetilde{L\Spin}(d)$, and it takes the fusion product $\tilde\mu$ on $\mathscr{T}_{\mathcal{G}_{bas}}$ to the fusion product $\tilde\mu^{\rho}$ on $\widetilde{L\Spin}(d)$.
\end{lemma}

\begin{proof}
Uniqueness is clear. For existence, we choose any isomorphism $\varphi$, and observe that $\sigma_{\nu} = (\varphi_{*}\circ \sigma_{*}) + f$, for a bounded linear map  $f: L\mathfrak{ spin}(d) \to \R$. We infer that the 2-cocycles associated to both sections, $\sigma_{*}$ and $\sigma_{\nu}$, coincide: they both give the basic 2-cocycle, see \cref{basicce} and \cref{th:trans}. Thus, using the formula \cref{eq:2cocycle} for the 2-cocycle, we see that $f$ vanishes on all commutators, in other words, it is a Lie algebra homomorphism. We would like to integrate it to a Lie group homomorphism $F:L\Spin(n) \to \U(1)$. To this end, we note that $L\Spin(n)$ is 1-connected and $\U(1)$ is regular, and that both are Lie groups modelled on a locally convex topological vector space. The integration is hence possible due to a theorem of Milnor \cite[Theorem 8.1]{Milnor1983}, also see \cite[Theorem III.1.5]{Neeb2006}. Now, the isomorphism $\varphi' \defeq \varphi \cdot F$ will have the claimed property.

Indeed, since $\varphi$ exchanges the sections $\sigma_{\nu}$ and $\sigma_{*}$, it follows immediately that it is connection-preserving for the induced connections $\nu'=\nu_{\sigma_{\nu}}$ and $\nu_{\sigma}$, respectively. The fusion products on both sides can be characterized by  fusion factorizations that are flat with respect to the connections $\nu'$ and $\nu_{\sigma}$ (see \cref{lem:fusfacflat} and \cref{rem:fusionfactorizationandtransgression}).
Using the fact that $\varphi$ is connection-preserving, $\varphi\circ \rho$ is another flat fusion factorization of $\mathscr{T}_{\mathcal{G}_{bas}}$. Two flat sections differ by a locally constant smooth map $PG \to \U(1)$, and since $PG$ is connected and both sections map the constant path  $1$ to $1\in \mathscr{T}_{\mathcal{G}_{bas}}$, this map is constant and equal to $1\in\U(1)$. Thus, $\varphi$ preserves the fusion factorizations, and hence the corresponding fusion products.
\end{proof}

We may now shift the connection $\nu_{\sigma}$ on $\widetilde{L\Spin}(d)$ by the 1-form $\beta$ of \cref{th:trans} (b), and obtain  a new connection $\tilde \nu \defeq  \nu_{\sigma}-q^{*}\beta$. The isomorphism $\varphi$ is then connection-preserving for the connections $\tilde\nu$ and $\nu$ on $\mathscr{T}_{\mathcal{G}_{bas}}$. In particular, this implies that $\tilde\nu$  
is superficial and symmetrizing, and that  we may use the path splitting $\kappa$ in \cref{item:trans:pathsplitting} of $\mathscr{T}_{\mathcal{G}_{bas}}$ for the connection $\tilde\nu$. Now, we have  equipped our operator-algebraic construction of $\widetilde{L\Spin}(d)$ with all of the structure \cref{item:trans:fusion,item:trans:connection,item:trans:pathsplitting}. Summarizing, we have the following result.

\begin{theorem}
\label{th:comparison}
Our operator-algebraic model $\widetilde{L\Spin}(d)$ of the basic central extension of $L\Spin(d)$ equipped with the fusion product $\tilde\mu^{\rho}$, the connection $\tilde\nu=\nu_{\sigma}+q^{*}\beta$, and the path splitting $\kappa$ is a fusion extension with connection, and it is canonically isomorphic to the fusion extension $\mathscr{T}_{\mathcal{G}_{bas}}$ obtained by transgression of the basic gerbe over $\Spin(d)$, as objects of the   category $\mathcal{F}us\mathcal{E}xt^{\nabla}(L\Spin(d))$.
\end{theorem}

By \cite[Theorem 5.3.1]{Waldorfc} every fusion extension of  with connection $L\Spin(d)$ corresponds to a diffeological multiplicative bundle gerbe with connection over $\Spin(d)$, via a procedure called \emph{regression}. The underlying regressed bundle gerbe is described in \cite[Section 5.1]{waldorf10}. It has the subduction (the diffeological analog of a surjective submersion) $ev_1: P_1\Spin(d)\to \Spin(d)$, where $ev_t: P\Spin(d)\to\Spin(d)$ is the evaluation at $t$, and $P_1\Spin(d)\defeq  ev_0^{-1}(1)$ is the subspace consisting of paths starting at the identity. On the 2-fold fibre product we have a smooth map $P_1\Spin(d)^{[2]}\subset P\Spin(d)^{[2]} \to L\Spin(d)$, along which we pull back the central extension $\widetilde{L\Spin}(d)$, considered as a principal $\U(1)$-bundle. Under the pullback, the fusion product $\tilde\mu^{\rho}$ becomes precisely a bundle gerbe product. The connection $\tilde\nu$ gives one part of the connection on the regressed bundle gerbe. The construction of a corresponding curving is more involved; it uses that $\tilde\nu$ is superficial, see \cite[Section 5.2]{waldorf10}.

The regressed multiplicative structure is strict; it is composed of the fact that $P_1\Spin(d)$ and $\widetilde{L\Spin}(d)$ are diffeological groups, and that the fusion product $\tilde\mu^{\rho}$ is multiplicative. This was mentioned in \cite[Section 5]{Waldorf} and is explained in more detail in \cite[Section 5.3]{Waldorfc}. By \cref{th:comparison} and the fact that regression is inverse to transgression (\cite[Theorem 5.3.1]{Waldorfc}), the above construction results in a diffeological, operator-algebraical construction of the basic gerbe over $\Spin(d)$, with the correct connection and multiplicative structure.

Two further objects, both important to string geometry, can be obtained from any model for the basic gerbe over $\Spin(d)$, and so in particular from our operator-algebraic one:  
\begin{enumerate}[(a)]

\item 
The Chern-Simons 2-gerbe $\mathbb{CS}$ with connection, following the construction in \cite[Sections 2.1 and 3.1]{waldorf8}. Geometric string structures can be viewed as trivializations of the 2-gerbe $\mathbb{CS}$. 

\item
The string 2-group $\mathrm{String}(n)$, following the construction in \cite[Section 3.2]{Waldorf}.
String structures can be viewed as principal 2-bundles for this 2-group, see \cite[Section 7]{Nikolausa}. In short, the underlying diffeological groupoid $\mathrm{String}(n)$ has objects $\mathrm{String}(n)_0 \defeq  P\Spin(d)$ and morphisms 
\begin{equation*}
\mathrm{String}(n)_1 \defeq  P\Spin^{[2]} \ttimes{\cup}{q}\widetilde{L\Spin}(d) =P\Spin^{[2]} \ttimes{M \circ \cup}{q} \Imp(V)\text{,}
\end{equation*}
 source and target maps are $s(\beta_1,\beta_2,U)\defeq \beta_1$ and $s(\beta_1,\beta_2,U)\defeq \beta_2$. and composition is given by the fusion product:
\begin{equation*}
(\beta_2,\beta_3,U_{23}) \circ (\beta_1,\beta_2,U_{12}) \defeq  (\beta_1,\beta_3, \tilde\mu^{\rho}_{\beta_1,\beta_2,\beta_3}(U_{12},U_{23}))\text{.}
\end{equation*}
Associativity of the fusion product implies the associativity of that composition. The identity element of a path $\beta\in P\Spin(d)$ is $id_{\beta}\defeq \rho(\beta)$, where $\rho$ is the fusion factorization. The definition of the fusion product $\tilde\mu^{\rho}$ shows immediately that this is neutral with respect to composition. The multiplication functor $\mathrm{String}(n) \times \mathrm{String}(n) \to \mathrm{String}(n)$ and the inversion functor $i: \mathrm{String}(n)\to\mathrm{String}(n)$ are both given by the group structures on objects and morphisms. The fact that the fusion product is multiplicative implies that the composition is compatible with these group structures.

\end{enumerate}

\appendix

\section{Central extensions of Banach-Lie groups}\label{app:BanachCentral}

In this section we provide the following well-known result used in the proof of \cref{theorem:Banachextension}. See \cite[Proposition 4.2]{Ne02} for a similar statement.

\begin{proposition}
\label{lem:CocycleExtension}
Let $G$ be a connected Banach-Lie group with, let $Z$ be an abelian Banach-Lie group, and let  
\begin{equation}
\label{appA:ex}
1\to Z \to \widehat{G} \overset{q}{\to} G \to 1
\end{equation}
be a  central extension of groups.  Let  $U \subset G$ be an open 1-neighborhood supporting a section $\sigma$, i.e. a map $\sigma: U \to \hat G$ such that $q \circ \sigma=id_{G}$. Suppose there exists an open $1$-neighborhood
$V \subset U$ with $V^2 \subset U$, such that the associated 2-cocycle $f_{\sigma}: V \times V \to Z$ defined by 
\begin{equation*}
\sigma(g_{1})\sigma(g_{2}) =f_{\sigma}(g_{1},g_{2}) \sigma(g_{1}g_{2})
\end{equation*}
is smooth in an open $(1,1)$-neighborhood. Then, $\widehat{G}$ carries a unique Banach-Lie group structure such that $\sigma$ is smooth in an open $1$-neighborhood. Moreover, when equipped with this Banach-Lie group structure, \cref{appA:ex} is a central extension of Banach-Lie groups.

\end{proposition}

We will use the following lemma, which appears as \cite[Lemma 4.1]{Ne02} or as \cite[p.14]{Ti83} in the finite-dimensional case, which goes through without changes. 
\begin{lemma}\label{lem:LocalSuf}
        Let $G$ be a group and $K\subset G$ be a subset with $1\in K$ and $K = K^{-1}$. We  assume that $K$ is a Banach manifold such that the inversion is smooth on $K$ and there exists an open $1$-neighbourhood $V \subseteq K$ with $V^{2} \subseteq K$, such that the  multiplication $m: V \times V \rightarrow K$ is  smooth. Further, we assume that for any $g \in G$ the conjugation map $C_{g}: G \to G,x \mapsto gxg^{-1}$ is  smooth in an open $1$-neighborhood. Then, there exists a unique Banach-Lie group structure on $G$ such that the inclusion map $K \hookrightarrow G$ is a local diffeomorphism at $1$.
\end{lemma}

Now we give the proof of  \cref{lem:CocycleExtension}.
Without loss of generality we assume that $U$ satisfies $U^{-1}=U$. We set $K\defeq  U \times Z$. We equip $K$ with the product  Banach manifold structure, and identify it with a subset of $\hat G$ along the injective map $(u,z)\mapsto \sigma(u)z$.
We consider the open subset
        \begin{equation*}
                W \defeq   \{((x_1,z_1),(x_2,z_2)) \mid x_1x_2\in U\} \subset K \times K,
        \end{equation*}
        and choose an open $1$-neighborhood $V$ in $K$ such that $V\times V \subset W$. 
The definition of $f_\sigma$ implies that the restriction of the group structure of $\hat G$ to $V \times V$ is given by
                \begin{align*}
                ((x_1,z_1),(x_2,z_2)) &\mapsto (x_{1}x_{2},z_1z_2f(x_{1},x_{2})),
        \end{align*}
        which is smooth. Likewise, the inversion map on $K$ is
        \begin{align*}
                (x,z) &\mapsto (x^{-1}, z^{-1} f(x,x^{-1})^{-1}),
        \end{align*}
and hence smooth, too.

Next, we claim that for any $\widehat{g} \in \widehat{G}$ the map $C_{\widehat{g}} : \hat G \to \hat G: x \mapsto \widehat{g} x \widehat{g}^{-1}$ is smooth in an open $1$-neighborhood. To this end, let $X \subseteq V$ be an again smaller  open $1$-neighborhood such that $X^{3} \subseteq V$ and $X^{-1} = X$. It then follows that for $\widehat{g} \in X$ we have that  $C_{\widehat{g}}: X \rightarrow K$ is smooth. 
Using the assumption that $G$ is connected, it  follows that $X$ generates $\widehat{G}$, and hence that any $\widehat{g} \in \widehat{G}$ can be decomposed as $\widehat{g} = \widehat{g}_{1} ... \widehat{g}_{n}$ with $n \in \mathbb{N}$ and $\widehat{g}_{1},...,\widehat{g}_{n} \in X$. It follows that $C_{\widehat{g}} = C_{\widehat{g}_{1}} ... C_{\widehat{g}_{n}}$ is smooth.     We see that now all the conditions of Lemma \ref{lem:LocalSuf} are met, and it follows that there is a unique Banach-Lie group structure on $\widehat{G}$ such that the inclusion $K \hookrightarrow \widehat{G}$ is a local diffeomorphism at $\1$. 

To complete the proof we now need to prove that $\widehat{G} \rightarrow G$ is a smooth principal $Z$-bundle, which boils down to prove that  $q: \widehat{G} \rightarrow G$ is a smooth surjective submersion. This is true in the open $\1$-neighbourhood $K$, and hence everywhere since it is a group homomorphism.

\section{Modular conjugation in the free fermions}\label{sec:ModularConjugation}

In this section we give a proof of \cref{thm:modular}, i.e.~we compute the  modular conjugation $J_{\Omega}$ for the triple $(\Cl(V_{-})'',\mathcal{F},\Omega)$, see \cref{sec:tomita} for the notation. The result and possible computations are probably well-known, and have appeared in slight variations of the setting in  \cite[Section 15]{wassermann98}, \cite{Hen14}, and \cite{janssens13}. In the latter reference, Janssens outlines how to transfer his computations into our setting, and in the following we have done this step by step, closely following \cite{janssens13}.  

The strategy will be to find a unitary operator $s: L \rightarrow L$ such that the Tomita operator $S$ is  $S = \operatorname{k}^{-1} \Lambda_{i \alpha s}$, and then to find a polar decomposition for $s$.
For simplicity, we set $d=1$, and thus only work with
 $V = L^{2}(\mb{S})$. All statements carry over to $L^{2}( \mathbb{S} )\otimes \C^{d}$ in a straightforward fashion. Further, it will be convenient to identify $V$ with the $L^{2}$-closure of $C^{\infty}(S^{1}, \C)$, see \cref{sec:SectionsOfSpinorBundle}.

Let us write $P_{L}: V \rightarrow L$, $P_{L}^{\perp}:V \rightarrow \alpha(L) = L^{\perp}$ and $P_{\pm}:V \rightarrow V_{\pm}$ for the orthogonal projections. We then define the operators
\begin{equation*}
T_{\pm} \defeq  (P_{L} - P_{\pm})^{2},
\end{equation*}
on $V$.
\begin{lemma}\label{lem:UnboundedInverses}
        The operators $T_{\pm}$ have unbounded inverses $T_{\pm}^{-1}$.
\end{lemma}

We will later  diagonalize the operators $T_{\pm}$; the fact that these operators have unbounded inverses will be evident from their diagonal form. For now, we assume that this lemma holds.

We recall that the Tomita operator is $S: a \lact \Omega \mapsto a^{*} \lact \Omega$. Suppose that $s:V \rightarrow V$ is an operator with the properties that $s(L) = \alpha(L)$, and that $v+ s(v) \in V_{-}$ for $v \in L$.
Now, let $v \in L \subset \mc{F}$. We compute $v = v \lact \Omega = (v + s(v)) \lact \Omega$, and hence $S(v) = S((v+s(v)) \lact \Omega) = (\alpha s (v) + \alpha(v) ) \lact \Omega = \alpha s (v)$. This means that $S|_{L} = \operatorname{k}^{-1} \Lambda_{i \alpha s}|_{L}$, in \cref{lem:SasWedge} we shall see that $S = \operatorname{k}^{-1} \Lambda_{i \alpha s}$.
It turns out that the densely defined operator $s$ on $V$ defined by
\begin{equation*}
        s \defeq  P_{L} P_{-} T_{-}^{-1} P_{L}^{\perp} + P_{L}^{\perp} P_{-} T_{+}^{-1} P_{L},
\end{equation*}
does the trick, as we show below.
\begin{lemma}
        The operator $s$ commutes with $\alpha$.
\end{lemma}
\begin{proof}
        Direct computation using the fact that $\alpha P_{L} = P_{L}^{\perp} \alpha$, and $\alpha P_{\pm} = P_{\pm} \alpha$, which implies that $\alpha T_{\pm} = T_{\mp} \alpha$.
\end{proof}
\begin{lemma}
        The operator $s$ maps $v \in L$ to the unique $w \in \alpha(L)$ such that $v+w \in V_{-}$, if such a $w$ exists. Similarly, it maps $w \in \alpha(L)$ to the unique $v \in L$ such that $v + w \in V_{-}$, again if such a $v$ exists.
\end{lemma}
\begin{proof}
        Let $v \in L$ be arbitrary. First let us prove the uniqueness claim. Suppose that there exist $w,w'\in \alpha(L)$ such that $v+w \in V_{-}$ and $v+ w' \in V_{-}$. Then it follows that $(v+w)-(v+w') = w-w' \in V_{-} \cap \alpha(L) = \{0 \}$ by \cref{lem:generalposition}.
        By direct computation one might verify that
        \begin{equation*}
        v + P_{L}^{\perp}P_{-}T_{+}^{-1} v = P_{-}T_{+}^{-1}v.
        \end{equation*}
        And hence
        \begin{equation*}
        v + sv = v + P_{L}^{\perp}P_{-}T_{+}^{-1}v = P_{-}T_{+}^{-1}v \in V_{-}.
        \end{equation*}
        This proves the first statement in the lemma. The second statement follows from a similar computation.
\end{proof}

As a consequence, we see that the operator $s$ squares to $\1$ and restricts to the identity on $V_{-}$.
The following result tells us precisely how $\alpha s$ is related to $S$.
\begin{lemma}\label{lem:SasWedge}
        For all $a \in \Cl(V_{-})$ we have $\operatorname{k}^{-1}\Lambda_{i \alpha s}(a \lact \Omega) = a^{*} \lact \Omega = S(a \lact \Omega)$.
\end{lemma}
\begin{proof}
        We will prove this by induction on the  degree in $\Cl(V_{-})$ (note that while the algebra $\Cl(V_{-})$ is not graded, it is filtered). Suppose that the claim holds for all $a \in \Cl(V_{-})$ for $a$ of degree $n$ or less. We shall prove that it follows that for all $f_{0},\dots ,f_{n} \in V_{-}$ we have that
        \begin{equation*}
        \operatorname{k}^{-1}\Lambda_{i \alpha s} (f_{0} \dots f_{n} \lact \Omega) = S(f_{0} \dots f_{n} \lact \Omega).
        \end{equation*}
        First off, we set $x \defeq  f_{1} \dots f_{n} \lact \Omega$. Furthermore, there exist $y_{i} \in \Lambda^{i} L$, where $i =0,...,n$ such that $x = \sum_{i=0}^{n} y_{i}$. Finally, we set $v = P_{L} f_{0}$ and $w= P_{L}^{\perp}f_{0}$. Note that $sv =w$ and $sw = v$.
        Now we compute
        \begin{align*}
        \operatorname{k}^{-1}\Lambda_{i \alpha s} (f_{0} \dots f_{n} \lact \Omega) &= \operatorname{k}^{-1} \Lambda_{i \alpha s} (f_{0} \lact x) \\
        &= \operatorname{k}^{-1} \Lambda_{i \alpha s} ( v \wedge x + \iota_{\alpha{w}} x) \\
        &= \operatorname{k}^{-1} (i \alpha (w) \wedge \Lambda_{i \alpha s}x + \Lambda_{i \alpha s} \iota_{\alpha(w)} x).
        \end{align*}
        Straightforward computations, using the induction hypothesis, then show that
        \begin{equation*}
                \operatorname{k}^{-1}(i\alpha(w) \wedge \Lambda_{i \alpha s} x) = \alpha(f_{n}) \dots \alpha(f_{0}) \lact \Omega + \sum_{k=1}^{n}(-1)^{k} \langle f_{k}, \alpha (w) \rangle \operatorname{k}^{-1} \Lambda_{i \alpha s} f_{1} \dots \widehat{f_{k}} \dots f_{n} \lact \Omega,
        \end{equation*} 
        and that 
        \begin{equation*}
                \operatorname{k}^{-1} \Lambda_{i \alpha s} \iota_{\alpha(w)} x = - \sum_{k=1}^{n} (-1)^{k} \langle f_{k}, \alpha(w) \rangle \operatorname{k}^{-1} \Lambda_{i \alpha s} f_{1} \dots \widehat{f_{k}} \dots f_{n} \lact \Omega.
        \end{equation*}
        Putting these results together we see that
        \begin{align*}
        \operatorname{k}^{-1}\Lambda_{i \alpha s} (f_{0} \dots f_{n} \lact \Omega) &=  \operatorname{k}^{-1} (i \alpha (w) \wedge \Lambda_{i \alpha s}x + \Lambda_{i \alpha s} \iota_{\alpha(w)} x) = \alpha(f_{n}) \dots \alpha(f_{0}) \lact \Omega = S(f_{0} \dots f_{n} \lact \Omega),
        \end{align*}
        which completes the induction step, and hence our proof.
\end{proof}
Let $\alpha s = u \delta^{1/2}$ be the polar decomposition of $\alpha s$. Note that the fact that $\alpha s$ preserves $L$ implies that both $u$ and $\delta^{1/2}$ preserve $L$. Similarly, the fact that $\alpha s$ commutes with $\alpha$ implies that both $u$ and $\delta^{1/2}$ commute with $\alpha$.
From the equality $S = \operatorname{k}^{-1} \Lambda_{i \alpha s}$ we obtain the following.
\begin{proposition}\label{cor:PolarDceompositionOfS}
        The polar decomposition of $S$ is given by $S = \operatorname{k}^{-1}\Lambda_{iu} \Lambda_{\delta^{1/2}}$, whence $J = \operatorname{k}^{-1}\Lambda_{iu}$ and $\Delta^{1/2} = \Lambda_{\delta^{1/2}}$.
\end{proposition}
We claim that $u|_{L} = -i \alpha \tau$, which implies that $J = \operatorname{k}^{-1} \Lambda_{\alpha \tau}$. The claim is proved in a sequence of lemmas in the remainder of this section.

\begin{lemma}\label{lem:PolarDecoOfAlphas}
        The equations
        \begin{equation*}
        u = \alpha \left( \frac{P_{L}^{\perp} P_{-} P_{L}}{\sqrt{T_{+}T_{-}}} + \frac{P_{L} P_{-} P_{L}^{\perp}}{\sqrt{T_{+}T_{-}}} \right),
        \end{equation*}
        and
        \begin{equation*}
        \delta^{1/2} = \sqrt{\frac{T_{-}}{T_{+}}} P_{L} + \sqrt{\frac{T_{+}}{T_{-}}} P_{L}^{\perp}
        \end{equation*}
        hold.
\end{lemma}
\begin{proof}
        The fact that $u \delta^{1/2} = \alpha s$ follows from a straightforward computation. Furthermore, the fact that $u$ is anti-unitary can be verified directly as well. The fact that $\delta^{1/2}$ is positive will be evident from an expression that we will give later.
\end{proof}

We now turn to the task of simultaneously diagonalizing $T_{\pm}$, and proving that they have densely defined inverses. First, we identify the circle with the one-point compactification of the real line by means of the diffeomorphisms
\begin{align*}
&\Gamma: S^{1} \rightarrow \R \cup \infty,
z  \mapsto - i\frac{z+1}{z-1},  &\Gamma^{-1}: \R \cup \infty \rightarrow S^{1},
x  \mapsto \frac{x - i}{x+i}.
\end{align*}
        We note that $\Gamma(I_{+}) = \R_{+}$ and $\Gamma(I_{-}) = \R_{-}$.
We define unitary transformations
\begin{align*}
&U_{\Gamma}: L^{2}(S^{1},\C) \rightarrow L^{2}(\R,\C), &        &U_{\Gamma}^{-1}: L^{2}(\R,\C) \rightarrow L^{2}(S^{1},\C),
\end{align*}
 by
\begin{align*}
U_{\Gamma}(f)(x) &= \frac{1}{\sqrt{\pi}}\frac{1}{(i + x)} f( \Gamma^{-1}x), & x &\in \R, \\
U_{\Gamma}^{-1}(g)(z) &= \sqrt{\pi} (i + \Gamma(z)) g(\Gamma(z)), & z &\in S^{1}.
\end{align*}
We recall that on $L^{2}(S^{1},\C)$ the maps $\alpha$ and $\tau$ act as follows
\begin{equation*}
\alpha(f)(z) = \frac{1}{z} \overline{f(z)}, \quad \tau(f)(z) = \frac{1}{z}f(\overline{z}).
\end{equation*}
We compute how $\alpha$ and $\tau$ transform under $U_{\Gamma}$:
%
\begin{equation*}
        U_{\Gamma} \alpha U_{\Gamma}^{-1} g (x) = \overline{g(x)}, \quad \quad U_{\Gamma} \tau U_{\Gamma}^{-1} g(x) = g(-x).
\end{equation*}
%

Let us write $\mb{H}_{\pm} = \{z \in \C \mid \pm \Im(z) > 0 \}$ for the upper and lower half plane.

\begin{lemma}
We have
        \begin{equation*}
        U_{\Gamma}(L\cap C^{\infty}(S^{1},\C)) \subseteq \{f \in L^{2}(\R,\C) \mid f \text{ extends to a holomorphic function } f: \overline{\mb{H}_{-}} \rightarrow \C \}\text{.}
        \end{equation*}
\end{lemma}
\begin{proof}
        We see that $L \cap C^{\infty}(S^{1},\C)$ is the span of the functions $z^{-n-1}$ for $n \geqslant 0$. We then compute
        \begin{equation*}
        (U_{\Gamma}z^{-n-1})(x) = \frac{1}{\sqrt{\pi}} \frac{(x+i)^{n}}{(x-i)^{n+1}}.
        \end{equation*}
        These functions are smooth and square integrable for all $n \in \mathbb{Z}$. Furthermore if $n \geqslant 0$, then $U_{\Gamma}z^{-n-1}$ extends to a holomorphic function on the lower half plane.
\end{proof}

Given a smooth function $g = U_{\Gamma} f \in C^{\infty}(\R, \C)$ with $f \in C^{\infty}(S^{1},\C)$, we wish to find two holomorphic functions, say $\hat{g}_{+}:\mb{H}_{+}\rightarrow \C$ and $\hat{g}_{-}:\mb{H}_{-} \rightarrow \C$ which extend to square-integrable functions on the real line (denoted by the same name), such that $g = \hat{g}_{+}|_{\R} + \hat{g}_{-}|_{\R}$. This is essentially a version of the Riemann-Hilbert problem, we follow the standard solution to such problems.

For $g \in L^{2}(\R,\C)$ we define the Cauchy transform
\begin{equation*}
        \hat{g}(z) \defeq  \frac{1}{2\pi i} \int_{-\infty}^{\infty} \frac{g(x)}{x- z} \text{d}x, \quad z \in \C\setminus \R.
\end{equation*}
The function $\hat{g}: \C \setminus \R \rightarrow \C$ is holomorphic. The following lemma is then a well-known consequence of the Sokhotski-Plemelj theorem, \cite[Section 4.2]{Gakhov66}, \cite[Section 17]{Musk58}, \cite[Chapter 14]{plemelj64}.
\begin{lemma}\label{lem:RiemannHilbertProblem}
        Let $f \in C^{\infty}(S^{1}, \C)$, and set $g = U_{\Gamma}f$. Then $g_{-} = U_{\Gamma} P_{L} f$ and $g_{+} = U_{\Gamma}P_{L}^{\perp}f$ are smooth functions on $\R$ that extend uniquely to holomorphic functions on the lower and upper half plane, $G_{-}$ and $G_{+}$, respectively. Furthermore, we have
$G_{-} = - \hat{g}|_{\mb{H}_{-}}$ and $G_{+} = \hat{g}|_{\mb{H}_{+}}$.
\end{lemma}
For any operator $X$ on $V$, let us write $X' \defeq  U_{\Gamma} X U_{\Gamma}^{-1}$. As a consequence of \cref{lem:RiemannHilbertProblem} we obtain for $g = U_{\Gamma}f$ with $f \in C^{\infty}(S^{1},\C)$:
        \begin{align*}
        (P_{L}'g)(x) &= -\lim_{\eps \downarrow 0} \hat{g}(x - i \eps), & x \in \R, \\
        ((P_{L}')^{\perp}g)(x) &= \lim_{\eps \downarrow 0} \hat{g}(x + i \eps), & x \in \R.
        \end{align*}
Next, we define the unitary $D: L^{2}(\R) \rightarrow L^{2}(\R) \oplus L^{2}(\R)$ by
\begin{equation*}
(Df)(t) = (f_{r}(t),f_{l}(t)) \defeq  (e^{t/2}f(e^{t}), e^{t/2}f(-e^{t})).
\end{equation*}
The inverse of $D$ is given by
\begin{equation*}
D^{-1}(f_{r},f_{l})(t) = \begin{cases}
\frac{1}{\sqrt{t}} f_{r}(\log (t)) & t > 0, \\
\frac{1}{\sqrt{-t}} f_{l}(\log(-t)) & t < 0.
\end{cases}
\end{equation*}
\begin{lemma}
        We have
        \begin{equation*}
        DP_{+}'D^{-1} = \begin{pmatrix}
        \1 & 0 \\
        0 & 0
        \end{pmatrix},
        \quad \quad
        DP_{-}'D^{-1} = \begin{pmatrix}
        0 & 0 \\
        0 & \1
        \end{pmatrix}.
        \end{equation*}
\end{lemma}
\begin{proof}
        Recall that $\Gamma^{-1}$ carries $\R_{\pm}$ into $I_{\pm}$. It follows that $P_{\pm}'$ is the projection $L^{2}(\R, \C) \rightarrow L^{2}(\R_{\pm}, \C)$, from which the result follows.
\end{proof}

Let us write
\begin{equation*}
        D(P_{L}^{\perp})'D^{-1} = \begin{pmatrix}
                P_{L}^{rr} & P_{L}^{rl} \\
                P_{L}^{lr} & P_{L}^{ll}
        \end{pmatrix}.
\end{equation*}
We define
\begin{equation*}
        c_{\eps}(u) \defeq  \frac{e^{-u/2}}{e^{-u}+1-i \eps}, \quad \quad
        s_{\eps}(u) \defeq  \frac{e^{-u/2}}{e^{-u} - 1 -i \eps}.
\end{equation*}
\begin{lemma}
        We have, for all $t \in \R$,
        \begin{align*}
        P_{L}^{rr}f_{r}(t) &= \frac{1}{2\pi i} \lim_{\eps \downarrow 0} (f_{r} \star s_{\eps}) (t), &
        P_{L}^{rl}f_{l}(t) &= -\frac{1}{2\pi i} \lim_{\eps \downarrow 0}  (f_{l} \star c_{-\eps})(t), \\
        P_{L}^{lr}f_{r}(t) &= \frac{1}{2\pi i} \lim_{\eps \downarrow 0} (f_{r} \star c_{\eps}) (t), &
        P_{L}^{ll}f_{l}(t) &= -\frac{1}{2\pi i} \lim_{\eps \downarrow 0} (f_{l} \star s_{-\eps}) (t), 
        \end{align*}
        where $\star$ stands for the convolution product.
\end{lemma}
\begin{proof}
        Straightforward, but tedious, computations.
\end{proof}

Next, we take the Fourier transforms of $s_{\eps}$, and $c_{\eps}$, where we use the following convention for the fourier transform
\begin{equation*}
\mc{F} f(k) \defeq  \int_{-\infty}^{\infty} f(u) e^{-iku} \text{d}u.
\end{equation*}
The Fourier transforms of $s_{\eps}$ and $c_{\eps}$ can be computed using the residue theorem, alternatively, they can be found in \cite[Section 3.2 (15)]{IntTransforms}.
For $s_{\eps}$ the result depends on the sign of $\eps$, suppose that $1/2 > \eps^{+} > 0$ and $-1/2 < \eps^{-} < 0$, and $-1/2 < \eps < 1/2$, then we obtain
\begin{align*}
        \mc{F} s_{\eps^{+}}(k) &= 2 \pi i (1+i \eps^{+})^{ik - 1/2} \frac{e^{\pi k}}{e^{-\pi k} + e^{\pi k}}, \\
        \mc{F} s_{\eps^{-}}(k) &= -2 \pi i (1+i \eps^{-})^{ik - 1/2} \frac{e^{-\pi k}}{e^{-\pi k} + e^{\pi k}}, \\
    \mc{F}c_{\eps}(k) &= 2 \pi (1 - i \eps)^{ik - 1/2} \frac{1}{e^{\pi k} + e^{-\pi k}}.
\end{align*}
It follows that
\begin{align*}
        (\mc{F} P_{L}^{rr}f_{r})(k) &= \left(\mc{F} \frac{1}{2\pi i} \lim_{\eps \downarrow 0} (f_{r} \star s_{\eps})\right) (k) 
        = \frac{e^{\pi k}}{e^{\pi k} + e^{-\pi k}}\mc{F}(f_{r})(k) .
\end{align*}
Similarly, we obtain
\begin{equation*}
        (\mc{F} P_{L}^{ll}f_{l})(k) = \frac{e^{-\pi k}}{e^{\pi k} + e^{-\pi k}} \mc{F} (f_{l})(k), \quad k \in \R.
\end{equation*}

Performing the limit $\eps \downarrow 0$ and multiplying with $(2 \pi i)^{-1}$ we obtain
\begin{equation*}
        (\mc{F} P_{L}^{lr}f_{r})(k) = \frac{-i}{e^{\pi k} + e^{- \pi k}} \mc{F}(f_{r})(k), \quad k \in \R.
\end{equation*}
Setting $\hat{X} \defeq  (\mc{F} \oplus \mc{F}) D X' D^{-1} (\mc{F}^{-1} \oplus \mc{F}^{-1})$ we obtain
\begin{equation*}
        \hat{P}_{L}^{\perp} = \frac{1}{e^{\pi k} + e^{- \pi k}} \begin{pmatrix}
        e^{\pi k} & i \\
        -i & e^{- \pi k}
        \end{pmatrix}.
\end{equation*}
In a similar manner one could compute $\hat{P}_{L}^{\perp}$, but it follows from the fact that $\1 = \hat{P}_{L} + \hat{P}_{L}^{\perp}$ that
\begin{equation*}
        \hat{P}_{L} =  \frac{1}{e^{\pi k} + e^{- \pi k}} \begin{pmatrix}
        e^{- \pi k} & -i \\
        i & e^{\pi k}
        \end{pmatrix}\text{.}
\end{equation*}
We set
\begin{equation*}
        a(k) = \frac{1}{e^{\pi k} + e^{-\pi k}}.
\end{equation*}
We then obtain
\begin{equation*}
        \hat{T}_{+} = a(k)e^{\pi k} \1, \quad \quad \hat{T}_{-} = a(k) e^{-\pi k} \1.
\end{equation*}
As promised, it is clear from these expression that $\hat{T}_{\pm}$ are injective operators with unbounded inverses; this proves \cref{lem:UnboundedInverses}.
We furthermore have
\begin{align*}
        \hat{\delta}^{1/2} &= \sqrt{\frac{\hat{T}_{-}}{\hat{T}_{+}}} \hat{P}_{L} + \sqrt{\frac{\hat{T}_{+}}{\hat{T}_{-}}} \hat{P}_{L}^{\perp} 
        = a(k) \begin{pmatrix}
        a(2k)^{-1} & i(e^{\pi k} - e^{-\pi k}) \\
        -i(e^{\pi k} - e^{-\pi k}) & 2\1
        \end{pmatrix}
\end{align*}
which is a positive operator, hence the expression $u \delta^{1/2} = \alpha s$ really is the polar decomposition of $\alpha s$, this was the missing part in the proof of \cref{lem:PolarDecoOfAlphas}. Next, we compute
\begin{equation*}
        \alpha u = \left( \frac{P_{L}^{\perp} P_{-} P_{L}}{\sqrt{T_{+}T_{-}}} + \frac{P_{L} P_{-} P_{L}^{\perp}}{\sqrt{T_{+}T_{-}}} \right) = a(k) \begin{pmatrix}
                -2\1 & i(e^{\pi k} - e^{-\pi k}) \\
                -i ( e^{\pi k} - e^{-\pi k}) & 2\1
        \end{pmatrix}\text{.}
\end{equation*}
On the other hand, we compute
\begin{equation*}
        i \hat{\tau}(\hat{P}_{L} - \hat{P}^{\perp}_{L}) = a(k) \begin{pmatrix}
        -2\1 & i(e^{\pi k} - e^{-\pi k}) \\
        -i ( e^{\pi k} - e^{- \pi k}) & 2\1
        \end{pmatrix}.
\end{equation*}
Which allows us to conclude that $u = -i \alpha \tau(P_{L} - P_{L}^{\perp})$, and hence $u|_{L} = -i \alpha \tau$; from \cref{cor:PolarDceompositionOfS} it follows that $J = \operatorname{k}^{-1}\Lambda(\alpha \tau)$.

\newcommand{\etalchar}[1]{$^{#1}$}

\def\kobiburl#1{
   \IfSubStr
     {#1}
     {://arxiv.org/abs/}
     {\kobibarxiv{#1}}
     {\kobiblink{#1}}}
\def\kobibarxiv#1{\href{#1}{\texttt{[arxiv:\StrGobbleLeft{#1}{21}]}}}
\def\kobiblink#1{
  \StrSubstitute{#1}{\~{}}{\string~}[\myurl]
  \StrSubstitute{#1}{_}{\underline{\;\;}}[\mylink]
  \StrSubstitute{\mylink}{&}{\&}[\mylink]
  \StrSubstitute{\mylink}{/}{/\allowbreak}[\mylink]
  \newline Available as: \mbox{\;}
  \href{\myurl}{\texttt{\mylink}}}

\raggedright
\addcontentsline{toc}{section}{\refname}
\bibliographystyle{kobib}
\bibliography{bibfile}

Peter Kristel ({\it peter.kristel@uni-greifswald.de})

\smallskip

Konrad Waldorf ({\it konrad.waldorf@uni-greifswald.de})

\medskip

Universit\"at Greifswald\\
Institut f\"ur Mathematik und Informatik\\
Walther-Rathenau-Str. 47\\
17487 Greifswald\\
Germany

\end{document}